\documentclass[preprint,aos]{imsart}
\RequirePackage{natbib}
\setattribute{journal}{name}{}

\RequirePackage[OT1]{fontenc}
\RequirePackage{amsthm,amsmath,natbib,graphicx,enumitem,mathrsfs}
\RequirePackage[colorlinks,citecolor=blue,urlcolor=blue]{hyperref}
\RequirePackage{hypernat}
\usepackage{amssymb,tabularx,multicol,multirow,booktabs}
\usepackage[top=1in, bottom=1in, left=1in, right=1in]{geometry}
\usepackage{mhequ}
\usepackage[usenames,dvipsnames]{color}
\usepackage{epstopdf}
\RequirePackage{xr}
\setattribute{journal}{name}{}

\startlocaldefs
\numberwithin{equation}{section}
\theoremstyle{plain}
\endlocaldefs
\input binary_macro.tex

\def \be{\begin{equs}}
	\def \ee{\end{equs}}
\def \tcr{\textcolor{black}}

\def \I{\mathcal{I}}

\def \E{\mathbb{E}}
\def \P{\mathbb{P}}

\def \fhat{\hat{f}}

\def \sumn{\sum_{i=1}^n}

\def \Ptheta {\mathbb{P}_{\boldsymbol{\Theta},a,b}}
\def \Pzero {\mathbb{P}_{\mathbf{1},a,b}}
\def \Etheta {\mathbb{E}_{\mathbf{\Theta},a,b}}
\def \Ezero {\mathbb{E}_{\mathbf{1},a,b}}

\def \Bin {\mathrm{Bin}}
\def \cdense {C_{\mathrm{dense}}}
\def \csparse {C_{\mathrm{sparse}}}
\def \cmax {C_{\mathrm{max}}}
\def\chc {C_{\mathrm{HC}}}
\def \muzero {\mu_{n0}}
\def \sigmazero {\sigma_{n0}}
\def \muone {\mu_{n1}}
\def \mutwo {\mu_{n2}}
\def \sigmaone {\sigma_{n1}}
\def \sigmatwo {\sigma_{n2}}
\def \sigmabeta {\sigma_n(\beta_1,\beta_2)}
\def \mubeta {\mu_n(\beta_1,\beta_2)}

\def \C {\mathcal{C}}
\def \hcpar {\C,\beta_1,\beta_2}
\def \hcparl {\tilde{S}_l,\beta_1,\beta_2}
\def \hcparlone {\tilde{S}_{l_1},\beta_1,\beta_2}
\def \hcparltwo {\tilde{S}_{l_2},\beta_1,\beta_2}

\def \genconk {\overline{\rho}_k(\beta_1,\beta_2)}
\def \muonenull {\mu^0_{n1}}
\def \mutwonull {\mu^0_{n2}}

\def \sigmabetanulll {\sigmazero(\beta_1,\beta_2,\tilde{S}_l)}

\def \sigmabetanulllone {\sigmazero(\beta_1,\beta_2,\tilde{S}_{l_1})}
\def \sigmabetanullltwo {\sigmazero(\beta_1,\beta_2,\tilde{S}_{l_2})}
\def \mubetanulll {\beta_1\muonenull(\tilde{S}_{l})+\beta_2\mutwonull(\tilde{S}_{l})}
\def \mubetanulllone {\beta_1\muonenull(\tilde{S}_{l_1})+\beta_2\mutwonull(\tilde{S}_{l_1})}
\def \mubetanullltwo {\beta_1\muonenull(\tilde{S}_{l_2})+\beta_2\mutwonull(\tilde{S}_{l_2})}
\newtheorem*{theorem*}{Theorem}

\begin{document}
	\begin{frontmatter}
		\title{Testing Degree Corrections in Stochastic Block Models}
		\runtitle{degree corrections for SBM}
		
		\begin{aug}
			
			\author{\fnms{Rajarshi} \snm{Mukherjee}\ead[label=e1]{rmukherj@hsph.harvard.edu}},
			\author{\fnms{Subhabrata} \snm{Sen}\ead[label=e2]{subhabratasen@fas.harvard.edu}}

			\affiliation{Harvard University}
			
			\address{Department of Biostatistics\\
				655 Huntington Avenue \\
				Boston, MA- 02115. \\
				\printead{e1}}	
			
			\address{Department of Statistics\\
				1 Oxford Street\\
				Cambridge, MA-02138.  \\
				\printead{e2}}

		\end{aug}
		
		\begin{abstract} We study sharp detection thresholds for degree corrections in Stochastic Block Models in the context of a goodness of fit problem, and explore the effect of the unknown community assignment (a high dimensional nuisance parameter) and the graph density on testing for degree corrections. When degree corrections are relatively dense, a simple test based on the total number of edges is asymptotically optimal. For sparse degree corrections, the results undergo several changes in behavior depending on density of the underlying Stochastic Block Model. For graphs which are not extremely sparse, optimal tests are based on  Higher Criticism or Maximum Degree type tests based on a linear combination of within and across (estimated) community degrees. In the special case of balanced communities, a simple degree based Higher Criticism Test \citep{mms2016} is optimal in case the graph is not completely dense, while the more complicated linear combination based procedure is required in the completely dense setting. 
			The ``necessity" of the two step procedure is demonstrated for the case of balanced communities by the failure of the ordinary Maximum Degree Test in achieving sharp constants. Finally for extremely sparse graphs the optimal rates change, and a version of the maximum degree test with a different rejection region is shown to be optimal.
		\end{abstract}

		
	\end{frontmatter}
	
	\section{Introduction}

	The analysis of network data has received considerable attention in diverse areas of research such as social sciences, biology, statistics, and computer science. 
	At a high level, the central task in this area is to study underlying structural characteristics, given the network data. A statistically principled approach formalizes any such question as an inference problem, given a suitable, probabilistic generative model for the observed data. Thus statistical research in this direction has focused on a few principal themes. The first theme concerns the design of suitable models which reflect some of the features observed in real networks \citep{barabasi1999emergence,watts1998collective}, while a second theme concentrates on developing statistical methodology for inference on data from these generative situations. \textcolor{black}{ In the scheme of these two themes, statistical literature tends to build on initially simpler models to eventually work with increasingly complex ones. In this statistical endeavor, however,  some inferential tasks can be easily adapted to the increased complexity of the model, whereas other inferential tasks are often significantly more complicated.} It is in these situations, that a third, perhaps equally important, but often less emphasized statistical question becomes important. This concerns the choice of simple vs. complex models for the data under consideration, and is intimately related to classical goodness of fit testing problems. In this paper, we concentrate on a concrete example of this general problem in the context of degree corrections for stochastic block models.
	
	To this end, we first describe the Stochastic Block Model (henceforth referred to as SBM). It has been empirically observed that real networks often have small groups of vertices which are more homogeneous compared to the remaining vertices. For example, in a social network setup, such a group might represent vertices which share the same profession. Such a group is loosely referred to as a ``community", and the task of finding such set of vertices from the data, referred to as the ``community detection problem", has emerged as a central challenge in network analysis.  The SBM, introduced by \cite{holland1983sbm}, has emerged as the canonical setup to study this problem. 
	
	
	The well-known SBM with $k$ communities can be described as follows. The model takes as input a labeling $z: [n] \to [k]$ (where $[n]$ stands for $\{1,\ldots,n\}$ and $[k]$ is similarly defined), and a matrix of probabilities $P = (P_n(r,s))_{1\leq r,s \leq k}$. Given these parameters, first consider the partition (induced by $z$) $\C=\{\tilde{S}_1,  \cdots, \tilde{S}_k\}$ of $[n] = \tilde{S}_1 \cup \cdots \cup \tilde{S}_k$, with $\tilde{S}_{\ell} = \{ i \in [n]: z(i) = \ell \}$. Now, we construct a simple labeled graph $G=(V,E)$ with vertex set $V = [n]$ and the edges in $E$  added independently with 
	\be
	\P[ \{i,j \} \in E] = P_n(\ell, \ell') \,\,\,\,\, \textrm{if} \,\,\, z(i) = \ell, z(j) = \ell'. \label{eq:model_vanilla}
	\ee
	In this context, each cluster $\{i: z(i) = \ell\}$ is referred to as a ``community", and the community detection problem formally refers to  recovering the latent generative labels $z$, given the graph $G$ generated from the SBM \eqref{eq:model_vanilla}. It is common to consider the ``assortative"  SBM, where $P_n(\ell, \ell) \geq P_n( \ell, \ell')$ for $ 1\leq \ell\neq \ell' \leq k$. We will restrict ourselves to this setting throughout. 
	A sharp analysis of the limits of statistical inference (especially of recovering the true underlying communities)  under this model has received considerable attention recently. We do not attempt to survey the extensive literature in this area, and instead refer the reader to the two excellent surveys \cite{abbe2017recent}, \cite{moore2017survey}, and the references therein, for an extensive overview of the recent progress on this problem and related open questions.  
	
	Practitioners often fit these models to real networks to form preliminary ideas about community structure, and for exploratory data analysis (\cite{snijders1997estimation}) -- \textcolor{black}{ before proceeding to explore further details in the model such as strength of connections between and across communities}. However, while a theoretical understanding of the model has attained considerable maturity, it has also been widely reported that the model is often inappropriate in practice. The model favors graphs where the vertex degrees concentrate around a fixed value, and fails to model networks with a non-trivial degree distribution, in addition to a community structure. If this issue is ignored, and algorithms for community detection developed in the context of the SBM used on these examples, the algorithm often splits the vertices into high and low degree groups, and fails completely to uncover the true community memberships. A notable example, which exhibits this phenomenon is the political blog data of \cite{adamic2005blogs}. To address this issue, \cite{karrer2011stochastic} have introduced the ``Degree Corrected Stochastic Block Model" (henceforth abbreviated as DCSBM), which incorporates a separate ``degree" parameter for each vertex. 
	
	%
	
	The DCSBM with $k$ communities takes as input a labeling $z: [n] \to [k]$, a matrix of probabilities $P = (P_n(r,s)) \in [0,1]^{k \times k}$, and, in addition, a vector $\mathbf{\Theta} = (\theta_1, \cdots , \theta_n) \in \mathbb{R}_+^n$ where $\mathbb{R}_{+} = [0, \infty)$. Given these inputs, as before, we first partition $[n] = \cup_{\ell=1}^{k} \tilde {S}_{\ell}$, where $\tilde{S}_{\ell}$ are as outlined above. Next, we construct a graph with vertices $[n]$ and add edges independently with probabilities 
	\be
	\P[\{i,j\} \in E] = \theta_i \theta_j  P_n(\ell, \ell') \,\,\,\,\, \textrm{if} \,\, z(i) = \ell, z(j) = \ell'. \label{eq:DCSBM}
	\ee
	Note that for the general DCSBM, we can multiply each $\theta_i$ by some factor $\lambda \in (0,1)$, and adjust the $P$ matrix suitably so that the connection probabilities are unchanged. Thus the model is not identifiable as specified, and requires some extra constraints on the parameters. The usual recourse in this context is to incorporate an extra normalization to constrain 
	$\sum_{i : z(i )= \ell} \theta_i$. 
	The model becomes identifiable under this extra set of constraints \citep{gao2016community,karrer2011stochastic,zhao2012consistency,gulikers2017spectral,chen2018convexified}. In the problem we consider, the relevant parameter space for $\mathbf{\Theta}$ will automatically render the model identifiable. We comment more on this after defining the parameter space through equation \eqref{eq:parameterspace}. Finally, note that the parameters $\theta_i$ control the degrees of the vertices in the constructed graph, and will therefore be referred to as the ``degree corrections". Observe that if we set $\mathbf{\Theta} =\mathbf{1}$ in the degree corrected model introduced above, we reduce to the stochastic block model. 
	\cite{karrer2011stochastic} show empirically that model fits are often considerably improved under this more general model \eqref{eq:model_main}. Motivated by the success of the DCSBM in modeling real networks, numerous authors have, in turn, developed powerful machinery for community detection under this model (\cite{zhao2012consistency},  \cite{jin2015fast}, \cite{gao2016community}, \cite{lei2016goodness}).

	{As mentioned earlier, given a real network which is believed to come from a model like \eqref{eq:DCSBM}, there are several tasks a practitioner might seek to perform. Indeed, two basic questions include (i) clustering the vertices into meaningful groups based on their interconnections, (i.e. understand the community assignments) and (ii) understand the strength of interconnections between and across communities (i.e. infer $P_n(\ell,\ell')$'s).  Indeed, given the added flexibility of the DCSBM, one might wonder why every dataset with a potential community structure should not be modeled as a DCSBM with an appropriate number of communities. In particular, as is evident from the sharp minimax rates of community detection in Hamming Loss \citep{gao2016community}, there exists procedures which can adapt to the increasing complexity of the DCSBM. Consequently, while considering inference question (i), the choice to always work with DCSBM seems natural. However, in context of inference problem (ii) it might actually be crucial to determine whether to fit a SBM or a DCSBM. 
		In particular, analyses beyond community recovery, such as estimation of inter-connectivity strength between and across communities, require more precise differentiation between and SBM and DCSBM. For example, estimation of the matrix $P_n$ is simpler and standard for SBM (see Section \ref{section:estimation_ab} for more details) whereas optimal estimation of the same for DCSBM often requires more  nuanced understanding of the degree corrections. Consequently, fitting (i.e. a complete analysis of the model beyond only community recovery) every network dataset as a DCSBM runs the risk of over-fitting.
		Finally, the problem of choosing between SBM and DCSBM has already been alluded to in the original paper of   \cite{karrer2011stochastic} on DCSBM. In particular, \cite{karrer2011stochastic} noted, and we quote, ``We have tested our algorithms on real-world networks
		ranging in size from tens to tens of thousands of vertices. In networks with highly homogeneous degree distributions we find little difference in performance between
		the degree-corrected and uncorrected blockmodels..." This already points at a necessity of testing for model selection -- which aims to rigorously quantify the lack of difference found in fitting both models 
		by \cite{karrer2011stochastic}.
	}
	
	A natural instinct at this point is to use a likelihood ratio test (LRT) for goodness of fit. However, 
	classical asymptotics of LRTs for goodness of fit are no longer immediately valid in this case, due to the divergence in the number of parameters. This concern had been raised classically by \cite{fienberg1981comment}, who emphasize the need for proper model selection criteria in the context of the $p_1$ model, which exhibits similar features. In our context, this issue was partially addressed by \cite{yan2014model}, who use techniques motivated by statistical physics to approximate the likelihood, and derive valid sampling distributions for the test statistic.  Following the work of \cite{yan2014model}, some other model selection approaches have also been introduced (see e.g. \cite{peixoto2015model}, \cite{yan2016bayesian}). 
	In this paper, we formalize and study the goodness of fit testing problem between a SBM and a DCSBM, in the challenging regime where the degree corrections are rare and weak. To present simple formula of sharp detection thresholds, we focus on the case when both $P_n(\ell,\ell')$ are the same for all $\ell\neq \ell'$  and $P_n(\ell,\ell)$ are equal for all $\ell=1,\ldots,k$ (see Section \ref{section:estimation_ab} for generalizations). In particular, given a partition, $\C=\{\tilde{S}_1,  \cdots, \tilde{S}_k\}$ of $[n] = \tilde{S}_1 \cup \cdots \cup \tilde{S}_k$, with $\tilde{S}_{\ell} = \{ i \in [n]: z(i) = \ell \}$, we consider the DCSBM specified by the probability measure $\Ptheta^{(\C)}$ defined by
	\be
	\Ptheta^{(\C)} [ \{i,j\} \in E] 
	= \begin{cases}
		\theta_i \theta_j \frac{a}{n}\wedge 1 \quad {\rm{if }}\, z(i)=z(j) \nonumber \\
		\theta_i \theta_j \frac{b}{n}\wedge 1 \quad {\rm{o.w.}} \nonumber
	\end{cases}\label{eq:model_main}
	\ee
	The expectation and variance operators under model \eqref{eq:model_main} will be denoted by $\Etheta^{(\C)}$ and $\mathrm{Var}^{(\C)}_{\mathbf{\Theta},a,b}$ respectively. In the sequel, whenever $\C$ is clear from the context, we drop the notational dependence of the above quantities on $\C$. Finally, we note that in the model above, $0<b < a <n$ are sequences dependent on $n$.  We assume throughout that $ 0< \liminf \frac{b}{a} \leq \limsup \frac{b}{a} <1$ and define $\tau_a := \lim_{n \to \infty} \frac{a}{n}$ and $\tau_b := \lim_{n \to \infty} \frac{b}{n}$.  

	To state this problem of choosing between SBM and DCSBM formally, we define the parameter space 
	\be
	\Xi(s,A):=\left\{\begin{array}{c}\mathbf{\Theta}\in \mathbb{R}_{+}^n: |S(\mathbf{\Theta})|=s, \theta_i\ge 1+A, i\in S(\mathbf{\Theta})\end{array}\right\}, \label{eq:parameterspace}
	\ee
	where $A>0$ and for any vector $\mathbf{\Theta}\in \mathbb{R}_+^n$ we let $S(\mathbf{\Theta}) = \{i : \theta_i \neq 1\}$. Note that $A>0$ will correspond to a one sided alternative of the presence of some high degree vertices compared to the rest (discussions on two sided alternatives can be found in Section \ref{section:estimation_ab}). The vertices $i \in S(\mathbf{\Theta})$ can be interpreted as the ``popular" vertices. We will concentrate on the regime when $s = o(n)$, so that the deviations from the null are rare.
	Finally, we observe that once we specify the number of communities $k=o(\log n)$, and consider $s =o(n)$, the DCSBM \eqref{eq:DCSBM} is automatically identifiable, and we do not need the added normalization described above. \textcolor{black}{This will be clear from Section \ref{section:estimation_ab} where we provide consistent estimators of $a,b$ for the case $s=o(n)$.}
	
	%
	\cite{karrer2011stochastic} emphasized that in many real networks, these ``popular" vertices are comparatively rare, and ensuring their correct classification is often more challenging. 
	Therefore, mathematically we consider the following sequence of hypothesis testing problems
	\be 
	H_0: \mathbf{\Theta} =\mathbf{1} \quad \textrm{vs.} \quad H_1: \mathbf{\Theta} \in \Xi(s_n, A_n) \subset \mathbb{R}_+^n\setminus \{\mathbf{0}\} \label{eqn:hypo}
	\ee
	for any pair of sequences $s_n, A_n$. Throughout we shall refer to $\mathbf{\Theta}$ as the signals and parametrize signal sparsity $s_n=n^{1-\alpha}$ with $\alpha \in (0,1)$. Setting $\mathbf{Y}=(y_{ij})_{ij=1}^n$ to denote the adjacency matrix of the random graph, a statistical test for $H_0$ versus $H_1$  is a measurable $\{0,1\}$ valued function of the data $\bY$, with $1$ denoting the rejection of the null hypothesis $H_0$ and $0$ denoting the failure to reject $H_0$. For any given class of partitions $\mathcal{S}^{\nu}_k$ of $[n]$ into $k$ communities (see Section \ref{section:notation_assumptions} assumption (I) for the definition of the  notation $\mathcal{S}^{\nu}_k$), we define the worst case risk of a test $T_n(\bY)$ is defined as  
	\be 
	\ &	\mathrm{Risk}_n(T_n,\Xi (s_n, A_n),\mathcal{S}_k^{\nu})\\&:= \max_{\mathcal{C} \in \mathcal{S}^{\nu}_k}\P^{(\C)}_{\mathbf{1}, a, b}\left(T_n=1\right)+\sup_{\mathbf{\Theta} \in \Xi(s_n, A_n)} \max_{\mathcal{C}\in \mathcal{S}^{\nu}_k}\P^{(\C)}_{\mathbf{\Theta},a,b}\left(T_n=0\right).
	\label{eq:risk}
	\ee
	A sequence of tests $T_n$ corresponding to a sequence of model-problem pairs \eqref{eq:model_main}-\eqref{eqn:hypo}, is said to be asymptotically powerful (respectively asymptotically  powerless) against $\Xi(s_n ,A_n)$ if $\limsup\limits_{n\rightarrow \infty}\mathrm{Risk}_n(T_n,\Xi(s_n, A_n),\mathcal{S}^{\nu}_k)= 0$ $\text{ (respectively }\liminf\limits_{n\rightarrow \infty}\mathrm{Risk}_n(T_n,\Xi(s_n, A_n),\mathcal{S}^{\nu}_k)=1).$ \textcolor{black}{For the sake of conveying the main ideas in a clear and concise way, we shall initially provide the analysis assuming $a,b$ are known. However, all our results can be seamlessly extended to the case of unknown $a,b$ and we discuss some simple methods to take care of this in Section \ref{section:estimation_ab}.}

	\tcr{The results in this paper explore  the smallest deviations necessary for statistical procedures to detect the ``inhomogeneity" of the vertex degrees. To our knowledge, this is one of few instances where sharp detection thresholds have been achieved in the presence of a high dimensional nuisance parameter (in this case, the unknown community assignment $\C$). Moreover, our results hint at some peculiar features in this detection problem. In particular, it is natural to expect that optimal tests about degree correction parameters should be naturally based on the degree vector of the observed graph. Although this intuition leads to correct asymptotic rates of detection thresholds, it surprisingly might fail to yield the correct sharp constants in certain cases.  Our main results, as summarized below, are aimed at capturing the sharp behaviors of  the detection thresholds in this context. }
	
	\begin{enumerate}
		\item [(a)]  For dense degree corrections ($\alpha\leq \frac{1}{2}$), detection thresholds are standard and easy to derive and is attainable by simply considering the total number of edges in the graph (see Theorem \ref{thm:dense}).

		\item [(b)] For sparse degree corrections ($\alpha>\frac{1}{2}$), we first provide asymptotic detection thresholds with sharp constants whenever $a,b\gg \log{n}$. The optimal tests are based on  Higher Criticism or Maximum Degree type tests based on a linear combination of within and across (estimated) community degrees. Also, unlike \cite{mms2016}, the analysis of the lower bound on detection thresholds depends on a more delicate truncated second moment method argument. In particular, the truncation event depends crucially on the community assignment $\C$ (see Theorem \ref{thm:sparse_lower}). 
		
		\item [(c)] \textcolor{black}{For sparse degree corrections $(\alpha>\frac{1}{2})$ with $a,b\lesssim \log{n}$, a different behavior of the testing problem emerges. Here  a version of the Maximum Degree Test with different rejection region is optimal (see Theorem \ref{thm:lower_below_logn}).} 
		
		\item [(d)] For the case of balanced communities, whether it is possible to attain the sharp optimal thresholds using tests based on only the degrees of the nodes, depends heavily on the density of the graph. Theorem \ref{thm:sparsesignal_vanilla_upper} collects these results which are summarized below. 
		
		\begin{enumerate} 
			\item [(I)] For non-dense graphs $(\tau_a=\tau_b=0)$, sharp detection thresholds are attainable without knowledge of the unknown $\C$ using either the degree based Higher Criticism Test (introduced in \cite{mms2016}) or the Maximum Degree Test (for $\alpha>\frac{3}{4}$). See Theorem \ref{thm:sparsesignal_vanilla_upper_gen} and Theorem \ref{thm:sparsesignal_vanilla_upper}.
			
			\item [(II)] In contrast, for dense graphs $(\tau_a>\tau_b>0)$ the Maximum Degree Test surprisingly fails to attain the sharp detection thresholds for extremely sparse degree corrections i.e. $\alpha>\frac{3}{4} $ (See Theorem \ref{thm:sparsesignal_vanilla_upper}\ref{thm:sparsesignal_vanilla_max_lower}).  The optimal procedure runs in two stages. It involves running a suitable community recovery algorithm in stage 1, followed by a Higher Criticism Test based on a linear combination of within and across (estimated) community degrees in stage 2. See Theorem \ref{thm:sparsesignal_optimal}.

			
		\end{enumerate}	
		
		\item [(e)] In order to demonstrate the failure of Maximum Degree Test in certain regimes, we derive asymptotic distribution of the Maximum Degree of SBM's with two balanced communities and connection probabilities satisfying $a\geq b\gg(\log{n})^3$. This is achieved by using a version of Stein's Method for Poisson approximation \citep{barbour1992poisson}.  See Theorem \ref{thm:max_deg_null}.

		\item [(f)] Our results rely on some detailed computations of both log-scale and exact exponential scale asymptotics of positive linear combinations of sparse Binomial random variables in both moderate deviation and local central limit theorem scales. Since not directly obtainable from standard references, these results (see Section \ref{sec:binom}) might be of independent interest.

	\end{enumerate} 
	

	\subsection*{Assumptions and Notation}\label{section:notation_assumptions} In this subsection we collect some assumptions and notation which drive the exposition of the rest of the paper. 
	
	We first describe some of the assumptions  -- which will be generalized in Section \ref{section:estimation_ab}. The assumptions help us present sharpest results in their cleanest forms -- while they capture the essential elements needed to generalize them to more general scenarios (see Section \ref{section:estimation_ab} for details). 
	\begin{enumerate}
		\item [(I)] \textbf{Comparable Community Sizes:} 
		We denote by $\mathcal{S}^{\nu}_k$ as the set of all partitions $\mathcal{C}=\{\tilde{S}_1,\ldots,\tilde{S_k}\}$ such that $|\tilde{S_l}|\in [n/k\nu,n\nu/k]$ for all $l=1,\ldots,k$ and some fixed $\nu\geq 1$. We shall always assume $\mathcal{C}\in \mathcal{S}^{\nu}_k$ for some $\nu\geq 1$. 
		
		\item [(II)]\textbf{Bounded Number of Communities:} We assume $k$ is fixed as $n\rightarrow \infty$.
		
		\item [(III)] \textbf{Given Assortative Non-vanishing $a,b$:}   We assume that $ 0< \liminf \frac{b}{a}\\ \leq \limsup \frac{b}{a} <1$ and define $\tau_a := \lim_{n \to \infty} \frac{a}{n}$ and $\tau_b := \lim_{n \to \infty} \frac{b}{n}$. We shall also assume that both $\liminf a,\liminf b>0$ and also (initially) assume that $a,b$ are given.
	\end{enumerate}
	Once again we will like to point out that assumption (III) about the knowledge of $a,b$ is for the simplification of exposition while capturing the essence of the main ideas and resuls. We discuss in detail in Section \ref{section:estimation_ab} how the case of unknown $a,b$ can be dealt with without changing the results. As it will turn out, the case of known $a,b$ captures the essential subtleties and ingredients for the more general case as well -- and therefore we initially focus on this case. Also, as mentioned earlier, some of these assumptions are for getting sharp constants and phase transitions in the paper. It will be clear from the results that assumption (I) with exact balanced communities (i.e. $\nu=1$) is only relevant for $a,b$ exactly order $n$ and even in that case,this assumption is irrelevant (and only comparable size of communities is relevant) whenifne is only interested in optimal rates. 
	About assumption (II), we also note that under comparable community sizes, the case of $k\rightarrow \infty$ is much easier to handle since the within community degrees contribute much less. 
	
	We next describe the notation which is followed throughout the paper. First we refer to elements of $\mathcal{S}_k^{\nu}$ as community assignments. In this regard we say two community assignments $\mathcal{C}_1=\{\tilde{S}^{(1)}_1,\ldots,\tilde{S}_k^{(1)}\} $ and $\mathcal{C}_2=\{\tilde{S}_1^{(2)},\ldots,\tilde{S}_k^{(2)}\}$ are equivalent if $\inf_{\pi\in \Pi_k}\sum_{l=1}^k\I(\tilde{S}_l^{(1)}\neq \tilde{S}_{\pi(l)}^{(2)})=0$ where $\Pi_k$ contains all permutations of $\{1,\ldots,k\}$ and $\I(\cdot)$ will always stand for the indicator function. 
	For any $n \in \mathbb{N}$, we let $[n]=\{1,\ldots,n\}$. For any $i\in [n]$ we denote the degree of vertex $i$ by $d_i:=\sum_{j=1}^nY_{ij}=\sum_{j\ne i}Y_{ij}.$ 
	Throughout $\Bin(n,p)$ will stand for a generic binomial random variable with $n\in \mathbb{N}$ trials and success probability $p \in [0,1]$. The results in this paper are mostly asymptotic in nature and thus requires some standard asymptotic  notations. If $a_n$ and $b_n$ are two sequences of real numbers then $a_n \gg b_n$ (and $a_n \ll b_n$) implies that ${a_n}/{b_n} \rightarrow \infty$ (respectively ${a_n}/{b_n} \rightarrow 0$) as $n \rightarrow \infty$. Similarly $a_n \gtrsim b_n$ (and $a_n \lesssim b_n$) implies that $\liminf{{a_n}/{b_n}} = C$ for some $C \in (0,\infty]$ (and $\limsup{{a_n}/{b_n}} =C$ for some $C \in [0,\infty)$). Alternatively, $a_n=o(b_n)$ will also imply $a_n \ll b_n$ and $a_n=O(b_n)$ will imply that $\limsup{{a_n}/{b_n}} =C$ for some $C \in [0,\infty)$). 
	We write $a_n \sim b_n$ if $\lim \frac{a_n}{b_n}= 1$ and $a_n=\Omega(b_n)$ if $\liminf a_n/b_n\geq C$ for some constant $C>0$. 
	We need the following function to define our detection thresholds. For $\beta_1,\beta_2>0$ and $k\in \mathbb{N}$ let 
	\be
	\ & \rho_k(\beta_1,\beta_2)\\&=\left[ \frac{( \frac{\beta_1^2}{k}\tau_a (1 - \tau_a) +  \frac{\beta_2^2(k-1)}{k}\tau_b(1- \tau_b) ) (\frac{1}{k} \tau_a (1 - \tau_a) + \frac{(k-1)}{k}\tau_b (1- \tau_{b}) ) }{(\frac{\beta_1}{k} \tau_a +  \frac{\beta_2(k-1)}{k}\tau_b )^2} \right],
	\\ \label{eqn:rho_def}
	\ee
	where $\tau_a,\tau_b$ are defined earlier. \textcolor{black}{We adopt the convention that $\rho_k(\beta_1,\beta_2)=k(\beta_1^2+\beta_2^2(k-1))/(\beta_1+(k-1)\beta_2)^2$ for $\tau_a,\tau_b=0$.} Further, we shall always assume $b<a\leq \frac{n}{2}$ for concreteness, although the particular choice of $n/2$ can be easily replaced by $cn$ for any fixed $c\in (0,1)$. Also, throughout we drop the subscript $n$ whenever it is understood that $s, A$ are allowed to vary with $n$.

	\section{Tests}\label{section:tests} 
	In this section we formally 
	describe the testing procedures to be used. In order to construct these tests we begin with a few definitions \tcr{and recall that $a,b$ are considered to be known for now with the case of unknown $a,b$ discussed in Section \ref{section:estimation_ab}}. 
	
	Fix any community assignment $\mathcal{C}=\{\tilde{S}_1,\ldots,\tilde{S}_k\}\in \mathcal{S}^{\nu}_k$ for some $\nu>0$ and let $|\tilde{S}_l|$ denote the size of the $l^{\mathrm{th}}$ community. For any $i \in [n]$, let $\C(i) = \tilde{S}_l$ if $i \in \tilde{S}_l$ for some $l=1,\ldots,k$. Define the within-group-degree of a vertex $i\in \tilde{S}_l$ to be $d_i(1,\tilde{S}_l) = \sum_{j \in \C(i)} Y_{ij}$. Similarly set the across-group-degree of a vertex $i$  to be $d_i(2,\tilde{S}_l) = \sum_{j \in \C(i)^c} Y_{ij}$. 
	Define for $i\in \tilde{S}_l$
	\begin{align}
	\muonenull(\tilde{S}_l):=\Ezero^{(\mathcal{C})} \Big[d_i(1,\tilde{S}_l) \Big] = \left(|\tilde{S}_l|-1\right) \cdot \frac{a}{n},\nonumber \\ \mutwonull(\tilde{S}_l):=\Ezero^{(\mathcal{C})} \Big[d_i(2,\tilde{S}_l) \Big]= (n-|\tilde{S}_l|) \cdot \frac{b}{n}, \nonumber \\
	\mathrm{Var}^{(\mathcal{C})}_{\mathbf{1},a,b} (d_i(1,\tilde{S}_l)) = \left(|\tilde{S}_l|-1\right) \cdot \frac{a}{n} \cdot \Big(1- \frac{a}{n} \Big),\nonumber\\ \mathrm{Var}^{(\mathcal{C})}_{\mathbf{1},a,b} (d_i(2,\tilde{S}_l)) =   (n-|\tilde{S}_l|) \cdot \frac{b}{n} \cdot \Big( 1 -\frac{b}{n} \Big),\nonumber 
	\end{align}
	and  note that under $H_0$ the above quantities (i.e. the means and variances) depend on $\mathcal{C}$ only through $|\tilde{S}_l|$ (which under assumption I for $\mathcal{S}^{\nu}_k$ in Section \ref{section:notation_assumptions} belongs in the interval $[n/k\nu,n\nu/k]$ with $\nu\geq 1$ for all $i=1,\ldots,n$). 
	
	Finally, for any fixed positive constants $\beta_1$ and $\beta_2$ define for $i\in \tilde{S}_l$
	
	\be 
	D_i(\tilde{S}_l,\beta_1,\beta_2):=\frac{\beta_1(d_i(1,\tilde{S}_l) - \muonenull(\tilde{S}_l)) + \beta_2 (d_i(2,\tilde{S}_l) - \mutwonull(\tilde{S}_l))}{\sigmazero(\tilde{S}_l,\beta_1,\beta_2)}, 
	\ee
	where 
	\be 
	\sigmazero(\tilde{S}_l,\beta_1,\beta_2):=\sqrt{\beta_1^2 \mathrm{Var}^{(\mathcal{C})}_{\mathbf{1},a,b} (d_i(1,\tilde{S}_l)) + \beta_2^2 \mathrm{Var}^{(\C)}_{\mathbf{1},a,b} (d_i(2,\tilde{S}_l))}.
	\ee
	We are now ready to define our testing procedures.
	\begin{description}[align=left]\itemsep15pt
		\item [\textbf{Total Degree Test} :] 
		This test is based on the total degree in the observed graph i.e. $\sum_{i=1}^n d_i$. The test rejects when the observed total degree is large. The calibration of this test can be achieved by looking at the behavior of $\sum_{i=1}^n d_i$ under the null hypothesis in \eqref{eqn:hypo}. More precisely, by the Total Degree Test we mean a testing procedure which rejects when $\sum_{i=1}^n d_i$ is large (See proof of Theorem \ref{thm:dense}\ref{thm:dense_upper}).

		\item [\textbf{The Higher Criticism Tests} :]

		For any $\beta_1,\beta_2>0$, a partition $\mathcal{C}=\{\tilde{S}_1,\ldots, \tilde{S}_k\}$ of  $[n]$, and $t>0$ let 
		\be 
		HC(\hcpar;t):=\sum_{l=1}^k\sum_{i\in \tilde{S}_l} \left(\I\left(D_i(\tilde{S}_l,\beta_1,\beta_2)>t\right)-\Pzero^{(\C)}\left(D_i(\tilde{S}_l,\beta_1,\beta_2)>t\right)\right).
		\ee
		
		We then construct a version of the higher criticism test as follows. Define
		\be 
		HC(\hcpar):=\sup\left\{\begin{array}{c}GHC(\hcpar;t):=\frac{HC(\hcpar;t)}{\sqrt{\mathrm{Var}_{\mathbf{1},a,b}\left(HC(\hcpar;t)\right)}},\\ t\in \{\sqrt{2r\log{n}}:r\in (0,5)\}\cap \mathbb{N}\end{array}\right\}. 
		\ee
		By the Higher Criticism Test based on $HC(\hcpar)$ we then mean a testing procedure that rejects when the observed value of $HC(\hcpar)$ defined above is large. In particular, we let $T_{HC}(\C,\beta_1,\beta_2)$ be the test that rejects when $HC(\C,\beta_1,\beta_2)>\sqrt{\log{n}}$. Note that for any two equivalent community assignments $\C,\C'\in \mathcal{S}^{\nu}_k$ one has $HC(\C,1,1)=HC(\C',1,1)$. In particular, balanced communities case i.e. $\mathcal{S}_{k}^{\nu}$ for $\nu=1$ the construction and calibration of the tests $HC(\C,1,1)$ is free of $\C$.	\textcolor{black}{Hence while considering balanced communities, any such test is referred to as the test based on $HC(1,1)$ (we do not use the dependence of this test on $k$ in the notation).  It is easy to see the test based on $HC(1,1)$ corresponds to the degree based Higher Criticism Test introduced in \cite{mms2016} and will be also referred to as the vanilla Higher Criticism Test.} \tcr{Proofs of Theorem \ref{thm:sparsesignal_vanilla_upper_gen}\ref{thm:sparsesignal_vanilla_hc_gen} and Theorem \ref{thm:sparsesignal_vanilla_upper}\ref{thm:sparsesignal_vanilla_hc} provide detailed analyses of this class of Higher Criticism Tests. Although the plan of the analyses are standard and motivated by those in \cite{mms2016}, the actual proofs are somewhat tedious due to the necessity to consider arbitrary positive linear combinations of within and across community degrees. In particular, we derive relevant local central limit theorem type asymptotics crucial to these analyses in Section \ref{sec:binom}. }

		\item [\textbf{The Maximum Degree Tests} :] For any $\beta_1,\beta_2>0$, and partition $\mathcal{C}=\{\tilde{S}_1,\ldots, \tilde{S}_k\}$ of  $[n]$ , by the Maximum Degree Test based on $d_{\max}(\C,\beta_1,\beta_2)$ we mean the procedure that rejects for large values of $d_{\max}(\C,\beta_1,\beta_2):=\max_{l=1}^k\max_{i\in \tilde{S}_l}D_i(\tilde{S}_l,\beta_1,\beta_2)$. In particular, for any $\delta>0$, we let $T_{d_{\max}}(\C,\beta_1,\beta_2,\delta)$ to be the test that rejects when $\max_{l=1}^k\max_{i\in \tilde{S}_l}D_i(\tilde{S}_l,\\ \beta_1,\beta_2)>\sqrt{2(1+\delta)\log{n}}$.
		Note that for any two equivalent community assignments $\C=\{\tilde{S}_1,\ldots, \tilde{S}_k\},\C'=\{\tilde{S}_1',\ldots, \tilde{S}_k'\}\in \mathcal{S}^{\nu}_k$, $\max_{l=1}^k\max_{i\in \tilde{S}_l}\\D_i(\tilde{S}_l,\beta_1,\beta_2)=\max_{l=1}^k\max_{i\in \tilde{S}_l'}D_i(\tilde{S}_l',\beta_1,\beta_2)$. In particular, balanced communities case i.e. $\mathcal{S}_{k}^{\nu}$ for $\nu=1$ the construction and calibration of the tests $d_{\max}(\C,1,1)$ is free of $\C$.	\textcolor{black}{Hence while considering balanced communities, any such test is referred to as
			the test based on $d_{\max}(1,1)$. It is easy to see the test based on $d_{\max}(1,1)$ is simply the test that rejects for large values of of the maximum degree $d_{\max}:=\max\{d_1,\ldots,d_n\}$ and will be also referred to as the vanilla Maximum Degree Test.} \tcr{Proofs of Theorem \ref{thm:sparsesignal_vanilla_upper_gen}\ref{thm:sparsesignal_vanilla_max_upper_gen}, Theorem \ref{thm:sparsesignal_vanilla_upper}\ref{thm:sparsesignal_vanilla_max_upper}, and Theorem \ref{thm:sparsesignal_vanilla_upper}\ref{thm:sparsesignal_vanilla_max_lower} provide detailed analyses of this class of Maximum Degree Tests. A crucial component in our proofs is the derivation of the asymptotic distribution of $d_{\max}(1,1)$ in Theorem \ref{thm:max_deg_null} which, in turn, relies on exponential scale asymptotics for a class of sums of sparse Binomial random variables in the moderate deviation regime (see Lemma \ref{lemma:binomial_tail_exp_scale}).}

	\end{description}

	\section{Main Results}\label{section:main_results}
	In this section we present the main results of the paper along with their implications. The proofs of these results can be found in the supplementary material \citep{ms2019}. Owing to the differential behavior of the detection problem, we divide our presentation into two main subsections based on the signal sparsity $\alpha$.

	\subsection{Dense Signal Regime $\alpha \leq \frac{1}{2}$} 
	The behavior of the detection problem in the dense signal $\alpha\leq \frac{1}{2}$ regime is particularly simple. Intuitively, since there are many vertices in the graph which have a higher connection probability than under the null hypothesis, under the dense signal regime a natural test is the Total Degree Test introduced in Section \ref{section:tests}. This intuition indeed turns out to be correct in the sense that no other test works when the Total Degree Test fails. The next theorem makes this precise.
	\begin{theorem}\label{thm:dense}
		Fix $0 < \alpha \leq \frac{1}{2}$ and set $\cdense(\alpha) = \frac{1}{2} - \alpha$. Then the following hold for $\mathcal{S}^{\nu}_k$  with any $\nu\geq 1$.
		\begin{enumerate}[label=\textbf{\roman*}.]
			\item\label{thm:dense_upper} The Total Degree Test is asymptotically powerful if
			\begin{align}
			A \geq \frac{n^{-r}}{\sqrt{a}}, \,\,\, r < \cdense(\alpha). \nonumber 
			\end{align}
			\item \label{thm:dense_lower}All tests are asymptotically powerless if
			\begin{align}
			A \leq \frac{n^{-r}}{\sqrt{a}}, \,\,\, r > \cdense(\alpha). \nonumber 
			\end{align}
		\end{enumerate}
	\end{theorem}
	One feature of Theorem \ref{thm:dense} above is that the detection thresholds given by $\cdense$ do not change based on the nature of $\tau_a$ and $\tau_b$. We will see later that this behavior is in stark contrast to that of the detection thresholds in the sparse regime $\alpha > \frac{1}{2}$.

	\subsection{Sparse Signal Regime $\alpha > \frac{1}{2}$} The behavior of detection problem in the sparse signal regime is much more subtle and is substantially different depending on whether $a,b\gg \log{n}$ or $a,b\lesssim \log{n}$. As we shall see, even in the regime $a,b\gg \log{n}$ a further subtlety arises when $a,b=\Omega(n)$ and communities are balanced. Below, we divide our study accordingly.
	
	\subsection{$a,b\gg \log{n}$}
	Intuitively, since we are testing for degree heterogeneity which are sparse in occurrence, one should in principle be able to produce tests similar to those in \cite{mms2016} by looking at abnormal behavior of extreme degrees. Indeed this intuition is captured by the degree based Higher Criticism Test and the Maximum Degree Test studied in \cite{mms2016}. Success of similar tests naturally fits into the narrative that the behavior of the detection problem for degree heterogeneity does not depend on the knowledge of community assignment. Although the heart of this narrative is correct, the implications should be taken with a grain of salt.  In particular, as we argue in this section, this intuition of constructing tests surprisingly fails for completely dense graphs (i.e. when $0<\tau_b<\tau_a$) -- especially for the case of balanced communities. More precisely, for these very dense graphs, the optimal procedures require the knowledge of the community assignments. Although this is problematic at first glance, the experienced reader will immediately realize that when $0<\tau_b<\tau_a$, it is very easy to recover the communities consistently, at least when the degree heterogeneity parameter $\theta_1,\ldots,\theta_n$ are not too rough \citep{gao2016community}.

	To elaborate on this peculiar behavior of the detection problem it is instructive to start with the information theoretic lower bound. We use the notation 
	\be
	\sigma_{n0}^2&=\frac{n}{k}\cdot\frac{a}{n}\left(1-\frac{a}{n}\right)+\frac{(k-1)n}{k}\cdot\frac{b}{n}\left(1-\frac{b}{n}\right) \label{eqn:sigmazero_def}
	\ee
	in all subsequent results.
	
	\begin{thm}\label{thm:sparse_lower}
		Assume $\log{n}\ll b<a$ and $\mathcal{S}^{\nu}_k$ with any $\nu\geq 1$. Let $\alpha>\frac{1}{2}$  and consider the signal strength 
		\be
		A = \sqrt{\frac{C\log n}{\sigmazero^2}}, \label{eq:signal_const}
		\ee
		where $\sigmazero^2$ is defined in \eqref{eqn:sigmazero_def}.
		Then all tests are asymptotically powerless if $C < \csparse(\alpha)$, where 
		\begin{align}
		\csparse(\alpha) = \begin{cases}
		2 \Big( \frac{\frac{1}{k} \tau_a (1-\tau_a) +  \frac{k-1}{k} \tau_b(1-\tau_b)}{\frac{1}{k} \frac{\tau_a}{1- \tau_a} + \frac{k-1}{k} \frac{\tau_b}{1-\tau_b}}  \Big) \Big(\alpha - \frac{1}{2} \Big) \quad {\rm{for}}\,\, \frac{1}{2} < \alpha < \frac{3}{4}. \\
		2 \Big( \frac{ \frac{1}{k}\tau_a(1-\tau_a) +  \frac{k-1}{k} \tau_b(1-\tau_b)}{\frac{1}{k} \frac{\tau_a}{1- \tau_a} + \frac{k-1}{k} \frac{\tau_b}{1-\tau_b}}  \Big) \Big( 1- \sqrt{1-\alpha}\Big)^2 \quad {\rm{for}}\,\, \alpha \geq \frac{3}{4}. 
		\end{cases} \nonumber 
		\end{align}
		In particular, when $\tau_a = \tau_b =0$, the correct constant $\csparse(\alpha)$ is obtained by taking the limit as $\tau_a, \tau_b \to 0$, so that 
		\begin{align}
		\csparse(\alpha) = \begin{cases}
		2 \Big( \alpha - \frac{1}{2} \Big) \quad {{\rm for }} \,\, \frac{1}{2} < \alpha < \frac{3}{4}. \\
		2 \Big( 1 - \sqrt{1- \alpha} \Big)^2 \quad{{\rm for }}\,\, \alpha > \frac{3}{4}.  
		\end{cases}\nonumber 
		\end{align}
	\end{thm}

	We derive Theorem \ref{thm:sparse_lower} using an information theoretic lower bound for a simpler problem, where the true community assignments are known in advance. The proof is based on the truncated second moment argument with the main challenge being the choice of the truncation event. Unlike \cite{mms2016}, a non-signal-edge-deleted degree based truncation is not enough to yield the desired sharp thresholds. Instead one needs to take into account the knowledge of community assignments as well (at least when $\tau_a,\tau_b$ are positive). Finally we note that this simpler problem with known community assignments always furnishes a lower bound for problem \eqref{eqn:hypo}. 
	
	
	We now explore valid statistical procedures which work up to this threshold. To this end, the next results establish performance bounds on Higher Criticism based tests. For the statement of the following results we recall the definition of $\rho_k(\beta_1,\beta_2)$ from \eqref{eqn:rho_def}.

	\begin{theorem}\label{thm:sparsesignal_vanilla_upper_gen}
		Assume either $\log{n}\ll b<a\ll n$ and $\mathcal{S}^{\nu}_k$ with any $\nu\geq 1$ or $a,b=\Omega(n)$ and $\mathcal{S}^{\nu}_k$ with $\nu=1$. Let $\alpha>\frac{1}{2}$  and consider the signal strength
		\begin{align}
		A = \sqrt{\frac{C\log n}{\sigmazero^2}},\nonumber
		\end{align}
		where $\sigmazero^2$ is defined in \eqref{eqn:sigmazero_def}.
		\begin{enumerate}[label=\textbf{\roman*}.]
			\item \label{thm:sparsesignal_vanilla_hc_gen} If  $C>\chc(\beta_1,\beta_2,\alpha)$	where  
			\begin{align}
			\chc(\beta_1,\beta_2,\alpha) = \begin{cases}
			2 \rho_k(\beta_1,\beta_2,) \Big(\alpha - \frac{1}{2} \Big) \quad {\rm{for}}\,\, \frac{1}{2} < \alpha < \frac{3}{4}. \\
			2 \rho_k(\beta_1,\beta_2,) \Big( 1- \sqrt{1-\alpha}\Big)^2 \quad {\rm{for}}\,\, \alpha \geq \frac{3}{4},
			\end{cases} \nonumber 
			\end{align}
			then
			\be 
			\ & \max_{\mathcal{C} \in \mathcal{S}_k^{\nu}}\P^{(\C)}_{\mathbf{1}, a, b}\left(T_{HC}(\C,\beta_1,\beta_2)=1\right)\\ &+\sup_{\mathbf{\Theta} \in \Xi(s_n, A_n)} \max_{\mathcal{C} \in \mathcal{S}_k^{\nu}}\P^{(\C)}_{\mathbf{\Theta},a,b}\left(T_{HC}(\C,\beta_1,\beta_2)=0\right)=0. \\\label{eq:risk_hc_gen}
			\ee

			\item \label{thm:sparsesignal_vanilla_max_upper_gen} If $C>\cmax(\beta_1,\beta_2,\alpha)$ with  
			\begin{align}
			\cmax(\alpha) = 
			2 \rho_k(\beta_1,\beta_2) \Big( 1- \sqrt{1-\alpha}\Big)^2 \quad {\rm{for}}\,\, \alpha \geq \frac{3}{4}, \nonumber
			\end{align}
			then there exists $\delta>0$ such that
			\be 
			\ & \max_{\mathcal{C} \in \mathcal{S}_k^{\nu}}\P^{(\C)}_{\mathbf{1}, a, b}\left(T_{d_{\max}}(\C,\beta_1,\beta_2,\delta)=1\right)\\
			&+\sup_{\mathbf{\Theta} \in \Xi(s_n, A_n)} 	\max_{\mathcal{C} \in \mathcal{S}_k^{\nu}}\P^{(\C)}_{\mathbf{\Theta},a,b}\left(T_{d_{\max}}(\C,\beta_1,\beta_2,\delta)=0\right)=0. \\\label{eq:risk_max_gen}
			\ee

			
		\end{enumerate}
	\end{theorem} 
	
	\begin{remark}\label{remark:unbalanced_communities1} \textcolor{black}{We note that for graphs which are not completely dense (as quantified by $a,b\ll n$) the results above are valid for community assignments $\mathcal{C}\in \mathcal{S}^{\nu}_k$ which can be unbalanced up to a factor $\nu$ for any $\nu\geq 1$. For  completely dense graphs (as quantified by $a,b=\Omega(n)$ or equivalently $\tau_a,\tau_b>0$), the condition can be relaxed to  $\nu\downarrow 1$ fast enough and our results will only match the lower bound in Theorem \ref{thm:sparse_lower} in a rate optimal sense otherwise. Since for balanced communities (as quantified by $\nu=1$ above) our results are sharp with optimal constants, we only present the results for the balanced communities in this case.}
		
	\end{remark}

	Note that $T_{HC}(\C,\beta_1,\beta_2)$ and $T_{d_{\max}}(\C,\beta_1,\beta_2,\delta)$ are not statistically valid tests for all $\beta_1, \beta_2$, since they assume the true community assignment $\C$ known (as is the case for \eqref{eq:risk_hc_gen} and \eqref{eq:risk_max_gen}). To state the optimal procedure we therefore need to define notation for community recovery algorithms.  For any two community assignments $\mathcal{C}_1=\{\tilde{S}^{(1)}_1,\ldots,\tilde{S}_k^{(1)}\},\mathcal{C}_2=\{\tilde{S}_1^{(2)},\ldots,\tilde{S}_k^{(2)}\}$  we define the distance between $\C_1,\C_2$ to be 
	\be 
	\mathrm{dist}(\C_1,\C_2)=\inf_{\pi\in \Pi_k}\sum_{l=1}^k\I\left(\tilde{S}_l^{(1)}\neq \tilde{S}_{\pi(l)}^{(2)}\right).
	\ee
	where $\Pi_k$ contains all permutations of $\{1,\ldots,k\}$.
	For any measurable $\hat{\C}=\{\widehat{\tilde{S}}_1,\ldots,\widehat{\tilde{S}}_k\}$,we define the risks of community recovery over $\C\in \mathcal{S}_{k}^{\nu}$ with $\nu\geq 1$ to be
	\be 
	\mathrm{Risk}_n(\hat{\C},\Xi(s, A)):=\sup_{\mathbf{\Theta} \in \Xi(s, A)} \max_{\mathcal{C}\in \mathcal{S}_k^{\nu}}\Etheta^{(\C)}\left(\mathrm{dist}(\hat{\C},\C)\right).
	\ee
	We say $\hat{\C}$ is strongly consistent if $\mathrm{Risk}_n(\hat{\C},\Xi(s, A))\rightarrow 0$. With this definition we are ready to state the sharp upper bound which matches the lower bounds in Theorem \ref{thm:sparse_lower}.
	
	\begin{theorem}\label{thm:sparsesignal_optimal}
		Assume either $\log{n}\ll b<a\ll n$ and $\mathcal{S}^{\nu}_k$ with any $\nu\geq 1$ or $a,b=\Omega(n)$ and $\mathcal{S}^{\nu}_k$ with $\nu=1$. Consider the signal strength
		\begin{align}
		A = \sqrt{\frac{C\log n}{\sigmazero^2}},\nonumber
		\end{align} 
		where $\sigmazero^2$ is defined in \eqref{eqn:sigmazero_def}. Let $\hat{\mathcal{C}}\subset [n]$ be measurable such that $\mathrm{Risk}_n(\hat{\C},\Xi(s, A))\rightarrow 0$ 
		and let
		\begin{align}
		\beta_1^* = \frac{1}{1- \tau_a} \frac{1}{\sqrt{ \frac{1}{k} \frac{\tau_a}{1- \tau_a} + \frac{k-1}{k} \frac{\tau_b}{1- \tau_b}}},\,\,\,\,
		\beta_2^* =  \frac{1}{1- \tau_b} \frac{1}{\sqrt{ \frac{1}{k} \frac{\tau_a}{1- \tau_a} +  \frac{k-1}{k} \frac{\tau_b}{1- \tau_b}}}, \quad \text{when}\quad \tau_a,\tau_b>0.\nonumber
		\end{align}
		Define $\beta_1^*=\beta_2^*=1$ otherwise.
		\begin{enumerate}[label=\textbf{\roman*}.]
			\item \label{thm:sparsesignal_optimal_hc} The test based on $HC(\hat{\mathcal{C}},\beta_1^*,\beta_2^*)$ is powerful {if} $C>\chc^{\mathrm{opt}}(\alpha)$	where  
			\begin{align}
			\chc^{\mathrm{opt}}(\alpha) = \begin{cases}
			2 \rho_k(\beta_1^*,\beta_2^*) \Big(\alpha - \frac{1}{2} \Big) \quad {\rm{for}}\,\, \frac{1}{2} < \alpha < \frac{3}{4}. \\
			2 \rho_k(\beta_1^*,\beta_2^*) \Big( 1- \sqrt{1-\alpha}\Big)^2 \quad {\rm{for}}\,\, \alpha \geq \frac{3}{4}.
			\end{cases} \nonumber 
			\end{align}

			\item \label{thm:sparsesignal_optimal_max} The test based on $d_{\max}(\hat{\mathcal{C}},\beta_1^*,\beta_2^*)$ is powerful {if} $C>\cmax^{\mathrm{opt}}(\alpha)$ with  
			\begin{align}
			\cmax^{\mathrm{opt}}(\alpha) = 
			2 \rho_k(\beta_1^*,\beta_2^*) \Big( 1- \sqrt{1-\alpha}\Big)^2 \quad {\rm{for}}\,\, \alpha \geq \frac{3}{4}. \nonumber
			\end{align}
		\end{enumerate}
	\end{theorem}
	\begin{remark}\label{remark:unbalanced_communities3} \textcolor{black}{We note that Remark \ref{remark:unbalanced_communities1} regarding the assumptions on the community sizes is also applicable here.}
		
	\end{remark}

	Since $\rho_k(\beta_1^*,\beta_2^*)$ matches the optimal threshold from Theorem \ref{thm:sparse_lower} i.e. $\Big( \frac{\frac{1}{k} \tau_a (1-\tau_a) +  \frac{k-1}{k} \tau_b(1-\tau_b)}{\frac{1}{k} \frac{\tau_a}{1- \tau_a} + \frac{k-1}{k} \frac{\tau_b}{1-\tau_b}}  \Big)$ (which equals $1$ by our convention when $\tau_a=\tau_b=0$), Theorem \ref{thm:sparsesignal_optimal} implies that the following two-stage procedure is optimal $a,b \gg \log n$.  
	\begin{enumerate}
		\item[(i)] Run a community detection algorithm to construct $\hat{\C}$ (e.g.  Algorithm 1 of \cite{gao2016community}). 
		\item[(ii)] Reject if  $T_{HC}(\hat{C},\beta_1^*,\beta_2^*)>\sqrt{\log{n}}$. 
	\end{enumerate}
	
	The proof of the validity of the above two-stage procedure is easy. In particular,  in the regime of dense graphs $(\text{i.e. when}\ a,b\gg \log{n})$, strongly consistent community detection $(\mathrm{Risk}_n(\hat{\C},\Xi(s, A))  \rightarrow 0)$ is indeed possible whenever $\|\mathbf{\Theta}\|_{\infty}=o(n)$ and $a,b\gg \log{n}$ (see e.g. \cite{gao2016community} -- \textcolor{black}{where the condition $\nu\in [1,\sqrt{5/3}]$ made in that paper is not necessary here if only want strong consistency}). As a consequence, for any such $\mathbf{\Theta}$, Theorem \ref{thm:sparsesignal_optimal} justifies the optimality of the test based on $T_{HC}(\hat{C},\beta_1^*,\beta_2^*)$. Finally, for $\|\mathbf{\Theta}\|_{\infty}\gtrsim n$, the problem is trivial by using a test based on the maximum degree (reject when the maximum degree is bigger than $cn$ for some universal constant $c>0$). Combining these two cases by union bound yields the desired sharp optimality of the two-stage procedure.  

	This two-stage procedure is enough to complete the story of sharp detection thresholds. However, the fact that we actually have to run a strongly consistent community detection algorithm to achieve the sharp results, raises the following question -- ``is it possible to simply use degree based tests without knowledge of the community assignments which achieve the same sharp results?" We now provide some evidence to suggest that for completely dense graphs this seems not to be the case. To this end we focus on the balanced communities case i.e. $\mathcal{S}_{k}^{\nu}$ for $\nu=1$. Indeed, for balanced communities the calibration of the tests $HC(\C,1,1)$ and $d_{\max}(\C,1,1)$ considered in Theorem \ref{thm:sparsesignal_vanilla_upper_gen} is free of $\C$. 
	Our next theorem collects the performance of these simple degree based tests for balanced communities.

	
	\begin{theorem}\label{thm:sparsesignal_vanilla_upper}
		Assume $\log{n}\ll b<a$ and $\mathcal{S}^{\nu}_k$ with $\nu=1$. 	
		Let $\alpha>\frac{1}{2}$  and consider the signal strength
		\be
		A = \sqrt{\frac{C\log n}{\sigmazero^2}},\label{eq:alt1}
		\ee
		where $\sigmazero^2$ is defined in \eqref{eqn:sigmazero_def}.
		\begin{enumerate}[label=\textbf{\roman*}.]
			\item \label{thm:sparsesignal_vanilla_hc} The test based on $HC(1,1)$ is powerful {if} $C>\chc(\alpha)$	where  
			\begin{align}
			\chc(\alpha) = \begin{cases}
			2 \rho_k(1,1) \Big(\alpha - \frac{1}{2} \Big) \quad {\rm{for}}\,\, \frac{1}{2} < \alpha < \frac{3}{4}. \\
			2 \rho_k(1,1) \Big( 1- \sqrt{1-\alpha}\Big)^2 \quad {\rm{for}}\,\, \alpha \geq \frac{3}{4}.
			\end{cases} \nonumber 
			\end{align}

			\item \label{thm:sparsesignal_vanilla_max_upper} The test based on $d_{\max}(1,1)$ is powerful if $C>\cmax(\alpha)$ with  
			\begin{align}
			\cmax(\alpha) = 
			2 \rho_k(1,1) \Big( 1- \sqrt{1-\alpha}\Big)^2 \quad {\rm{for}}\,\, \alpha \geq \frac{3}{4}. \nonumber
			\end{align}
			
			\item \label{thm:sparsesignal_vanilla_max_lower} \textcolor{black}{Assume $a,b\gg(\log{n})^3$}. The test based on $d_{\max}(1,1)$ is powerless if $C<\cmax(\alpha)$.
		\end{enumerate}
	\end{theorem} 
	
	%

	To understand why this result is surprising, it is instructive to compare these results to analogous ones derived in the context of the sparse signal detection problem for sequence models. In particular, motivated by the long series of results on sparse signal detection problems \citep{Ingster4,Jin1,Candes,mukherjee2015hypothesis,arias2015sparse} and recent work on heterogeneity detection over sparse Erd\H{o}s-R\'{e}nyi random graphs under the $\beta$-model \citep{mms2016}, we expect that the Maximum Degree Test and the Higher Criticism should both perform optimally with sharp constants for very sparse signals $(\alpha\geq \frac{3}{4})$. Moreover, the Higher Criticism Test should be provably better than the Maximum Degree Test for denser signals with $\alpha \in (1/2,3/4)$. The observation that $\rho(1,1)=1$ for $\tau_a = \tau_b =0$, in conjunction with Theorem \ref{thm:sparse_lower} establishes the expected intuitive picture for all $a,b$ sequences with $\tau_a = \tau_b =0$. However, next consider the regime where $\tau_a > \tau_b >0$ and note that the tests based on $HC(1,1)$ and $d_{\max}(1,1)$ are based on the degrees $(d_1,\ldots,d_n)$ only. However, we note that  $\tau_a>\tau_b>0$ implies
	\textcolor{black}{$\rho_k(1,1)>\Big( \frac{\frac{1}{k} \tau_a (1-\tau_a) +  \frac{k-1}{k} \tau_b(1-\tau_b)}{\frac{1}{k} \frac{\tau_a}{1- \tau_a} + \frac{k-1}{k} \frac{\tau_b}{1-\tau_b}}  \Big)$} and thus there is a gap between the thresholds derived in Theorem \ref{thm:sparsesignal_vanilla_upper} and Theorem \ref{thm:sparse_lower}. Although we do not have a similar performance lower bound for the vanilla Higher Criticism Test, we strongly believe that at least in the extreme signal sparsity regime $(\alpha\geq \frac{3}{4})$, the Maximum Degree Test and the Higher Criticism Test are essentially similar. 
	This directly implies the rather surprising result that on dense graphs ($\tau_a \geq \tau_b >0$), for very sparse alternatives, the maximum degree test is not, in fact, optimal in terms of detection thresholds. This is in sharp contrast to the usual results expected for Gaussian sequence models, or for random graph models with ``exchangeable" degrees. We illustrate the differences between the two thresholds in Fig \ref{fig:main}.
	
	\begin{figure}
		\begin{center}
			\includegraphics[width=6in, height=3in]{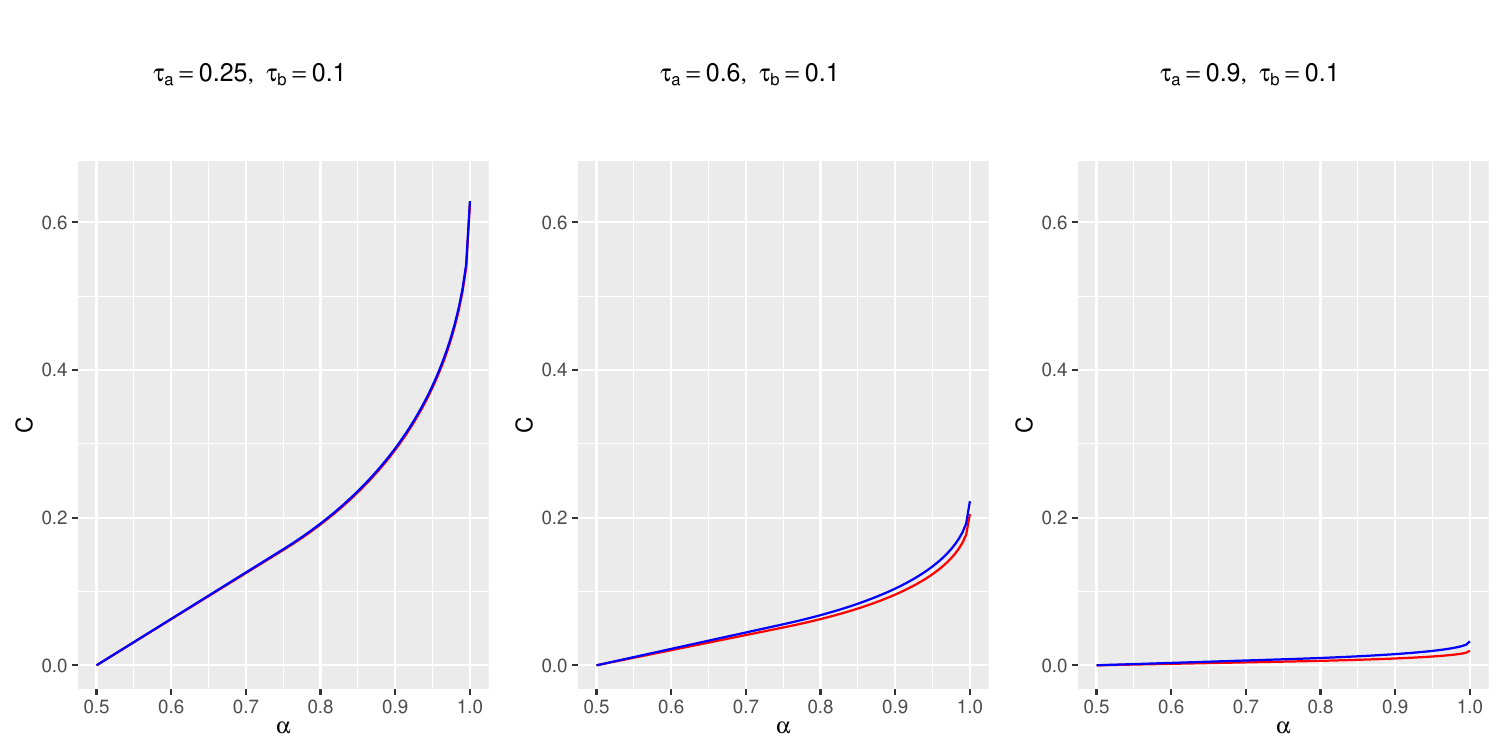}
		\end{center}
		\caption{The naive threshold $\rho_2(1,1)$ (in BLUE) and the correct information theoretic threshold $\rho_2(\beta_1^*,\beta_2^*)$ (in RED) are shown for different values of $\tau_a, \tau_b$. Note the vanishing difference between these thresholds as $\tau_a$ and $\tau_b$ becomes smaller.}
		\label{fig:main} 
	\end{figure}

	We now briefly comment on the analysis in the proof of Theorem \ref{thm:sparsesignal_vanilla_upper}\ref{thm:sparsesignal_vanilla_max_lower}.
	As mentioned earlier, the lower bound statement on the Maximum Degree Test in Theorem \ref{thm:sparsesignal_vanilla_upper}\ref{thm:sparsesignal_vanilla_max_lower} is indeed necessary to demonstrate the competition between the HC and max-degree based procedures. Analysis of the lower bounds for the vanilla Maximum Degree Test requires good control over the null distribution of the test statistic. Although the null distribution of the maximum degree of an Erd\H{o}s-R\'{e}nyi graph is standard in literature \citep[Theorem 3.3$^{'}$]{bollobas}, we could not find the corresponding results for Stochastic Block Models. To this end, our next result derives the asymptotic sampling distribution of the maximum degree under the null hypothesis, after appropriate centering and scaling. Dropping notational dependence on the true underlying community assignment $\C$ we let $\muzero = \E_{\mathbf{1}, a,b}[d_1]$, $\sigmazero = {\textrm{Var}}_{\mathbf{1},a,b}[d_1]$ (note these do not depend on particular community $\C$ for balanced communities $\C\in \mathcal{S}_{k}^{\nu}$ with $\nu=1$). Then we have the the following result.
	\begin{theorem}
		\label{thm:max_deg_null}
		Let $b \gg (\log n)^3$. In this case, we have, as $n \to \infty$, 
		\begin{align}
		\ & \P_{\mathbf{1},a,b}\Big[ \frac{\max_i d_i - \muzero }{\sigmazero} \leq \sqrt{2 \log n} \Big( 1 - \frac{\log \log n + \log (4 \pi)}{4 \log n} + \frac{y}{2 \log n}\Big) \Big] \\ &\to \exp\Big[- {\textrm{e}}^{-y} \Big]. \nonumber 
		\end{align}
	\end{theorem}

	\begin{remark}
		We note that after appropriate centering and scaling, the null distribution of the maximum converges to a Gumbel distribution. It is specifically interesting to compare this result to the asymptotic distribution of the maximum degree in an Erd\H{o}s-R\'{e}nyi random graph. 
		A direct proof in that case proceeds using the method of moments \citep[Theorem 3.3$^{'}$]{bollobas}. In the case of the SBM, the individual degrees are no longer binomial, but rather a sum of two independent Binomial random variables. As a result, many direct computations involving the degrees become considerably more involved. In our proof, we circumvent this difficulty, and establish this result using a softer argument, based on a version of Stein's method for Poisson approximation \citep{barbour1992poisson}. Finally, the result and proof of Theorem \ref{thm:max_deg_null} is actually valid for $d_{\max}(\C,1,1)$ as well for any $\C\in S_{k}^{\nu}$ with $\nu\geq 1$.
	\end{remark}


	\subsection{$a,b\lesssim \log{n}$}
	The behavior of the thresholds for $a,b\lesssim \log{n}$ is significantly different from the earlier subsection. 
	%
	In this section, we take a closer look at this regime, and derive the optimal testing thresholds in this case. Recall that we parametrize the signal sparsity $s= n^{1-\alpha}$, and focus on $\alpha \in (0, 1/2)$. Our results in this case are analogous to the sparse regime in \cite{arias2015sparse}. To an expert, this is hardly surprising, as the degrees of our graph behave like i.i.d. Poisson samples under the null in this regime. We assume throughout this section that \textcolor{black}{$0<\liminf b \leq \liminf a \lesssim \log n$. For technical convenience, we assume that $0<\liminf b \leq \liminf a$ i.e. both $a,b$ are bounded away from $0$}  To derive a sharp detection threshold in this regime, we parametrize 
	\begin{align}
	1+A = \left(\frac{\log n}{\frac{a}{k} + \frac{(k-1)b}{k}}  \right)^{\gamma} \nonumber 
	\end{align}
	for some $\gamma >0$. 
	
	\begin{thm}\label{thm:lower_below_logn}\textcolor{black}{
			Assume $\alpha>1/2$, $b\leq a \lesssim \log n$ and $\liminf b>0$ as $n \to \infty$. Then the following hold for any $\mathcal{S}_{k}^{\nu}$ with $\nu\geq 1$.}
		
		\begin{enumerate}[label=\textbf{\roman*}.]
			\item \label{thm:lower_below_logn_lower} All tests are asymptotically powerless if $\gamma < \alpha$. 
			
			\item  \label{thm:lower_below_logn_upper} A test which rejects for large values of the maximum degree is asymptotically powerful if $\gamma > \alpha$.
		\end{enumerate}
	\end{thm}
	Theorem \ref{thm:lower_below_logn} demonstrates a different behavior of the detection problem compared to $a,b\gg \log{n}$ where a vanishing $A$ is detectable even in the sparse signal regime. Here instead, one needs $A\rightarrow \infty$ in the prescribed rate above for successful detection. The optimal test is once again based on a version of the Maximum Degree Test -- albeit with a different rejection region. A careful look at the proof of Theorem \ref{thm:lower_below_logn} establishes that the Max-degree test does not require prior estimation of $a,b$ and $\C$. 
	
	\section{Generalizing the Assumptions}\label{section:estimation_ab} In this section we discuss how to go beyond the assumptions  made throughout the paper.
	
	\begin{description}[align=left]\itemsep15pt
		
		\item [\textbf{Unknown $a,b$} :] 
		We first focus on assumption (III) made in Section \ref{section:notation_assumptions} about the known nature of $a,b$ and begin by making a couple of observations. First note that while understanding unknown $a,b$ we can safely assume that $a,b\gg \log{n}$ since the results in the other regime, as covered by Theorem \ref{thm:lower_below_logn}, do not depend on the knowledge of $a,b$.  Next, note that Total Degree Test which is optimal in the dense regime (see Theorem \ref{thm:dense}) only depends on $a,b$ in deciding its cut-off. In particular, the proof of Theorem \ref{thm:dense} shows that the optimal test rejects when the total degree ($\sum_{i}d_i$) is larger than $C(a+b+ Abs)$ for a constant depending on $k$. Consequently it is easy to see that it is enough to find estimators $\hat{a},\hat{b}$ of $a,b$ such that $\hat{a}/a,\hat{b}/b$ is bounded away from $0$ and $\infty$ with high probability. \textcolor{black}{Next, for simplicity, consider the results regarding the maximum degree test by replacing the unknown $a,b$ in the construction of the test by estimators $\hat{a},\hat{b}$. Provided $\max\{|\frac{\hat{a}}{a}-1|,|\frac{\hat{b}}{b}-1|\}=o_p\left(\sqrt{\log{n}/(a+b)}\right)$, a simple modification of the proofs show that the exact same results from Theorem \ref{thm:sparsesignal_optimal}\ref{thm:sparsesignal_optimal_max} and Theorem \ref{thm:sparsesignal_vanilla_upper}\ref{thm:sparsesignal_vanilla_max_upper} remain valid. This automatically provides the correct rate of detection thresholds for sparse signals and unknown $a,b$ -- and also the right constants for $\alpha>3/4$. We believe that the same results continue to hold even for the HC test by plugging in $\hat{a},\hat{b}$ with the desired properties as above. However, the computations are much more complicated.} 
		Finally, the $a,b\lesssim \log{n}$ case is only relevant for the dense regime analysis -- where as noted above it is enough to find estimators $\hat{a},\hat{b}$ of $a,b$ such that $\hat{a}/a,\hat{b}/b$ is bounded away from $0$ and $\infty$ with high probability. This can be easily achieved by simply estimating them both by the average degree in the graph. These constructions of $\hat{a},\hat{b}$ show the validity of the results in the paper even for unknown $a,b$ -- at least for rate optimal hypothesis testing. 

		In conclusion, in the case of unknown $a,b$ and given $\alpha > 1/2$, we construct a test as follows. First, let $\mathcal{T}_n$ denote $t_n$ vertices chosen at random from the graph.  Then, with probability converging to $1$, all the vertices in $\mathcal{T}_n$ correspond to null vertices as long as $t_n\ll n/s$. 
		Now, let  $\hat{E}(C)$ denote the event that $\sum_{i,j\in \mathcal{T}_n} Y_{i,j}/t_n^2>C\log{n}/n$. Then, on the event $ \hat{E}(C)$, perform the optimal community detection algorithm from \cite{gao2016community} on the \textit{entire} graph to obtain $\hat{\C}$. Given the community assigments, use the  $\mathcal{T}_n$ vertices to estimate $a,b$ by simply using within and across estimated community degrees -- denote these by $\hat{a},\hat{b}$. It is easy to check that on the event $\hat{\C}=\C$ (where $\C$ is the true underlying community assigment), one has that
		$\max\{|\frac{\hat{a}}{a}-1|,|\frac{\hat{b}}{b}-1|\}=O_p(1/\sqrt{t_n^2 (a+b)/n})$ which satisfies the desired rate of convergence whenever $t_n\gg \sqrt{n/\log{n}}$. This can be allowed since one can choose $\sqrt{n}\ll t_n\ll n/s $ when $s=n^{1-\alpha}$ and $\alpha>\frac{1}{2}$. Thus on the event $ \hat{E}(C)$, we use these estimates $\hat{\C},\hat{a},\hat{b}$ to construct a plug-in version of the Maximum Degree Test described on Theorem \ref{thm:sparsesignal_vanilla_upper_gen}\ref{thm:sparsesignal_vanilla_max_upper_gen}. On the event $\hat{E}(C)^c$, we instead perform the optimal test from Theorem \ref{thm:lower_below_logn}\ref{thm:lower_below_logn_upper} (which does not require estimation of communities or $a,b$). Observe that in the regime $a,b \gg \log n$, we can safely assume that $\| \mathbf{\Theta} \|_{\infty} = o(n)$, as otherwise, with probability converging to 1, the graph has at least one vertex with degree $n$ under the alternative--- and thus it is straightforward to detect deviations from the null). Next, note that strongly consistent community detection is possible whenever $a,b\gg \log{n}$ and $\|\mathbf{\Theta}\|_{\infty}=o(n)$ \citep{gao2016community}. Therefore if either $1\leq a,b\ll \log{n}$ or $a,b\gg \log{n}$, for any fixed $C>0$, the test described above is optimal for $s=n^{1-\alpha}$, $\alpha>\frac{1}{2}$  by choosing $\sqrt{n}\ll t_n\ll n/s $. Here we have used Bernstein's Inequality to check that, whenever $t_n\gg \max\{\sqrt{n/(a+b)},\sqrt{n/\log{n}}\}$, with probability converging to $1$, one has occurence of $\hat{E}(C)$ for $a,b\gg \log{n}$ and the occurence of $\hat{E}(C)^c$ for $1\leq a,b\ll \log{n}$ -- and this choice of $t_n$ can be allowed for $s=n^{1-\alpha}$ and $\alpha>\frac{1}{2}$  by choosing $\sqrt{n}\ll t_n\ll n/s $ as before . Moreover this test achieves the sharp constant for $\alpha>3/4$. Finally, for $\alpha\leq \frac{1}{2}$ we can simply use $\sum_{i,j\in \mathcal{T}_n} Y_{i,j}/t_n^2$ to get the correct order of $a,b$ and this allows constructing the optimal test described in Theorem \ref{thm:dense}\ref{thm:dense_upper}.

		
		
		\item [\textbf{Community Sizes} :]  Next consider assumption (I). We recall that the use of this assumption with $\nu=1$ is only relevant for $a,b$ exactly order $n$ (see Remarks \ref{remark:unbalanced_communities1} and \ref{remark:unbalanced_communities3}) and even in that case this assumption is irrelevant (and only comparable size of communities is relevant) when one is only interested in optimal rates.

		\item [\textbf{Finite $k$ and Comparable $a,b$} :]  In terms of assumption (II), our proofs show that the results for $k\rightarrow \infty$ are somewhat easier to obtain. This can be understood by following our proof which essentially reduces to sharp understanding (in both moderate deviation and local CLT sense) of linear combination of two Binomial random variables -- both the number of trials and the success probabilities of which are comparable. When $k\rightarrow \infty$ one of the Binomials dominate and the proof becomes much simpler. A similar comment holds for the case when $b/a\rightarrow 0$ as opposed to the assumption made in (III) of Section \ref{section:notation_assumptions}.
		
		\item [\textbf{Two Sided $A$} :]
		The assumption of one sided signal allows for some simplifications of the already long proofs, without sacrificing the conceptual challenges inherent in this problem. Indeed, as we discuss in Section 4,  our results can be easily extended to the case where $\theta_i=1\pm A$. Specifically, for dense signal regimes, one can consider a chi-square type test statistic using the degrees. In the sparse regime, tests based on $|D_i(\hat{\mathcal{C}},\beta_1^*,\beta_2^*)|$, instead of $D_i(\hat{\mathcal{C}},\beta_1^*,\beta_2^*)$, are expected to be optimal.  In the one-sided signal case, the current upper bound proof crucially uses a special monotonicity property, which is absent in the two-sided problem. 
		This can be dealt with by dividing the proof into cases, where $A$ is just above the detection thresholds vs. when it is far from it (which is an easier case)--- see e.g., Mukherjee, Pillai, Lin (2015) for a related analysis.  The lower bounds (at least for $\alpha>\frac{1}{2}$) follow verbatim. 

		
		\item [\textbf{General $P_n(\ell,\ell')$} :] Finally we comment on the assumption that all the within community connection probabilities in the null SBM is assumed to be same as $a/n$ and across communities as $b/n$. First we note that the results in the dense regime (see Theorem \ref{thm:dense}) remain unchanged even in the general case. While considering $k>2$ communities this helps us in eventually reducing our analysis in to sharp understanding of log scale, exponential scale, as well as local central limit theorem type of asymptotics of arbitrary positive linear combination of \textit{two} Binomial random variables in the moderate deviation regime (see proofs of the results in Section \ref{sec:binom}). Our results on linear combination of two Binomial random variables can be extended to $k$ linear combinations -- albeit with much more tedious computations using similar techniques to the ones presented here. Given similar log scale, exponential scale, as well as local central limit theorem type of asymptotics  for positive linear combination of \textit{two} Binomial random variables-- the rest of our proofs follow almost verbatim (the proof for the HC tests take extra care but essentially are byproducts of very careful use of the results in Section \ref{sec:binom}). 
		%
		Finally, if one is only interested in rate optimal results and not sharp constants, then the case of this general $P_n$ matrix is much easier to analyze. In particular, the lower bounds no longer require a subtle truncation argument and consequently the sharp moderate deviation exponents are no longer required for the lower bound proofs. As for the upper bound, the proof goes by simply using the Maximum Degree Tests with Chernoff \textit{Upper} Bounds on the linear combination of Binomial tails. 

	\end{description}

	\section{Properties of linear combination of Binomial random variables} 
	\label{sec:binom}
	Our analyses depend very heavily on a detailed understanding of deviation properties of linear combinations of binomial random variables. These arise very naturally in our context--- for example, each vertex degree under the null is a sum of two independent Binomial random variables and the optimal tests in Theorem \ref{thm:sparsesignal_optimal} depend on linear combination of two independent Binomial random variables. We establish some relevant results in this section, which are invaluable in the proofs of the main results stated Section \ref{section:main_results}. The proofs of these results are outlined in \cite{ms2019}.
	
	\subsection{Moderate Deviation properties}
	
	Moderate deviation  and local CLT type properties of linear combinations of independent Binomial random variables form a cornerstone of our analysis.  We note that while these results are conceptually straight-forward, the proofs are often involved due to the discrete structure of the random variables involved. We defer the proofs to the Supplement \citep{ms2019}.
	
	To this end, let $X\sim \Bin\left(\gamma_1\frac{n}{k},\frac{a'}{n}\right)\perp Y\sim \Bin\left(\gamma_2\frac{n}{k},\frac{b'}{n}\right)$ with $a'\geq b'\gg \log{n}$, $ 0<c< \liminf \frac{b'}{a'} \leq \limsup \frac{b'}{a'} \leq 1$ for a constant $c$, and constants $\gamma_1,\gamma_2>0$ such that both $\gamma_1n/k$ and $\gamma_2n/k$ are integers. Let

	\begin{align}
	\muone:=\E(X) = \gamma_1\frac{n}{k} \cdot \frac{a'}{n} ,\,\,\,\, \mutwo:=\E(Y) = \gamma_2\frac{n}{k} \cdot \frac{b'}{n}. \nonumber 
	\end{align}
	\begin{align}
	\sigmaone^2:=\mathrm{Var} (X) = \gamma_1\frac{n}{k} \cdot \frac{a'}{n} \cdot \Big(1- \frac{a'}{n} \Big),\,\,\,\, \sigmatwo^2:=\mathrm{Var} (Y) = \gamma_2\frac{n}{k} \cdot \frac{b'}{n} \cdot \Big(1- \frac{b'}{n} \Big).\nonumber 
	\end{align}
	Hereafter for any fixed positive constants $\beta_1$ and $\beta_2$ define
	\be 
	\sigmabeta:=\sqrt{\beta_1^2 \sigmaone^2 + \beta_2^2 \sigmatwo^2},\\
	\mubeta:=\beta_1\muone+\beta_2\mutwo,\\
	d(\beta_1,\beta_2):={\beta_1X+\beta_2Y-\mubeta}.
	\ee
	Also for $s_1,s_2 \leq {n^{1-\alpha}}$ with $1/2<\alpha<1$, and $X'\sim \Bin\left(s_1,\frac{a''}{n}\right) \perp Y'\sim\Bin\left(s_2,\frac{b''}{n}\right)$ with $a''/a'\rightarrow 1$, $b''/b'\rightarrow 1$, let
	\be 
	d'(\beta_1,\beta_2):={\beta_1(X+X')+\beta_2(Y+Y')-\mubeta}.
	\ee

	\subsubsection{Log Scale Asymptotics}
	In this section we study moderate deviations of linear combinations of binomial random variables on the logarithmic scale. Along the way, we shall also study bounds on the probability of such linear combinations belonging to specific subintervals corresponding to moderate deviation regimes. 
	\begin{lemma}\label{lemma:binomial_master} 
		Let $h=h_n$ be such that $c<\liminf\frac{h}{\sigmabeta\sqrt{\log{n}}}\leq\\ \limsup\frac{h}{\sigmabeta\sqrt{\log{n}}}<c'$ for constants $0<c<c'<\infty$ and $C_n\rightarrow C>0$ be a positive sequence.
		\begin{enumerate}
			\item \label{lemma:binomial_equal}Fix any sequence $\{\xi_n\}$ such that $|\xi_n| \ll \log{n}$. Then the following hold for any $\varepsilon>0$ and $n$ sufficiently large (depending on $c,c',\varepsilon,\beta_1,\beta_2$)
			
			\begin{enumerate}
				\item\label{lemma:binomial_equal_pure} $${\sup\limits_{|t|\leq \xi_n}\P\left(d(\beta_1,\beta_2)=h+t\right)}\leq \frac{1}{\sigmabeta}\exp\left(-\frac{h^2}{2\sigmabeta^2}(1-\varepsilon)\right).
				$$
				
				\item\label{lemma:binomial_equal_contam}  $${\sup\limits_{|t|\leq \xi_n}\P\left(d'(\beta_1,\beta_2)=h+t\right)}\leq\frac{1}{\sigmabeta}\exp\left(-\frac{h^2}{2\sigmabeta^2}(1-\varepsilon)\right).
				$$
			\end{enumerate}
			\item\label{lemma:binomial_tail} The following moderate deviation asymptotics hold.
			\begin{enumerate}
				\item\label{lemma:binomial_tail_pure} $$\lim_{n \to \infty} \frac{\log\P\left(d(\beta_1,\beta_2)>C_n\sigmabeta\sqrt{\log{n}}\right)}{\log n}=-\frac{C^2}{2}.$$
				
				\item\label{lemma:binomial_tail_contam} $$\lim_{n \to \infty}\frac{\log \P\left(d'(\beta_1,\beta_2)>C_n\sigmabeta\sqrt{\log{n}}\right)}{\log n}=-\frac{C^2}{2}.$$
			\end{enumerate}
			
		\end{enumerate}
	\end{lemma}

	\subsubsection{Exponential Scale Asymptotics}
	In this section, we first characterize the upper tail of the sum of two independent binomial random variables in the moderate deviation regime on the exponential scale, which requires much more subtle analysis than usual log-scale asymptotics. This result is used in establishing the lower bound for the maximum degree test. Specifically, we will establish the following result. Recall the definition of $d(\beta_1,\beta_2)=\beta_1X+\beta_2Y-\mubeta$, $\mubeta=\E(\beta_1X+\beta_2Y)$ and $\sigmabeta^2=\var(\beta_1X+\beta_2Y)$.	
	\begin{lemma}\label{lemma:binomial_tail_exp_scale}

		Let $b' \gg (\log n)^3$ and $x_n = \sqrt{2 \log n} (1 + o(1))$. In this case, we have, as $n \to \infty$, 
		\begin{align}
		\ & \lim_{n \to \infty} \frac{\P[X+Y-\E(X+Y) > \sqrt{\var(X+Y)} x_n ]}{1 - \Phi(x_n)}\\
		&=\lim_{n \to \infty} \frac{\P[d(1,1) > \sigma_n(1,1) x_n ]}{1 - \Phi(x_n)} \to 1, \nonumber 
		\end{align}
		where $\Phi(\cdot)$ is the cdf of the standard normal distribution. 
	\end{lemma}

	\subsection{A Change of Measure Lemma}
	The next lemma is a simple change of measure argument which is necessary for truncated second moment arguments involved in proving information theoretic lower bounds.
	
	\begin{lemma}\label{lemma:binomial_change_of_measure}
		Let $X\sim \Bin(n_1,p_1)$ and $Y\sim \Bin(n_2,p_2)$ be independent. Then for any positive scalars $\alpha_1,\alpha_2,\beta_1,\beta_2$ and Borel set $B$ of $\mathbb{R}$
		\be 
		\ & 	\E\left(\alpha_1^X\alpha_2^Y\mathbf{1}\left(\beta_1X+\beta_2Y\in B\right)\right)\\
		&=(1-p_1+\alpha_1 p_1)^{n_1}(1-p_2+\alpha_2 p_2)^{n_2}\P(\beta_1X'+\beta_2Y'\in B),
		\ee
		where $X'\sim \Bin(n_1,p_1')$ is independent of $Y'\sim \Bin(n_2,p_2')$ with 
		\be 
		p_1'=\frac{\alpha_1 p_1}{1-p_1+\alpha_1 p_1},\quad p_2'=\frac{\alpha_2 p_2}{1-p_2+\alpha_2 p_2}.
		\ee
	\end{lemma}
	We establish Lemma \ref{lemma:binomial_change_of_measure} in Appendix   \ref{section:technical_lemmas} of \cite{ms2019}.

	\section*{Acknowledgments} The authors thank the Editors, the Associate Editor, and two anonymous referees for numerous helpful comments and suggestions which substantially improved the content and presentation of the paper.
	
	\begin{supplement} 
		\stitle{Supplement to ``Testing Degree Corrections in Stochastic Block Models"}
		\slink[doi]{COMPLETED BY THE TYPESETTER}
		\sdatatype{.pdf}
		\sdescription{The supplementary material contain the proofs of additional technical results.}
	\end{supplement}

	\bibliographystyle{imsart-nameyear}
	\bibliography{biblio_dcsbm}
	
	\appendix
	
	\section{Proofs of main results}\label{section:proofs}
	
	\subsection{Proof of Theorem \ref{thm:dense}} We prove each part of the theorem in separate subsections below. 	\\
	
	\textit{Proof of Theorem \ref{thm:dense} \ref{thm:dense_upper}} 
	In this theorem, since all computations are under the true underlying $\C$ and the Total Degree Test does not depend on it, we drop the notational dependence on $\C=\{\tilde{S}_1,\ldots,\tilde{S}_k\}\in \mathcal{S}_k$ from $\Ptheta^{(\C)}$, $\Etheta^{(\C)}$, and $\mathrm{Var}^{(\C)}_{\mathbf{\Theta},a,b}$.
	
	We will establish the stronger result that the total degree test is powerful whenever there exists a sequence $t_n \to \infty$ such that $s A \sqrt{\frac{a}{n} } \gg t_n$. To this end, we need the following elementary lemma bounding the variance of the total degree. 
	
	\begin{lemma}
		\label{lemma:var_totaldegree}
		For any $\mathbf{\Theta} \in \Xi(s_n,A_n)$ with $\| \mathbf{\Theta} \|_{\infty} \leq 1$, ${\rm{Var}}_{\mathbf{\Theta},a,b}\Big[ \sum_i d_i \Big] \leq 8 a n$. 
	\end{lemma}
	\begin{proof}
		The proof proceeds using the elementary observations ${\rm{Var}}_{\mathbf{\Theta}, a,b}(Y_{ij}) \leq \E_{\mathbf{\Theta},a,b}[Y_{ij}] \leq \frac{4a}{n}$ and ${\rm{cov}}_{\mathbf{\Theta}, a,b} (d_i , d_j) = {\rm{Var}}_{\mathbf{\Theta},a,b}[Y_{ij}] \leq \frac{4a}{n}$. 
	\end{proof}
	
	We are now ready to prove Theorem \ref{thm:dense} \ref{thm:dense_upper}.	We first compute the expectation of the total degree under the null. 
	\begin{align}
	\E_{\mathbf{1},a,b}\left[\sum_{i=1}^{n} d_i \right] =  \sum_{i \neq j } \E[Y_{ij}] = 2\left[\sum_{\ell=1}^{k}{ |\tilde{S}_{\ell}| \choose 2}\frac{a}{n} + \sum_{\ell < \ell'} |\tilde{S}_{\ell} | | \tilde{S}_{\ell'}| \frac{b}{n} \right] := \mu_n. \nonumber 
	\end{align}
	We consider a total degree test which rejects the null for $\sum_i d_i > \mu_n + K_n$ for some sequence $K_n$ to be chosen suitably during the proof. By Chebychev's inequality, we have, 
	\begin{align}
	\P_{\mathbf{1},a,b}\Big[\sum_i d_i > \mu_n + K_n \Big] \leq \frac{{\rm{Var}}_{\mathbf{\Theta},a,b}\Big[ \sum_i d_i \Big] }{K_n^2} \leq \frac{8an}{K_n^2}, \nonumber 
	\end{align}
	where the last inequality follows using Lemma \ref{lemma:var_totaldegree}. Thus the type I error is controlled as soon as $K_n^2 \gg an$. 
	We next turn to the type II error, and note that by monotonicity, it suffices to restrict ourselves to alternatives $\mathbf{\Theta} = (1 + A) \mathbf{1}_{S} + \mathbf{1}_{S^c}$ for some $A \leq 1$. We set $S_{\ell} = \tilde{S}_{\ell} \cap S(\mathbf{\Theta})$ for notational simplicity, denote $s_{\ell} = | S_{\ell} |$. In this case, we have, 
	\begin{align*}
	&\ \E_{\mathbf{\Theta},a,b}\left[\sum_i d_i  \right] \\
	&=  2 \left[\begin{array}{c}(1+A)^2 \frac{a}{n} \left[ \sum_{\ell=1}^{k}{s_{\ell} \choose 2} \right] + (1 + A) \frac{a}{n} \left[ \sum_{\ell=1}^{k} s_{\ell} ( |\tilde{S}_{\ell}| - s_{\ell})   \right] +  (1 + A)^2 \frac{b}{n} \sum_{\ell < \ell'} s_{\ell} s_{\ell'}\nonumber\\
	+ \frac{a}{n} \left[ \sum_{\ell=1}^{k} { |\tilde{S}_{\ell}| - s_{\ell} \choose 2}  \right]  +  (1 + A) \frac{b}{n} \left[ \sum_{\ell \neq \ell'} s_{\ell} \Big( |\tilde{S}_{\ell'}| - s_{\ell'} \Big) \right] + \sum_{\ell < \ell'} (|\tilde{S}_{\ell}| - s_{\ell}) (|\tilde{S}_{\ell'}| - s_{\ell'})   \frac{b}{n} \end{array}\right]. \nonumber 
	\end{align*}
	Therefore, we have, 
	\begin{align}
	&\E_{\mathbf{\Theta},a,b}\Big[\sum_i d_i  \Big] - \mu_n \geq  2\Big[A \frac{a}{n} \Big[ \sum_{\ell} s_\ell(s_\ell -1) \Big] + 2 A  \frac{b}{n} \sum_{\ell< \ell'} s_{\ell} s_{\ell'}  \nonumber\\
	&+ A \frac{a}{n} \Big[ \sum_{\ell} s_{\ell} (|\tilde{S}_{\ell}| - s_{\ell}) \Big] + A \frac{b}{n} \Big[ \sum_{\ell \neq \ell'} s_{\ell} (|\tilde{S}_{\ell'}| - s_{\ell'} )\Big]  \Big] \nonumber \\
	&\geq 2A \frac{b}{n} (n-1) s \geq  A b s.  \nonumber 
	\end{align}
	Therefore, we have, 
	\begin{align}
	\P_{\mathbf{\Theta},a,b}\Big[ \sum_i d_i - \mu_n < K_n \Big] &= \P_{\mathbf{\Theta},a,b} \Big[ \sum_i d_i - \E_{\mathbf{\Theta},a,b}\Big[\sum_i d_i  \Big] < K_n - \Big(\E_{\mathbf{\Theta},a,b}\Big[\sum_i d_i  \Big] - \mu_n \Big) \Big] . \nonumber \\
	&\leq  \P_{\mathbf{\Theta},a,b} \Big[ \sum_i d_i - \E_{\mathbf{\Theta},a,b}\Big[\sum_i d_i  \Big] < K_n -  Abs \Big].  \nonumber 
	\end{align}
	Thus if $K_n <  Abs$, using Chebychev's inequality, we have, 
	\be
	\P_{\mathbf{\Theta},a,b}\Big[ \sum_i d_i - \mu_n < K_n \Big]  &\leq \frac{8an}{\Big(K_n - \frac{1}{2} Abs \Big)^2 } . \label{eq:type2}
	\ee
	The type II error is controlled as soon as the RHS in \eqref{eq:type2} goes to zero as $n \to \infty$. Finally, it remains to choose $K_n$. We set $K_n = \frac{1}{4} Abs $ and note that under the theses of this theorem, both type I and type II errors are controlled asymptotically under this choice. The proof is complete. \qed

	\textit{Proof of Theorem \ref{thm:dense}. \ref{thm:dense_lower}}
	The proof proceeds by the usual argument of analyzing the second moment of the marginal likelihood. To this end, we fix a prior 
	$\pi$ which fixes the community labels $[n] = \tilde{S}_1 \cup \cdots \cup \tilde{S}_{k}$, where $|\tilde{S}_{\ell}| = \frac{n}{k}$ for $1 \leq \ell \leq k$. Note that this choice is allowed by the assumption that $\nu\geq 1$ for $\mathcal{S}^{\nu}_k$.

	The prior $\pi$ selects $s/k$ locations at random from each $\tilde{S}_{\ell}$ to form the set $S(\mathbf{\Theta})$. Given $S(\mathbf{\Theta})$, we set $\theta_i = 1+ A$ for $ i \in S(\Theta)$. 
	Now, given $\mathbf{\Theta}$, the likelihood ratio 
	\begin{align}
	&L_S = \frac{d \P_{\mathbf{\Theta}, a, b} } { d \P_{\mathbf{1},a,b}} \nonumber \\
	&= \prod_{ i < j, \mathcal{C}(i) = \mathcal{C}(j) } (\theta_i \theta_j)^{Y_{ij}} \Big(\frac{1- \theta_i \theta_j \frac{a}{n}}{1- \frac{a}{n}} \Big)^{(1- Y_{ij})} \prod_{ i < j, \mathcal{C}(i) \neq \mathcal{C}(j) } (\theta_i \theta_j)^{Y_{ij}} \Big(\frac{1- \theta_i \theta_j \frac{b}{n}}{1- \frac{b}{n}} \Big)^{(1- Y_{ij})}. \nonumber
	\end{align}
	We define the marginal likelihood $L_{\pi} = \E_{S}[L_S]$, where $\E_S[\cdot]$ denotes the expectation with respect to $S\sim \pi$. It suffices to establish that under the thesis of Theorem \ref{thm:dense_lower}, $L_{\pi} = 1 + o(1)$. To this end, we note that $\E_{\mathbf{1}, a,b} [ L_{\pi}] = \E_{S} [ E_{\mathbf{1},a,b} [ L_{S}]] = 1$ by Fubini's theorem. The result follows once we establish that $\E_{\mathbf{1}, a,b} [ L_{\pi}^2] = 1 + o(1)$ under the assumptions of Theorem \ref{thm:dense_lower}. This will be established in the rest of the proof. We note that $\E_{\mathbf{1},a,b}[ ( L_{\pi})^2 ] = \E_{S_1, S_2} [ \E_{\mathbf{1}, a, b} [ L_{S_1} L_{S_2}]]$, where $S_1, S_2$ are i.i.d draws from the measure $\pi$. 
	Setting $\mathbf{\Theta} : = \mathbf{\Theta}(S_1) = (\theta_1, \cdots, \theta_n)$ and $\overline{\mathbf{\Theta}} := \mathbf{\Theta}(S_2) = (\overline{\theta}_1, \cdots , \overline{\theta}_n)$ to denote the true parameter vectors corresponding to $S_1,S_2$ obtained under iid sampling from $\pi$, we have, 
	\begin{align}
	L_{S_1}L_{S_2} = &\prod_{ i < j : \mathcal{C}(i)= \mathcal{C}(j) } (\theta_i \theta_j \overline{\theta}_i \overline{\theta}_j)^{Y_{ij}} \Big[\Big( \frac{1- \theta_i \theta_j \frac{a}{n} }{1- \frac{a}{n}}\Big) \Big(\frac{1- \overline{\theta}_i \overline{\theta}_j \frac{a}{n}}{1- \frac{a}{n} } \Big) \Big]^{(1- Y_{ij})} \times \nonumber\\
	&\prod_{ i < j : \mathcal{C}(i) \neq \mathcal{C}(j) } (\theta_i \theta_j \overline{\theta}_i \overline{\theta}_j)^{Y_{ij}} \Big[\Big( \frac{1- \theta_i \theta_j \frac{b}{n} }{1- \frac{b}{n}}\Big) \Big(\frac{1- \overline{\theta}_i \overline{\theta}_j \frac{b}{n}}{1- \frac{b}{n} } \Big) \Big]^{(1- Y_{ij})}   \nonumber \\
	&:= \prod_{i < j } T_{ij}. \nonumber 
	\end{align}
	Further, we define 
	\begin{align}
	L_{S_1} L_{S_2} = \prod_{\ell =1}^{k} T_{\ell} \prod_{1\leq \ell < \ell' \leq k} T_{\ell, \ell'}, \nonumber
	\end{align}
	where 
	\begin{align}
	T_{\ell} := \prod_{i<j : \mathcal{C}(i) = \mathcal{C}(j) = \tilde{S}_{\ell} } T_{ij} , \,\,\,\, T_{\ell, \ell'} := \prod_{i < j : \mathcal{C}(i) =\tilde{S}_\ell, \mathcal{C}(j) =\tilde{S}_{\ell'} } T_{ij}. \nonumber  
	\end{align}
	Note that under the null hypothesis $H_0$, $\{T_{\ell}, T_{\ell, \ell'}\}$ are mutually independent and thus to analyze $\E_{\mathbf{1}, a,b}[L_{S_1} L_{S_2}]$, it suffices to study $\E_{\mathbf{1}, a,b}[T_{\ell}]$, $\E_{\mathbf{1},a,b}[T_{\ell, \ell'}]$ separately. We first analyze $T_{\ell}$, $1\leq \ell \leq k$. To this end. we define $Z_{\ell} = | S_1 \cap S_2 \cap \tilde{S}_{\ell} |$, $1\leq \ell \leq k$. Using independence of edges, we have, 
	\begin{align}
	&\E_{\mathbf{1}, a,b}[T_{\ell}] = \prod_{i < j : \mathcal{C}(i) = \mathcal{C}(j) = \tilde{S}_{\ell}} \E_{\mathbf{1},a,b}[ T_{ij}], \nonumber \\
	& \E_{\mathbf{1},a,b}[T_{ij}] = \frac{a}{n} \theta_i \theta_j \bar{\theta}_i \bar{\theta}_j + \Big( 1 - \frac{a}{n} \Big)  \Big[\Big( \frac{1- \theta_i \theta_j \frac{a}{n} }{1- \frac{a}{n}}\Big) \Big(\frac{1- \overline{\theta}_i \overline{\theta}_j \frac{a}{n}}{1- \frac{a}{n} } \Big) \Big].  \nonumber
	\end{align}
	We will encounter the following cases. 
	\begin{enumerate}
		\item $i,j \in S_1 \cap S_2 \cap \tilde{S}_{\ell}$. In this case, 
		\begin{align}
		\E_{\mathbf{1}, a,b}[T_{ij}] = (1 + A)^4 \frac{a}{n} + \Big( 1 - \frac{a}{n} \Big) \Big( \frac{1 - (1+A)^2 \frac{a}{n}}{1- \frac{a}{n}}\Big)^2.
		= 1 + \frac{\frac{a}{n}A^2}{1- \frac{a}{n} } (2+ A)^2. \nonumber 
		\end{align}
		There are ${ Z_{\ell} \choose 2 }$ such terms. 
		
		\item $i, j \in S_1 \cap S_2^c \cap \tilde{S}_{\ell}$ or $i,j \in S_1^c \cap S_2 \cap \tilde{S}_{\ell}$ or $i \in S_1 \cap S_2 \cap \tilde{S}_{\ell}$ while $j \in \tilde{S}_{\ell} \cap S_1^c \cap S_2^c$ or $i \in S_1 \cap S_2^c \cap \tilde{S}_{\ell}$, $j \in S_1^c \cap S_2 \cap \tilde{S}_{\ell}$. In this case, 
		\begin{align}
		\E_{\mathbf{1}, a,b} [T_{ij}] = (1+ A)^2 \frac{a}{n} + \Big( 1 - \frac{a}{n} \Big) \Big( \frac{1 - (1+ A) \frac{a}{n}}{1- \frac{a}{n}} \Big)^2 = 1 + \frac{\frac{a}{n} A^2 }{ 1 - \frac{a}{n}}. \nonumber 
		\end{align}
		There are $2 { \frac{s}{k} - Z_{\ell} \choose 2 } + Z_{\ell} \Big( \frac{n}{k} -\frac{2s}{k} + Z_{\ell} \Big) + \Big( \frac{s}{k} - Z_{\ell} \Big)^2$ many $(i,j)$ pairs which have this contribution. 
		
		\item $i \in S_1 \cap S_2 \cap \tilde{S}_{\ell}$, $j \in S_1 \cap S_2^c \cap \tilde{S}_{\ell}$ or $i \in S_1 \cap S_2 \cap \tilde{S}_{\ell}$, $ j \in S_1^c \cap S_2 \cap \tilde{S}_{\ell}$. We have, 
		\begin{align}
		\E_{\mathbf{1}, a,b}[T_{ij}] = (1+ A)^3 \frac{a}{n} + \Big( \frac{1- (1+A)^2 \frac{a}{n}}{ 1- (1+A) \frac{a}{n}}  \Big) \Big(\frac{1- (1+A) \frac{a}{n} }{1- \frac{a}{n}} \Big) \Big(1 - \frac{a}{n} \Big)
		= 1 + \frac{\frac{a}{n} A^2 }{ 1- \frac{a}{n}}(2 + A). \nonumber 
		\end{align}
		There are $2 Z_\ell \Big( \frac{s}{k} - Z_\ell \Big)$ many terms with this contribution. 
		
		\item For all other $(i,j)$ pairs, it is easy to check that $\E_{\mathbf{1}, a, b}[T_{ij}] =1$.

	\end{enumerate}
	
	We note that under the thesis of the Theorem, $A \to 0$ as $n \to \infty$. Thus we have the upper bound 
	\begin{align}
	&\E_{\mathbf{1}, a,b} [ T_{\ell} ] \leq \Big( 1 + C \frac{\frac{a}{n} A^2 }{1 - \frac{a}{n}} \Big)^ {U_\ell}, \nonumber \\
	&U_\ell = {Z_\ell \choose 2} + 2  {\frac{s}{k } - Z_\ell \choose 2} + Z_\ell \Big( \frac{n}{k} - 2\frac{s}{k} + Z_\ell \Big) + 2 Z_\ell \Big( \frac{s}{k} - Z_\ell \Big) + \Big( \frac{s}{k} - Z_\ell \Big)^2, \nonumber 
	\end{align}
	for some absolute constant $C >0$. Upon simplification, we obtain the bound 
	\be
	\E_{\mathbf{1}, a, b} [T_{\ell} ] \leq \Big( 1 + C \frac{\frac{a}{n}A^2 }{1 - \frac{a}{n}}  \Big)^{\frac{n}{k} Z_{\ell} + \frac{9}{2k^2} s^2 }. \label{eq:T1bound}
	\ee
	Finally, it remains to bound $T_{\ell, \ell'}$. To this end, our analysis proceeds similar to that of $T_{\ell}$ described above, and will thus be sketched briefly. Using independence of edges under $H_0$, we have $\E_{\mathbf{1},a,b}[T_{\ell, \ell'}] = \prod_{i<j: \mathcal{C}(i) = \tilde{S}_{\ell}, \mathcal{C}(j) = \tilde{S}_{\ell'}} \E_{\mathbf{1}, a,b}[T_{ij}]$. We encounter the following cases:
	\begin{enumerate}
		\item $i \in S_1 \cap S_2 \cap \tilde{S}_{\ell}$ and $j \in S_1 \cap S_2 \cap \tilde{S}_{\ell'}$. In this case, we have, 
		\begin{align}
		\E_{\mathbf{1},a,b}[T_{ij}] = (1 + A)^4 \frac{b}{n} + \Big( \frac{1 - (1+A)^2 \frac{b}{n} }{1 - \frac{b}{n} } \Big)^2 \Big( 1 - \frac{b}{n} \Big)= 1 + \frac{\frac{b}{n}A^2}{1- \frac{b}{n} } (2+ A)^2 . \nonumber 
		\end{align}
		There are $Z_{\ell} Z_{\ell'} $ terms with this contribution. 
		
		\item $i \in S_1 \cap S_2 \cap \tilde{S}_{\ell}, j \in S_1 \cap S_2^c \cap \tilde{S}_{\ell'}$ or $ i \in S_1 \cap S_2 \cap \tilde{S}_{\ell}, j \in S_1^c \cap S_2 \cap \tilde{S}_{\ell'}$ and the related pairs $i \in S_1 \cap S_2^c \cap \tilde{S}_{\ell} , j \in S_1 \cap S_2 \cap \tilde{S}_{\ell'}$ and $ i \in S_1^c \cap S_2 \cap \tilde{S}_{\ell}, j \in S_1 \cap S_2 \cap \tilde{S}_{\ell'}$. Each pair contributes
		\begin{align}
		\E_{\mathbf{1}, a,b}[T_{ij}]  = (1+ A)^3 \frac{b}{n} + \Big( \frac{1- (1+A)^2 \frac{b}{n}}{ 1- (1+A) \frac{b}{n}}  \Big) \Big(\frac{1- (1+A) \frac{b}{n} }{1- \frac{b}{n}} \Big) \Big(1 - \frac{b}{n} \Big)
		= 1 + \frac{\frac{b}{n} A^2 }{ 1- \frac{b}{n}}(2 + A). \nonumber 
		\end{align}
		There are $2 Z_{\ell} \Big( \frac{s}{k} - Z_{\ell'} \Big) + 2 Z_{\ell'} \Big( \frac{s}{k} - Z_{\ell} \Big)$ many terms with this contribution. 
		
		\item $i \in S_1 \cap S_2 \cap \tilde{S}_{\ell}, j \in S_1^c \cap S_2^c \cap \tilde{S}_{\ell'}$, or $i \in S_1^c \cap S_2^c \cap \tilde{S}_{\ell}, j \in S_1 \cap S_2 \cap \tilde{S}_{\ell'}$ or $ i \in S_1 \cap S_2^c \cap \tilde{S}_{\ell}, j \in S_1^c \cap S_2 \cap \tilde{S}_{\ell'}$ or $i \in S_1^c \cap S_2 \cap \tilde{S}_{\ell}, j \in S_1 \cap S_2^c \cap \tilde{S}_{\ell'}$. Each term contributes 
		\begin{align}
		\E_{\mathbf{1},a,b} [T_{ij}] = (1+ A)^2 \frac{b}{n} + \Big( 1 - \frac{b}{n} \Big) \Big( \frac{1 - (1+ A) \frac{b}{n}}{1- \frac{b}{n}} \Big)^2 = 1 + \frac{\frac{b}{n} A^2 }{ 1 - \frac{b}{n}}. \nonumber 
		\end{align}
		There are $Z_{\ell} \Big( \frac{n}{k} - \frac{2s}{k} + Z_{\ell'} \Big) + Z_{\ell'} \Big( \frac{n}{k} - \frac{2s}{k} + Z_\ell \Big) + 2 \Big( \frac{s}{k} - Z_{\ell} \Big) \Big( \frac{s}{k} - Z_{\ell'} \Big)$ many terms with this contribution. 
		
		\item Every other pair has $\E_{\mathbf{1},a,b} [T_{ij}]  =1$. 
	\end{enumerate}
	Similar considerations as for $T_{\ell}$ above lead to the upper bound 
	\begin{align}
	&\E_{\mathbf{1},a,b}[T_{\ell, \ell'}] \leq \Big( 1 + C \frac{\frac{b}{n}A^2}{1 - \frac{b}{n}} \Big)^{V_{\ell, \ell'}}, \nonumber \\
	&V_{\ell, \ell'} = Z_{\ell} Z_{\ell'} +  2 Z_{\ell} \Big( \frac{s}{k} - Z_{\ell'} \Big) + 2 Z_{\ell'} \Big( \frac{s}{k} - Z_{\ell} \Big)+ Z_{\ell} \Big( \frac{n}{k} - \frac{2s}{k} + Z_{\ell'} \Big) \nonumber \\
	&+ Z_{\ell'} \Big( \frac{n}{k} - \frac{2s}{k} + Z_{\ell} \Big) + 2 \Big( \frac{s}{k} - Z_{\ell} \Big) \Big( \frac{s}{k} - Z_{\ell'} \Big),  \nonumber 
	\end{align}
	for some absolute constant $C >0$. We note that $Z_{\ell} , Z_{\ell'} \leq \frac{s}{k}$ and thus $V_1 \leq \frac{7s^2}{k^2} + \frac{n}{k} (Z_{\ell} + Z_{\ell'})$. Finally, this yields the following upper bound on $T_{\ell, \ell'}$. 
	\be
	\E_{\mathbf{1},a,b}[T_{\ell, \ell'}] \leq \Big( 1 + C \frac{\frac{b}{n}A^2}{1 - \frac{b}{n}} \Big)^{\frac{7s^2}{k^2} + \frac{n}{k} (Z_\ell + Z_\ell')}. \label{eq:T3bound}
	\ee
	Combining \eqref{eq:T1bound} and \eqref{eq:T3bound}, we obtain, 
	\begin{align}
	&\E_{\mathbf{1},a,b}[L_{S_1} L_{S_2}] \leq \Big( 1 + C \frac{\frac{a}{n}A^2 }{1 - \frac{a}{n}}  \Big)^{\frac{n}{k} \sum_{\ell} Z_{\ell} + \frac{9}{2k} s^2 } \cdot \Big( 1 + C \frac{\frac{b}{n}A^2}{1 - \frac{b}{n}} \Big)^{\frac{7s^2 (k-1)}{2k} + \frac{n (k-1)}{k} \sum_{\ell}Z_{\ell} }. \nonumber \\
	&\leq \exp{\Big[ \frac{9C}{2} s^2 A^2 \Big( \frac{\frac{a}{n}}{1- \frac{a}{n}} + \frac{\frac{b}{n}}{1- \frac{b}{n} } \Big) \Big] } \cdot \exp{\Big[ Cn \Big[\sum_{\ell} Z_{\ell} \Big] A^2 \Big( \frac{\frac{a}{n}}{1- \frac{a}{n}} + \frac{\frac{b}{n}}{1- \frac{b}{n} }   \Big)\Big]}. \nonumber 
	\end{align}
	We note that under $\pi$, $\{Z_\ell: 1 \leq \ell \leq k \}$ are independent Hypergeometric($\frac{n}{k}$, $\frac{s}{k}$, $\frac{s}{k}$ ) random variables. Therefore, they are stochastically bounded by a ${\rm{Bin}} (\frac{s}{k}, \frac{s}{n-s})$ random variable and finally, we have, $\sum_{\ell} Z_\ell \stackrel{\mathrm{stoch}}{\lesssim} {\rm{Bin}}\Big( s, \frac{s}{n-s} \Big)$. This implies that 
	\begin{align}
	&\E_{S_1, S_2}\Big[ \exp{\Big[ Cn \Big[ \sum_{\ell} Z_{\ell} \Big] A^2 \Big( \frac{\frac{a}{n}}{1- \frac{a}{n}} + \frac{\frac{b}{n}}{1- \frac{b}{n} }   \Big)\Big]} \Big] \nonumber \\
	&\leq \Big( 1 - \frac{s}{n} + \frac{s}{n} \exp{\Big[ Cn A^2 \Big( \frac{\frac{a}{n}}{1- \frac{a}{n}} + \frac{\frac{b}{n}}{1- \frac{b}{n} } } \Big) \Big] \Big)^{s}\nonumber \\
	&\leq \exp{\Big[ \frac{s^2}{n} \Big( {\rm{e}}^{ CA^2 n ( \frac{a/n}{1- a/n} + \frac{b/n}{1- b/n} )   } -1 \Big) \Big] }. \label{eq: bound_temp1}
	\end{align}
	Finally, we note that under the assumptions of this theorem, $\alpha \leq \frac{1}{2}$ implies that $A^2 a \to 0$ as $n\to \infty$. Thus using the bound obtained in \eqref{eq: bound_temp1}, we obtain, 
	\begin{align}
	\E_{\mathbf{1},a,b}[L_{\pi}^2] &\leq \exp{\Big[ \frac{9C}{2} s^2 A^2 \Big( \frac{\frac{a}{n}}{1- \frac{a}{n}} + \frac{\frac{b}{n}}{1- \frac{b}{n} } \Big) \Big] } \cdot  \exp{\Big[ \frac{s^2}{n} \Big( {\rm{e}}^{ CA^2 n ( \frac{a/n}{1- a/n} + \frac{b/n}{1- b/n} )   } -1 \Big) \Big] } \nonumber \\
	&\leq \exp{\Big[ C_0 s^2 A^2 \Big( \frac{\frac{a}{n}}{1- \frac{a}{n}} + \frac{\frac{b}{n}}{1- \frac{b}{n} } \Big) \Big] } = 1 + o(1), \label{eq:useful_upper} 
	\end{align}
	where $C_0 >0 $ is some absolute constant, sufficiently large, and the final result follows using the assumptions of this theorem. This completes the proof.	
	
	\subsection{Proof of Theorem \ref{thm:sparse_lower}}
	This section will also have a common proof for both cases $\tau_a = \tau_b =0$ and $\tau_a > \tau_b >0$. The proof proceeds by an analysis of the truncated likelihood ratio under the least favorable prior. To this end, consider the prior $\pi$ which fixes the partition $[n] = \tilde{S}_1 \cup \cdots \cup \tilde{S}_k$, with $|\tilde{S}_{\ell} | = n/k$ for all $1\leq \ell \leq n$. For any $i \in \{1, \cdots , n\}$, let $\mathcal{C}(i) = \tilde{S}_{\ell}$ if $i \in \tilde{S}_{\ell}$. Note that this choice is allowed by the assumption that $\nu\geq 1$ for $\mathcal{S}^{\nu}_k$.
	
	Further, the prior chooses $s/k$ elements (assuming $s$ is divisible by $k$ w.l.o.g.) randomly from each $\tilde{S}_{\ell}$ to form the set $S(\mathbf{\Theta})$. Given $S(\mathbf{\Theta})$, we set $\theta_i = 1+ A$ for $i \in S(\mathbf{\Theta})$ and $\theta_i = 1$ otherwise. In the rest of the proof, we denote the set $S(\Theta)$ as $S$. 
	
	Now, for any such given $\mathbf{\Theta}$, the likelihood ratio 
	\begin{align}
	&L_S = \frac{d \P_{\mathbf{\Theta}, a, b} } { d \P_{\mathbf{1},a,b}} \label{eq:ls} \\
	&= \prod_{ i < j, \mathcal{C}(i) = \mathcal{C}(j) } (\theta_i \theta_j)^{Y_{ij}} \Big(\frac{1- \theta_i \theta_j \frac{a}{n}}{1- \frac{a}{n}} \Big)^{(1- Y_{ij})} \prod_{ i < j, \mathcal{C}(i) \neq \mathcal{C}(j) } (\theta_i \theta_j)^{Y_{ij}} \Big(\frac{1- \theta_i \theta_j \frac{b}{n}}{1- \frac{b}{n}} \Big)^{(1- Y_{ij})}. \nonumber
	\end{align} 
	For $i \in S$, with slight abuse of notation, we define the out-degree to vertices in $S^c \cap \mathcal{C}(i)$ as $d_i(1) = \sum_{j \in \mathcal{C}(i)\cap S^c} Y_{ij}$ while the out-degree to vertices in the opposite block corresponds to $ d_i(2) = \sum_{j \in \mathcal{C}(i)^c \cap S^c} Y_{ij}$. Under $H_0$, we have, 
	\begin{align}
	\E_{\mathbf{1},a,b} \Big[d_i(1) \Big] = \frac{n-s}{k} \cdot \frac{a}{n} ,\,\,\,\, \E_{\mathbf{1},a,b} \Big[d_i(2) \Big]= (k-1)\frac{n-s}{k} \cdot \frac{b}{n}. \nonumber \\
	\var_{\mathbf{1},a,b} [d_i(1)] = \frac{n-s}{k} \cdot \frac{a}{n} \cdot \Big(1- \frac{a}{n} \Big),\,\,\,\, \var_{\mathbf{1},a,b} [d_i(2)] =  (k-1)\frac{n-s}{k} \cdot \frac{b}{n} \cdot \Big( 1 -\frac{b}{n} \Big). \nonumber 
	\end{align}

	Further, we define the constants 
	\begin{align}
	\beta_1^* = \frac{1}{1- \tau_a} \frac{1}{\sqrt{ \frac{1}{k} \frac{\tau_a}{1- \tau_a} + \frac{k-1}{k} \frac{\tau_b}{1- \tau_b}}},\,\,\,\,
	\beta_2^* =  \frac{1}{1- \tau_b} \frac{1}{\sqrt{ \frac{1}{k} \frac{\tau_a}{1- \tau_a} +  \frac{k-1}{k} \frac{\tau_b}{1- \tau_b}}}\nonumber
	\end{align}
	For $i \in S$, consider the ``good" event 
	\begin{align}
	\Gamma_{S,i} = \Big\{ \frac{\beta_1^*(d_i(1) - \E_{\mathbf{1}, a,b} [d_i(1)]) + \beta_2^* (d_i(2) - \E_{\mathbf{1},a,b} [d_i(2)])}{\sqrt{(\beta_1^*)^2 \var_{\mathbf{1},a,b}[d_i(1)] + (\beta_2^*)^2 \var_{\mathbf{1},a,b}[d_i(2)]}} \leq \sqrt{2 \log n}\Big\}. \nonumber 
	\end{align} 
	We set $\Gamma_S = \cap_{i \in S} \Gamma_{S,i}$. 
	We define $\tilde{L}_{\pi} = \E_{S}[L_S \mathbf{1}_{\Gamma_S}]$, where $\E_{S}[\cdot]$ denotes the expectation with respect to $S\sim \pi$. Then it suffices to establish that if $A$ is of the form \eqref{eq:signal_const} with $C < \csparse(\alpha)$, $\E_{\mathbf{1},a,b}[\tilde{L}_{\pi}] = \E_{\mathbf{1},a,b}[(\tilde{L}_{\pi} )^2] = 1 + o(1)$. This will complete the proof of the required lower bound. 
	
	To this end, we note that by Fubini's theorem, 
	$\E_{\mathbf{1},a,b}[\tilde{L}_{\pi}] = \E_{S} [ \E_{\mathbf{1},a,b}[ L_S \mathbf{1}_{\Gamma_S} ]]$. Further, we have, $\E_{\mathbf{1},a,b}[L_S \mathbf{1}_{\Gamma_S}] = 1 - \E_{\mathbf{1},a,b}[L_S \mathbf{1}_{\Gamma_S^c} ]$ and that 
	\begin{align}
	\E_{\mathbf{1},a,b}[L_S \mathbf{1}_{\Gamma_S^c}] \leq \sum_{i \in S} \E_{\mathbf{1},a,b}[ L_S \mathbf{1}_{\Gamma_{S,i}^c}] = \sum_{i \in S} \P\Big[ \frac{\beta_1^* ( X - \frac{n-s}{k}\frac{a}{n}) + \beta_2^* (Y- \frac{(n-s)(k-1)}{k} \frac{b}{n}) }{\sqrt{(\beta_1^* )^2 \frac{n-s}{k} \frac{a}{n} \Big( 1 - \frac{a}{n} \Big) + (\beta_2^*)^2 \frac{(n-s)(k-1)}{k} \frac{b}{n} \Big( 1 - \frac{b}{n} \Big)}} > \sqrt{2\log n} \Big], \nonumber 
	\end{align}
	using Lemma \ref{lemma:binomial_change_of_measure}, with $X \sim {\rm{Bin}}( \frac{n-s}{k}, \frac{a}{n}(1 + A) )$, $Y \sim {\rm{Bin}}(\frac{(n-s)(k-1)}{k}, \frac{b}{n}(1+A))$. We note that 
	\begin{align}
	&\P\Big[ \frac{\beta_1^* ( X - \frac{n-s}{k}\frac{a}{n}) + \beta_2^* (Y- \frac{(n-s)(k-1)}{k} \frac{b}{n}) }{\sqrt{(\beta_1^* )^2 \frac{n-s}{k} \frac{a}{n} \Big( 1 - \frac{a}{n} \Big) + (\beta_2^*)^2 \frac{(n-s)(k-1)}{k} \frac{b}{n} \Big( 1 - \frac{b}{n} \Big)}} > \sqrt{2\log n} \Big] \nonumber \\
	&=\P\Big[ \frac{\beta_1^* (X - \E[X]) + \beta_2^* (Y- \E[Y])}{\sqrt{\var (\beta_1^* X + \beta_2^* Y ) }} > \sqrt{2 \log n} - \sqrt{C \log n} \sqrt{\frac{ \frac{1}{k} \frac{\tau_a}{1- \tau_a} + \frac{k-1}{k}\frac{\tau_b}{1 - \tau_b }}{ \frac{1}{k} \tau_a (1 - \tau_a) + \frac{k-1}{k} \tau_b(1-\tau_b)}} \Big]. \nonumber \\
	&\leq \exp{\Big\{ - \log n \Big(1- \sqrt{\frac{C}{2} \frac{ \frac{1}{k} \frac{\tau_a}{1- \tau_a} + \frac{k-1}{k} \frac{\tau_b}{1-\tau_b} }{ \frac{1}{k} \tau_a(1-\tau_a) + \frac{k-1}{k}  \tau_b (1-\tau_b)}} \Big)^2 (1+o(1)) \Big\} } \nonumber \\
	&= n^{-  \Big(1- \sqrt{\frac{C}{2} \frac{ \frac{1}{k} \frac{\tau_a}{1- \tau_a} + \frac{k-1}{k} \frac{\tau_b}{1-\tau_b} }{ \frac{1}{k} \tau_a(1-\tau_a) + \frac{k-1}{k} \tau_b (1-\tau_b)}} \Big)^2 (1+o(1))}, \nonumber
	\end{align}
	using Lemma \ref{lemma:binomial_master} Part \ref{lemma:binomial_tail_pure}. Thus we finally have,
	\begin{align}
	\E_{\mathbf{1},a,b}[L_S \mathbf{1}_{\Gamma_S^c}] \leq n^{1- \alpha - \Big(1- \sqrt{\frac{C}{2} \frac{ \frac{1}{k} \frac{\tau_a}{1- \tau_a} + \frac{k-1}{k} \frac{\tau_b}{1-\tau_b} }{ \frac{1}{k} \tau_a(1-\tau_a) + \frac{k-1}{k} \tau_b (1-\tau_b)}} \Big)^2 + o(1) } = o(1) \nonumber 
	\end{align}
	if $C < \csparse(\alpha)$. This completes the first part of the proof. 
	
	To study the truncated second moment, we note that $\E_{\mathbf{1},a,b} [ ( \tilde{L}_{\pi})^2 ] = \E_{\mathbf{1}, a,b} [ \E_{S_1,S_2}[ L_{S_1} L_{S_2} \mathbf{1}_{\Gamma_{S_1} \cap \Gamma_{S_2}}] ]$, where $S_1, S_2$ are iid draws from the measure $\pi$. Now, we note that on the event $\Gamma_{S_1} \cap \Gamma_{S_2}$, for $i \in S_1 \cap S_2$, we have, 
	\begin{align}
	&\beta_1^* \Big( \sum_{j \in S_1^c \cap S_2^c \cap \mathcal{C}(i)} Y_{ij}  \Big) + \beta_2^* \Big( \sum_{j \in S_1^c \cap S_2^c \cap \mathcal{C}(i)^c } Y_{ij} \Big) \leq  \nonumber\\
	&\beta_1^* \Big(\frac{n-s}{k} \Big)\frac{a}{n} \Big( 1- \frac{a}{n} \Big) + \beta_2^*  \frac{(n-s)(k-1)}{k}  \frac{b}{n} \Big( 1- \frac{b}{n} \Big) \nonumber\\
	&+ \sqrt{2 \log n} \sqrt{\frac{n-s}{k}  (\beta_1^*)^2 \frac{a}{n} \Big( 1- \frac{a}{n}\Big) + (\beta_2^*)^2 \frac{(n-s)(k-1)}{k}\frac{b}{n} \Big(1 - \frac{b}{n}  \Big) }. \nonumber  
	\end{align}
	For $i \in S_1 \cap S_2$, we denote the above event as $\mathscr{C}_{S_1, S_2, i}$. Finally, we set $\mathscr{C}_{S_1, S_2} = \cap_{ i \in S_1 \cap S_2} \mathscr{C}_{S_1, S_2, i}$. The above discussion implies that $\Gamma_{S_1} \cap \Gamma_{S_2} \subseteq \mathscr{C}_{S_1, S_2}$ and therefore $\E_{\mathbf{1},a,b}[L_{S_1}L_{S_2} \mathbf{1}_{\Gamma_{S_1}\cap \Gamma_{S_2}}] \leq \E_{\mathbf{1},a,b}[L_{S_1}L_{S_2} \mathbf{1}_{\mathscr{C}_{S_1,S_2}}]$. Setting $\mathbf{\Theta} : = \mathbf{\Theta}(S_1) = (\theta_1, \cdots, \theta_n)$ and $\overline{\mathbf{\Theta}} := \mathbf{\Theta}(S_2) = (\overline{\theta}_1, \cdots , \overline{\theta}_n)$ to denote the true parameter vectors corresponding to $S_1,S_2$ obtained under iid sampling from $\pi$, we have, 
	\begin{align}
	L_{S_1}L_{S_2} = &\prod_{ i < j : \mathcal{C}(i)= \mathcal{C}(j) } (\theta_i \theta_j \overline{\theta}_i \overline{\theta}_j)^{Y_{ij}} \Big[\Big( \frac{1- \theta_i \theta_j \frac{a}{n} }{1- \frac{a}{n}}\Big) \Big(\frac{1- \overline{\theta}_i \overline{\theta}_j \frac{a}{n}}{1- \frac{a}{n} } \Big) \Big]^{(1- Y_{ij})} \times \nonumber\\
	&\prod_{ i < j : \mathcal{C}(i) \neq \mathcal{C}(j) } (\theta_i \theta_j \overline{\theta}_i \overline{\theta}_j)^{Y_{ij}} \Big[\Big( \frac{1- \theta_i \theta_j \frac{b}{n} }{1- \frac{b}{n}}\Big) \Big(\frac{1- \overline{\theta}_i \overline{\theta}_j \frac{b}{n}}{1- \frac{b}{n} } \Big) \Big]^{(1- Y_{ij})} . \nonumber\\
	&:= \gamma_0 \prod_{i \in S_1 \cap S_2} \tilde{T}_i , \nonumber 
	\end{align}
	where 
	\begin{align}
	\tilde{T}_i = &\prod_{j \in S_1^c \cap S_2^c \cap \mathcal{C}(i)} (1+A)^{2 Y_{ij}} \Big[\frac{1- (1+A)\frac{a}{n}}{1- \frac{a}{n}} \Big]^{2(1-Y_{ij})} 
	\prod_{j \in S_1^c \cap S_2^c \cap \mathcal{C}(i)^c} (1+ A)^{2 Y_{ij}} \Big[ \frac{1- (1+A)\frac{b}{n}}{1- \frac{b}{n}} \Big]^{2(1-Y_{ij})}. \nonumber 
	\end{align}
	Further, it is easy to see that under $H_0$, $\gamma_0$ and $\prod_{i \in S_1 \cap S_2} \tilde{T}_i$ are independent and therefore 
	\begin{align}
	\E_{\mathbf{1},a,b}[L_{S_1}L_{S_2} \mathbf{1}_{\mathscr{C}_{S_1,S_2}}] = \E_{\mathbf{1},a,b}[\gamma_0] \E_{\mathbf{1},a,b}\Big[\Big( \prod_{i \in S_1 \cap S_2} \tilde{T}_i \Big)\mathbf{1}_{\mathscr{C}_{S_1,S_2}} \Big]. \label{eq:int1}
	\end{align} 
	We will use the following lemma. The proof is similar to the case for $\alpha \leq 1/2$ and will thus be deferred to the end of the section. 
	
	\begin{lemma}
		\label{lemma:gamma0}
		As $n \to \infty$, $\E_{\mathbf{1},a,b}[\gamma_0] = 1+o(1)$, uniformly over all $S_1, S_2 \subset \{1, \cdots, n\}$, with $|S_i| =s= n^{1- \alpha}$, $i =1.2$, such that $|S_i \cap \tilde{S}_{\ell} | = \frac{s}{k}$, $i=1,2$, $1\leq \ell \leq k$. 
	\end{lemma}
	We will complete the lower bound proof assuming Lemma \ref{lemma:gamma0}. Using Lemma \ref{lemma:binomial_change_of_measure}, we have, setting $Z_\ell = |S_1 \cap S_2 \cap \tilde{S}_{\ell}|$,  $1 \leq \ell \leq k$, 
	\begin{align}
	&\E_{\mathbf{1},a,b} [\tilde{T}_i \mathscr{C}_{S_1, S_2,i} ] = \Big( 1 + \frac{\frac{a}{n} A^2 }{ 1 - \frac{a}{n}} \Big)^{\frac{n}{k} - \frac{2s}{k} + Z_{\mathcal{C}(i)}} \Big( 1 + \frac{\frac{b}{n} A^2}{1- \frac{b}{n}} \Big)^{\frac{(k-1)n}{k} -  \frac{2(k-1)s}{k} + \sum_{\ell \neq \mathcal{C}(i)} Z_\ell} \times \nonumber \\ 
	&\P \Big[ \beta_1^* X' + \beta_2^* Y' \leq  \frac{n-s}{k}   \beta_1^*\frac{a}{n} +  \frac{(k-1)(n-s)}{k}\beta_2^* \frac{b}{n}  \nonumber \\
	&+ \sqrt{2\log n} \sqrt{ \frac{n-s}{k} (\beta_1^*)^2  \frac{a}{n} \Big(1 - \frac{a}{n} \Big) +  \frac{(n-s)(k-1)}{k}(\beta_2^*)^2  \frac{b}{n} \Big(1- \frac{b}{n} \Big) } \Big], \nonumber 
	\end{align}
	where $X' \sim {\rm{Bin}} \Big(\frac{n}{k} - \frac{2s}{k} + Z_{\mathcal{C}(i)}, \frac{\frac{a}{n}(1+ A)^2}{1 + \frac{\frac{a}{n} A^2}{1- \frac{a}{n}}} \Big)$ and $Y' \sim {\rm{Bin}} \Big( \frac{n(k-1)}{k} - \frac{2(k-1)s}{k} + \sum_{\ell \neq \mathcal{C}(i)} Z_\ell , \frac{\frac{b}{n}(1+ A)^2}{1 + \frac{\frac{b}{n} A^2}{1- \frac{b}{n}}} \Big)$. Upon using Taylor approximation, we have, 
	\begin{align}
	&\P \Big[ \beta_1^* X' + \beta_2^* Y' \leq  \frac{n-s}{k}   \beta_1^*\frac{a}{n} + \frac{(k-1)(n-s)}{k} \beta_2^* \frac{b}{n}  \nonumber\\
	&+ \sqrt{2\log n} \sqrt{ \frac{n-s}{k} (\beta_1^*)^2  \frac{a}{n} \Big(1 - \frac{a}{n} \Big) + \frac{(n-s)(k-1)}{k}(\beta_2^*)^2  \frac{b}{n} \Big(1- \frac{b}{n} \Big) } \Big] \nonumber\\
	&= \P\Big[ \frac{\beta_1^* (X' - \E[X']) + \beta_2^* (Y' - \E[Y'])}{\sqrt{\var(\beta_1^* X' + \beta_2^* Y')}} < \sqrt{2 \log n} \Big( 1 - 2 C(\tau_a, \tau_b) (1+ o(1)) \Big)  ]. \nonumber 
	\end{align}
	where we set 
	\begin{align}
	C(\tau_a, \tau_b) = \sqrt{\frac{C}{2} \frac{ \frac{1}{k} \frac{\tau_a}{1- \tau_a} + \frac{k-1}{k} \frac{\tau_b}{1- \tau_b} }{ \frac{1}{k} \tau_a (1 - \tau_a) +  \frac{k-1}{k}\tau_b(1-\tau_b)} } . \nonumber 
	\end{align}
	
	We next run into two cases. Consider first the case when $2 C(\tau_a, \tau_b) <1$. In this case, we bound the above probability by $1$. Therefore, we have, 
	\begin{align}
	\E_{\mathbf{1},a,b}[L_{S_1} L_{S_2} \mathbf{1}_{\mathscr{C}_{S_1,S_2}} ] &\leq  \prod_{\ell = 1}^{k} \Big[ \Big( 1 + \frac{\frac{a}{n} A^2 }{1- \frac{a}{n}}\Big)^{\frac{n}{k} -  \frac{2s}{k} + Z_\ell} \Big( 1 + \frac{\frac{b}{n} A^2 }{1- \frac{b}{n}} \Big)^{\frac{(k-1)n}{k} - \frac{(k-1)s}{k} + \sum_{\ell' \neq \ell} Z_{\ell'}  } \Big]^{Z_{\ell} } \nonumber \\
	&\leq \exp{\Big[ n \Big( \sum_{\ell} Z_{\ell}  \Big) A^2 \Big( \frac{1}{k} \frac{ \frac{a}{n} }{1 - \frac{a}{n}} + \frac{k-1}{k}\frac{ \frac{b}{n} }{ 1 - \frac{b}{n}} \Big) \Big]} \nonumber \\
	&\leq \exp{\Big[ \Big( \sum_{\ell} Z_{\ell}  \Big)\Big( \frac{ \frac{1}{k}  \frac{\tau_a}{1 - \tau_a} + \frac{k-1}{k} \frac{\tau_b }{1- \tau_b}}{ \frac{1}{k} \tau_a ( 1 - \tau_a) + \frac{k-1}{k} \tau_b (1 - \tau_b) } \Big) C \log n\Big]}. \nonumber 
	\end{align}
	Now, $\sum_{\ell} Z_{\ell} $ can be dominated stochastically by $U \sim {\rm{Bin}}( s, \frac{s}{n})$. Therefore, 
	\begin{align}
	\E_{\mathbf{1}, a,b} [ \tilde{L}_{\pi}^2]& \leq \E_{S_1, S_2} [ \E_{\mathbf{1},a,b}[ L_{S_1} L_{S_2} \mathbf{1}_{\mathscr{C}_{S_1, S_2}} ] ]. \nonumber  \\
	&\leq \E[ n^{ 2 C^2(\tau_a, \tau_b) U }] = \Big[ \Big( 1 - \frac{s}{n} \Big) + \frac{s}{n} n^{2 C^2(\tau_a, \tau_b) } \Big]^s \nonumber \\ 
	&\leq \exp {\Big[ \frac{s^2}{n} n^{2 C^2(\tau_a, \tau_b) } \Big]} = \exp{[n^{1 - 2 \alpha + 2 C^2(\tau_a, \tau_b)} ]} = 1 + o(1)   \nonumber 
	\end{align}
	if $C < 2 \Big( \frac{ \frac{1}{k}  \tau_a(1-\tau_a) +  \frac{k-1}{k} \tau_b(1-\tau_b)}{ \frac{1}{k} \frac{\tau_a}{1- \tau_a} + \frac{k-1}{k}  \frac{\tau_b}{1-\tau_b}}  \Big) \Big(\alpha - \frac{1}{2} \Big) $. This concludes the proof in this case. 
	
	Next, we deal with the case $2C(\tau_a, \tau_b) >1$. It is easy to see that for $C < \csparse(\alpha)$, this is possible only for $\alpha > 3/4$. In this case, using Lemma \ref{lemma:binomial_master} Part \ref{lemma:binomial_equal_pure},  
	\begin{align}
	&\P\Big[ \frac{\beta_1^* (X' - \E[X']) + \beta_2^* (Y' - \E[Y'])}{\sqrt{\var(\beta_1^* X' + \beta_2^* Y')}} < \sqrt{2 \log n} \Big( 1 - 2 C(\tau_a, \tau_b) \Big) \Big] \nonumber\\
	&\leq \exp{ \Big[ - \log n \Big(  1 - 2 C(\tau_a, \tau_b) \Big)^2  (1+o(1)) \Big]}= n^{- (1- 2C(\tau_a,\tau_b) )^2 (1 +o(1)) }\nonumber 
	\end{align}
	In this case, upon repeating the calculation above, we obtain, 
	\begin{align}
	\E_{\mathbf{1},a,b}[ \tilde{L}_{\pi}^2] &\leq \E_{U} \Big[ \exp{\{ U \log n  f(\tau_a, \tau_b) \} }\Big]\leq \exp{ \{ n^{1- 2 \alpha + f(\tau_a, \tau_b)} \} }, \nonumber \\
	f(\tau_a, \tau_b) &= 2 C(\tau_a, \tau_b) ^2 - (1 - 2 C(\tau_a, \tau_b))^2 . \nonumber 
	\end{align}
	It is easy to see by direct computation that $1- 2\alpha -f(\tau_a, \tau_b) <0$ when $C < \csparse(\alpha)$. The proof will thus be complete, once we establish Lemma \ref{lemma:gamma0}.
	\begin{proof}[Proof of Lemma \ref{lemma:gamma0}:] 
		The proof borrows heavily from that of Theorem \ref{thm:dense}\ref{thm:dense_lower}. Upon using the same notation as in the proof of Theorem \ref{thm:dense}\ref{thm:dense_lower}, we have, $\gamma_0 = \prod_{\{i,j\} \in \mathscr{A}} T_{ij}$, where 
		\begin{align}
		\mathscr{A} = \{ \{i,j \} : i \in S_1 \cap S_2 , j \in S_1^c \cap S_2^c \}^c. \nonumber
		\end{align}
		As in the proof of Theorem \ref{thm:dense}\ref{thm:dense_lower}, we decompose $ \gamma_0 = \prod_{\ell=1}^{k} T_{\ell} \prod_{\ell < \ell'} T_{\ell, \ell'}$, with $T_\ell = \prod_{\{i, j\} \in \mathscr{A}_\ell} T_{ij}$, $1\leq \ell \leq k$, $T_{\ell, \ell'} =\prod_{\{i, j\} \in \mathscr{A}_\ell} T_{ij}$, $1\leq \ell < \ell' \leq k$, where we set
		\begin{align}
		\mathscr{A}_\ell &= \{ \{i,j\} \in \mathscr{A} : i, j \in \tilde{S}_{\ell} \}, \nonumber \\
		\mathscr{A}_{\ell, \ell'} &= \{ \{i,j\} \in \mathscr{A} : i \in \tilde{S}_{\ell}, j \in \tilde{S}_{\ell'} \}, \nonumber
		\end{align}
		We note that under $\P_{\mathbf{1},a,b}[\cdot]$, $\{T_\ell: 1 \leq \ell \leq k\}$, $\{T_{\ell, \ell'}: 1 \leq \ell < \ell' \leq k\}$ are independent--- we will bound each expectation in turn. Further, using independence of the edges, we have, 
		\begin{align}
		\E_{\mathbf{1},a,b}[T_\ell] = \prod_{\{i, j\} \in \mathscr{A}_\ell} \E_{\mathbf{1},a,b}[T_{ij}]. \nonumber
		\end{align}
		
		We will encounter the following cases. 
		\begin{enumerate}
			\item $i,j \in S_1 \cap S_2 \cap \tilde{S}_{\ell}$. In this case, 
			\begin{align}
			\E_{\mathbf{1}, a,b}[T_{ij}] = (1 + A)^4 \frac{a}{n} + \Big( 1 - \frac{a}{n} \Big) \Big( \frac{1 - (1+A)^2 \frac{a}{n}}{1- \frac{a}{n}}\Big)^2.
			= 1 + \frac{\frac{a}{n}A^2}{1- \frac{a}{n} } (2+ A)^2. \nonumber 
			\end{align}
			There are ${ Z_\ell \choose 2 }$ such terms. 
			
			\item $i, j \in S_1 \cap S_2^c \cap \tilde{S}_{\ell}$ or $i,j \in S_1^c \cap S_2 \cap \tilde{S}_{\ell}$ or $i \in S_1 \cap S_2^c \cap \tilde{S}_{\ell}$, $j \in S_1^c \cap S_2 \cap \tilde{S}_{\ell}$. In this case, 
			\begin{align}
			\E_{\mathbf{1}, a,b} [T_{ij}] = (1+ A)^2 \frac{a}{n} + \Big( 1 - \frac{a}{n} \Big) \Big( \frac{1 - (1+ A) \frac{a}{n}}{1- \frac{a}{n}} \Big)^2 = 1 + \frac{\frac{a}{n} A^2 }{ 1 - \frac{a}{n}}. \nonumber 
			\end{align}
			There are $2 { \frac{s}{k} - Z_{\ell} \choose 2 } + \Big( \frac{s}{k} - Z_\ell \Big)^2$ many $(i,j)$ pairs which have this contribution. 
			
			\item $i \in S_1 \cap S_2 \cap \tilde{S}_{\ell}$, $j \in S_1 \cap S_2^c \cap \tilde{S}_{\ell}$ or $i \in S_1 \cap S_2 \cap \tilde{S}_{\ell}$, $ j \in S_1^c \cap S_2 \cap \tilde{S}_{\ell}$. We have, 
			\begin{align}
			\E_{\mathbf{1}, a,b}[T_{ij}] = (1+ A)^3 \frac{a}{n} + \Big( \frac{1- (1+A)^2 \frac{a}{n}}{ 1- (1+A) \frac{a}{n}}  \Big) \Big(\frac{1- (1+A) \frac{a}{n} }{1- \frac{a}{n}} \Big) \Big(1 - \frac{a}{n} \Big)
			= 1 + \frac{\frac{a}{n} A^2 }{ 1- \frac{a}{n}}(2 + A). \nonumber 
			\end{align}
			There are $2 Z_\ell \Big( \frac{s}{k} - Z_\ell \Big)$ many terms with this contribution. 
			
			\item For all other $(i,j)$ pairs, it is easy to check that $\E_{\mathbf{1}, a, b}[T_{ij}] =1$.

		\end{enumerate}
		
		We note that under the thesis of the Theorem, $A \to 0$ as $n \to \infty$. Thus we have the upper bound 
		\begin{align}
		&\E_{\mathbf{1}, a,b} [ T_1 ] \leq \Big( 1 + C \frac{\frac{a}{n} A^2 }{1 - \frac{a}{n}} \Big)^ {U_\ell}, \nonumber \\
		&U_\ell = {Z_\ell \choose 2} + 2  {\frac{s}{k} - Z_\ell \choose 2}  + 2 Z_\ell \Big( \frac{s}{k} - Z_\ell \Big) + \Big( \frac{s}{k} - Z_\ell \Big)^2, \nonumber 
		\end{align}
		for some absolute constant $C >0$. Upon simplification, we obtain the bound 
		\begin{align}
		\E_{\mathbf{1}, a, b} [T_1 ] \leq \Big( 1 + C \frac{\frac{a}{n}A^2 }{1 - \frac{a}{n}}  \Big)^{\frac{9}{2k^2} s^2 }. \nonumber
		\end{align}
		Finally, it remains to bound $\{T_{\ell, \ell'}: 1\leq \ell < \ell' \leq k \}$. We follow the same argument, and encounter the following cases. 
		\begin{enumerate}
			\item $i \in S_1 \cap S_2 \cap \tilde{S}_{\ell}$ and $j \in S_1 \cap S_2 \cap \tilde{S}_{\ell'}$. In this case, we have, 
			\begin{align}
			\E_{\mathbf{1},a,b}[T_{ij}] = (1 + A)^4 \frac{b}{n} + \Big( \frac{1 - (1+A)^2 \frac{b}{n} }{1 - \frac{b}{n} } \Big)^2 \Big( 1 - \frac{b}{n} \Big)= 1 + \frac{\frac{b}{n}A^2}{1- \frac{b}{n} } (2+ A)^2 . \nonumber 
			\end{align}
			There are $Z_{\ell} Z_{\ell'} $ terms with this contribution. 
			
			\item $i \in S_1 \cap S_2 \cap \tilde{S}_{\ell}, j \in S_1 \cap S_2^c \cap \tilde{S}_{\ell'}$ or $ i \in S_1 \cap S_2 \cap \tilde{S}_{\ell}, j \in S_1^c \cap S_2 \cap \tilde{S}_{\ell'}$ and the related pairs $i \in S_1 \cap S_2^c \cap \tilde{S}_{\ell} , j \in S_1 \cap S_2 \cap \tilde{S}_{\ell'}$ and $ i \in S_1^c \cap S_2 \cap \tilde{S}_{\ell}, j \in S_1 \cap S_2 \cap \tilde{S}_{\ell'}$. Each pair contributes
			\begin{align}
			\E_{\mathbf{1}, a,b}[T_{ij}]  = (1+ A)^3 \frac{b}{n} + \Big( \frac{1- (1+A)^2 \frac{b}{n}}{ 1- (1+A) \frac{b}{n}}  \Big) \Big(\frac{1- (1+A) \frac{b}{n} }{1- \frac{b}{n}} \Big) \Big(1 - \frac{b}{n} \Big)
			= 1 + \frac{\frac{b}{n} A^2 }{ 1- \frac{b}{n}}(2 + A). \nonumber 
			\end{align}
			There are $2 Z_{\ell} \Big( \frac{s}{k} - Z_{\ell'} \Big) + 2 Z_{\ell'} \Big( \frac{s}{k} - Z_\ell \Big)$ many terms with this contribution. 
			
			\item  $ i \in S_1 \cap S_2^c \cap \tilde{S}_{\ell}, j \in S_1^c \cap S_2 \cap \tilde{S}_{\ell'}$ or $i \in S_1^c \cap S_2 \cap \tilde{S}_{\ell}, j \in S_1 \cap S_2^c \cap \tilde{S}_{\ell'}$. Each term contributes 
			\begin{align}
			\E_{\mathbf{1},a,b} [T_{ij}] = (1+ A)^2 \frac{b}{n} + \Big( 1 - \frac{b}{n} \Big) \Big( \frac{1 - (1+ A) \frac{b}{n}}{1- \frac{b}{n}} \Big)^2 = 1 + \frac{\frac{b}{n} A^2 }{ 1 - \frac{b}{n}}. \nonumber 
			\end{align}
			There are $ 2 \Big( \frac{s}{k} - Z_\ell \Big) \Big( \frac{s}{k} - Z_{\ell'} \Big)$ many terms with this contribution. 
			
			\item Every other pair has $\E_{\mathbf{1},a,b} [T_{ij}]  =1$. 
		\end{enumerate}
		Similar considerations as for $T_{\ell}$ above lead to the upper bound 
		\begin{align}
		&\E_{\mathbf{1},a,b}[T_{\ell, \ell'}] \leq \Big( 1 + C \frac{\frac{b}{n}A^2}{1 - \frac{b}{n}} \Big)^{V_{\ell, \ell'}}, \nonumber \\
		&V_{\ell, \ell'} = Z_\ell Z_{\ell'} +  2 Z_\ell \Big( \frac{s}{k} - Z_{\ell'} \Big) + 2 Z_\ell \Big( \frac{s}{k} - Z_{\ell'} \Big)+ 2 \Big( \frac{s}{k} - Z_{\ell} \Big) \Big( \frac{s}{k} - Z_{\ell'} \Big),  \nonumber 
		\end{align}
		for some absolute constant $C >0$. We note that $Z_\ell , Z_{\ell'} \leq \frac{s}{k}$ and thus $V_{\ell, \ell'} \leq \frac{7s^2}{k^2}$. Finally, this yields the following upper bound on $T_{\ell, \ell'}$. 
		\begin{align}
		\E_{\mathbf{1},a,b}[T_{\ell, \ell'}] \leq \Big( 1 + C \frac{\frac{b}{n}A^2}{1 - \frac{b}{n}} \Big)^{\frac{7s^2}{k^2}}. \nonumber
		\end{align}
		The rest of the proof can be completed following the same argument as in that of Theorem \ref{thm:dense}\ref{thm:dense_lower}
	\end{proof}

	\subsection{Proof of Theorem \ref{thm:sparsesignal_vanilla_upper_gen}}	
	We prove each part of the theorem in separate subsections below.\\	
	
	\textit{Proof of Theorem \ref{thm:sparsesignal_vanilla_upper_gen} \ref{thm:sparsesignal_vanilla_hc_gen}}
	
	Throughout $\C$ denotes the underlying community assignment and all results are uniform in this $\C$.

	By virtue of centering and scaling of individual $HC(\hcpar;t)$ under the null, we have by union bound and Chebyshev's Inequality,
	\be 
	\ &\Pzero^{(\C)}\left(HC(\hcpar)\geq \sqrt{\log{n}}\right)\\ &\leq \sum_t \Pzero^{(\C)} \left(GHC(\hcpar;t) > \sqrt{\log {n}} \right)\leq \frac{\sqrt{10\log{n}}}{\log{n}}\rightarrow 0 \quad \text{as}\ n\rightarrow \infty.
	\ee
	This controls the Type I error of this test. It remains to control the Type II error. We will establish as usual that the non-centrality parameter under the alternative beats the null and the alternative variances of the statistic. We consider alternatives as follows. Let $\Ptheta$ be such that $\theta_i=1+A$ for $i\in S$ and $\theta_i=1$ otherwise, where $A=\sqrt{\frac{C^*\log{n}}{\sigmazero^2}}$ with $2\rho(\beta_1,\beta_2)\geq C^*>\chc(\beta_1,\beta_2,\alpha)$, $|S|=s=n^{1-\alpha}$, $\alpha \in (1/2,1)$. The case of higher signals can be handled by standard monotonicity arguments and are therefore omitted. Also, let
	$S_{\ell}=\tilde{S}_{\ell}\cap S,\ s_{\ell}=|S_{\ell}|$ for $\ell=1,\ldots,k$. Also let 
	
	$$\genconk:=1/\rho(\beta_1,\beta_2).$$
	$$\rho_k(\beta_1,\beta_2)=\left[ \frac{(\beta_1^2 \tau_a (1 - \tau_a) + \beta_2^2 \frac{k}{k-1}\tau_b(1- \tau_b) ) ( \tau_a (1 - \tau_a) + \frac{k}{k-1}\tau_b (1- \tau_{b}) ) }{(\beta_1 \tau_a + \beta_2 \frac{k}{k-1}\tau_b )^2} \right].$$

	The following Lemma studies the behavior of this statistic under this class of alternatives. 
	\begin{lemma}\label{lemma:power_hcnew} Let $t=\sqrt{2r\log{n}}$ with $r=\min\left\{1,2C^*\genconk\right\}$. Then
		\begin{enumerate} 
			\item[(a)] $\Etheta^{(\C)}\left(GHC(\hcpar;t)\right)\gg \sqrt{\log{n}}.$ \label{lemma:power_hcnew_a}
			\item[(b)]  $\left(\Etheta^{(\C)}\left(GHC(\hcpar;t)\right)\right)^2\gg \mathrm{Var}_{\boldsymbol{\theta},a,b}\left(GHC(\hcpar;t)\right).$ \label{lemma:power_hcnew_b}
		\end{enumerate}
	\end{lemma}
	
	The Type II error of the HC statistic may be controlled immediately using Lemma \ref{lemma:power_hcnew}. This is straightforward--- however, we include a proof for the sake of completeness. For any alternative considered above, we have, using Chebychev's inequality and Lemma \ref{lemma:power_hcnew}, 
	\be
	\Ptheta^{(\C)}[HC(\hcpar) > \sqrt{\log n}] &\geq \Ptheta^{(\C)}[GHC(\hcpar;t) \geq \sqrt{\log n}] 
	\\&\geq 1 - \frac{ \mathrm{Var}^{(\C)}_{\boldsymbol{\theta},a,b}\left(GHC(\hcpar;t)\right)}{(\Etheta^{(\C)}\left(GHC(\hcpar;t)\right) - \sqrt{\log n})^2} \to 1
	\ee
	as $n \to \infty$. This completes the proof, modulo that of Lemma \ref{lemma:power_hcnew}. 
	\qed

	\begin{proof}[Proof of Lemma \ref{lemma:power_hcnew}]
		The proof requires a detailed understanding of the mean and variance of the $HC(\hcpar;t)$ statistics. Due to centering, $HC(\hcpar;t)$ has mean $0$ under the null hypothesis. Our next proposition estimates the variances of the $HC(\hcpar;t)$ statistics under the null and the class of alternatives introduced above. We also lower bound the expectation of the $HC(\hcpar;t)$ statistics under the alternative.
		
		\begin{prop}
			\label{lemma:hcnew_main}
			For $t = \sqrt{2 r \log n}$ with $r > \frac{C^*\genconk}{2}$, we have,
			\be
			\lim_{n\to \infty} \frac{\log \mathrm{Var}^{(\C)}_{\mathbf{\Theta}=\boldsymbol{1},a,b}\left(HC(\hcpar;t)\right) }{\log n} &= 1- r, \label{eq:null_varnew} \\ 
			\lim_{n \to \infty} \frac{\log \Etheta^{(\C)}\left(HC(\hcpar;t)\right)}{\log n}  &\geq 1- \alpha -\frac{1}{2} \left(\sqrt{2r}-\sqrt{C^*\genconk}\right)^2 , \label{eq:alt_expnew}\\
			\lim_{n \to \infty}  \frac{ \log \mathrm{Var}^{(\C)}_{\mathbf{\Theta},a,b}\left(HC(\hcpar;t)\right) }{\log n}&= \max\left\{ 1-\alpha -\frac{1}{2} \left(\sqrt{2r}-\sqrt{C^*\genconk}\right)^2, 1- r \right\}\\ \label{eq:alt_varnew}. 
			\ee
		\end{prop}
		
		We defer the proof of Proposition \ref{lemma:hcnew_main} to Appendix \ref{section:technical_lemmas}. The rest of the proof follows along the lines of the proof of Lemma 6.4 in \cite{mms2016} by noting that $C^*/8(1-\theta)$ in Proposition 6.4 of \cite{mms2016} can be mapped to the constant $C^*\genconk$ in Lemma \ref{lemma:power_hcnew} and Proposition \ref{lemma:hcnew_main}.
	\end{proof}

	\textit{Proof of Theorem \ref{thm:sparsesignal_vanilla_upper_gen} \ref{thm:sparsesignal_vanilla_max_upper_gen}} 
	
	Throughout $\C$ denotes the underlying community assignment and all results are uniform in this $\C$.  For $1\leq l \leq k$, we set $\muzero(\tilde{S}_l,\beta_1,\beta_2)=\beta_1\muonenull (\tilde{S}_l)+\beta_2\mutwonull (\tilde{S}_l)$ and recall the definition of $\sigmazero(\tilde{S}_l,\beta_1,\beta_2)$ from Section \ref{section:tests}.
	First, we control the Type I error of $\phi(\beta_1, \beta_2 ,\delta)$ for any $\delta >0$. Indeed, we have, using Lemma \ref{lemma:binomial_master} Part \ref{lemma:binomial_tail_pure} and an union bound, 
	\be
	\E^{(\C)}_{\mathbf{1},a,b}[T_{d_{\max}}(\C,\beta_1,\beta_2,\delta)] \leq n\cdot n^{- (1 + \delta) + o (1)  } \to 0. \nonumber
	\ee
	Using stochastic monotonicity of the test statistic in $A$, it suffices to analyze the Type II error for $A = \sqrt{\frac{C\log n}{\sigmazero^2}}$ with
	\begin{align*}
	\cmax(\beta_1,\beta_2,\alpha)<C \leq 2 \rho_k(\beta_1, \beta_2).
	\end{align*}
	
	Now, we note that 
	\be
	\ & \P^{(\C)}_{\mathbf{\Theta}, a,b} \Big[ \max_{1\leq l \leq k} \max_{i \in \tilde{S}_l} D_i (\tilde{S}_{l},\beta_1 , \beta_2) > \sqrt{2 (1 + \delta) \log n} \Big] \\
	&\geq \P^{(\C)}_{\mathbf{\Theta},a,b} \Big[ \max_{1\leq l \leq k} \max_{ i \in S(\mathbf{\Theta}) \cap \tilde{S}_l} D_i (\tilde{S}_l,\beta_1 ,\beta_2) > \sqrt{2 (1 + \delta) \log n} \Big]. 
	\ee
	Further, for $i \in S(\mathbf{\Theta}) \cap \tilde{S}_l$, we set $d_i'(1, \tilde{S}_l) = \sum_{j \in \mathcal{C}(i) \cap S(\mathbf{\Theta})^c} Y_{ij}$ and $d_i'(2, \tilde{S}_l) = \sum_{j \in \mathcal{C}(i)^c \cap S(\mathbf{\Theta})^c} Y_{ij}$. We observe that $d_i(1, \tilde{S}_l) \geq d_i'(1, \tilde{S}_l)$ and $d_i(2, \tilde{S}_l) \geq d_i'(2, \tilde{S}_l)$, and thus 
	\be
	\ &\P^{(\C)}_{\mathbf{\Theta},a,b} \Big[  \max_{1\leq l \leq k}\max_{ i \in S(\mathbf{\Theta}) \cap \tilde{S}_l} D_i (\tilde{S}_l,\beta_1 ,\beta_2) > \sqrt{2 (1 + \delta) \log n} \Big] \nonumber \\
	&\geq \P^{(\C)}_{\mathbf{\Theta},a,b} \Big[ \max_{1\leq l \leq k} \max_{i \in \tilde{S}_l \cap S(\mathbf{\Theta})} D_i'(\tilde{S}_l, \beta_1, \beta_2)  > \sqrt{2 (1 + \delta) \log n } \Big], \label{eq:max_upper_int1} 
	\ee
	where we set, for $i \in \tilde{S}_l$, 
	\begin{align}
	D_i'(\tilde{S}_l, \beta_1, \beta_2) = \frac{\beta_1 d_i'(1, \tilde{S}_l) + \beta_2 d_i'(2, \tilde{S}_l) - \ \muzero(\tilde{S}_l,\beta_1,\beta_2) }{\sigmazero(\tilde{S}_l,\beta_1,\beta_2)}. \nonumber
	\end{align}
	Finally, we note that for $j \in \mathcal{C}(i)$, $\E^{(\C)}_{\mathbf{1},a,b}[Y_{ij}] = \frac{a}{n}$, and for $j \in \mathcal{C}(i) \cap S(\mathbf{\Theta})^c$, $\E^{(\C)}_{\mathbf{\Theta},a,b}[Y_{ij}] \geq (1+ A) \frac{a}{n}$. Similarly, for $j \in \mathcal{C}(i)^c$, $\E^{(\C)}_{\mathbf{1},a,b}[Y_{ij}] = \frac{b}{n}$, while for $j \in \mathcal{C}(i)^c \cap S(\mathbf{\Theta})^c$, $\E^{(\C)}_{\mathbf{\Theta},a,b}[Y_{ij}] \geq (1 + A) \frac{b}{n}$. Plugging these into \eqref{eq:max_upper_int1} and simplifying, we get 
	\begin{align}
	&\P^{(\C)}_{\mathbf{\Theta},a,b} \Big[ \max_{1\leq l \leq k} \max_{i \in S(\mathbf{\Theta}) \cap \tilde{S}_l} D_i'(\tilde{S}_l, \beta_1, \beta_2) >\sqrt{2 (1 + \delta) \log n } \Big] \geq \nonumber \\
	&\P^{(\C)}_{\mathbf{\Theta},a,b} \Big[ \cup_{l=1}^{k} \Big\{  \max_{i \in S(\mathbf{\Theta}) \cap \tilde{S}_l} \frac{\beta_1 d_i'(1, \tilde{S}_l) + \beta_2 d_i'(2, \tilde{S}_l) - \E^{(\C)}_{\mathbf{\Theta},a,b}[\beta_1 d_i'(1, \tilde{S}_l) + \beta_2 d_i'(2, \tilde{S}_l)] }{\sqrt{\var_{\mathbf{\Theta},a,b}[\beta_1 d_i'(1) + \beta_2 d_i'(2)]  }} >  C'_l \sqrt{\log n}  \Big\}\Big],\nonumber
	\end{align}
	where $C'_l$ is given as 
	\begin{align}
	C'_l= \sqrt{2 (1 + \delta)} - \frac{\beta_1 \lambda_l \tau_a + \beta_2 (1- \lambda_l) \tau_b }{ \sqrt{( \beta_1^2  \lambda_l \tau_a(1 - \tau_a) + \beta_2^2 (1- \lambda_l) \tau_b (1 - \tau_b)  )(\frac{1}{k} \tau_a (1- \tau_a) +  \frac{k-1}{k}\tau_b(1 -\tau_b)  )}} \sqrt{C}, \nonumber 
	\end{align}
	and we set $|\tilde{S}_l| = \lambda_l n$, $1 \leq l \leq k$.  
	We note that under our assumptions, $C'_l>0$. We note that $\{ D_i'(\mathcal{C}(i), \beta_1, \beta_2): i \in S(\mathbf{\Theta})\}$ are independent and for any fixed $i \in S(\mathbf{\Theta}) \cap \tilde{S}_l$, $d_i'(1, \tilde{S}_l)$ is stochastically larger than a $X_l \sim {\rm{Bin}}\Big( \lambda_l n-s, \frac{a}{n} \Big)$, while $d_i'(2, \tilde{S}_l)$ stochastically dominates $Y_l \sim {\rm{Bin}}\Big(n (1- \lambda_l) -s, \frac{b}{n} \Big)$. Thus we have,  
	\begin{align}
	&\P^{(\C)}_{\mathbf{\Theta},a,b} \Big[ \cap_{1\leq l \leq k} \Big\{  \max_{i \in S(\mathbf{\Theta}) \cap \tilde{S}_l} \frac{\beta_1 d_i'(1) + \beta_2 d_i'(2) - \E^{(\C)}_{\mathbf{\Theta},a,b}[\beta_1 d_i'(1) + \beta_2 d_i'(2)] }{\sqrt{\var_{\mathbf{\Theta},a,b}[\beta_1 d_i'(1) + \beta_2 d_i'(2)]  }} \leq  C'_l \sqrt{\log n} \Big\} \Big] \nonumber\\
	& \leq \prod_{l=1}^{k}\Big( \P\Big[ \frac{\beta_1 X_l + \beta_2 Y_l - \E[\beta_1 X_l + \beta_2 Y_l ]}{\sqrt{\var(\beta_1 X_l + \beta_2 Y_l)}} \leq C'_l \sqrt{\log n} \Big]  \Big)^{|S(\mathbf{\Theta}) \cap \tilde{S}_l|} \nonumber\\
	&=  \prod_{l=1}^{k} \Big(1- \P\Big[ \frac{\beta_1 X_l + \beta_2 Y_l - \E[\beta_1 X_l + \beta_2 Y_l ]}{\sqrt{\var(\beta_1 X_l + \beta_2 Y_l)}} > C'_l \sqrt{\log n} \Big]  \Big)^{|S(\mathbf{\Theta}) \cap \tilde{S}_l|} \nonumber\\
	&\leq \exp\Big[ -\sum_{l=1}^{k} |S(\mathbf{\Theta}) \cap \tilde{S}_l| \P\Big[\frac{\beta_1 X_l + \beta_2 Y_l - \E[\beta_1 X_l + \beta_2 Y_l ]}{\sqrt{\var(\beta_1 X_l + \beta_2 Y_l)}} > C'_l \sqrt{\log n}  \Big] \Big] \to 0\nonumber
	\end{align}
	as $n \to \infty$, by separately analyzing the cases $\tau_a = \tau_b =0$ and $\tau_a, \tau_b >0$, $\nu =1$. This controls the Type II error and completes the proof.\qed
	
	\subsection{Proof of Theorem \ref{thm:sparsesignal_vanilla_upper}} The proof of Theorem \ref{thm:sparsesignal_vanilla_upper}\ref{thm:sparsesignal_vanilla_hc} and \ref{thm:sparsesignal_vanilla_max_upper} follows from Theorem \ref{thm:sparsesignal_vanilla_upper_gen}\ref{thm:sparsesignal_vanilla_hc_gen} and \ref{thm:sparsesignal_vanilla_max_upper_gen}. Hence here we only prove Theorem \ref{thm:sparsesignal_vanilla_upper} \ref{thm:sparsesignal_vanilla_max_lower}.\\
	
	\textit{Proof of Theorem \ref{thm:sparsesignal_vanilla_upper} \ref{thm:sparsesignal_vanilla_max_lower}} 	In this theorem, since all computations are under the true underlying $\C$ and the test based $d_{\max}(1,1)$ does not depend on it, we drop the notational dependence on $\C$ from $\Ptheta^{(\C)}$, $\Etheta^{(\C)}$, and $\mathrm{Var}^{(\C)}_{\mathbf{\Theta},a,b}$.
	
	Consider any alternative $\mathbf{\Theta} \in \Xi(s,A)$ with $A$ as in equation \eqref{eq:alt1} of the statement of the theorem, such that $\theta_i = 1+A$ for $i \in S(\mathbf{\Theta})$ and $\theta_i = 1$ otherwise. Recall that we have set $\muzero = \E_{\mathbf{1},a,b}[d_1]$ and $\sigmazero^2 = \Var_{\mathbf{1},a,b}[d_1]$.  If possible, suppose there exists a consistent sequence of tests based on the max degree with asymptotically zero risk again the alternative sequence under consideration. In this case, there exists a sequence of cut-offs $\{k_n\}$ such that 
	\begin{align}
	\P_{\mathbf{1},a,b}[\max_i d_i < k_n ] \to 1,\,\,\,\, \P_{\mathbf{\Theta},a,b}[\max_i d_i > k_n ] \to 1, \nonumber 
	\end{align}
	as $n \to \infty$. Without loss of generality, we set 
	\begin{align}
	k_n = \muzero + \sigmazero\sqrt{2 \log n}  \Big( 1 - \frac{\log \log n  + \log (4 \pi)}{4 \log n} + \frac{y_n}{2 \log n} \Big). \nonumber 
	\end{align}
	We first observe that for any such sequence $\{k_n\}$, $\P_{\mathbf{1},a,b}[\max_i d_i < k_n] \to 1 $ as $n \to \infty$ implies that $y_n \to \infty$ as $n \to \infty$. To this note, suppose $y_n \leq M$ along any subsequence. Thus along this subsequence, using Theorem \ref{thm:max_deg_null}, we have $\P_{\mathbf{1},a,b}[\max_i d_i < k_n ] \leq \exp[ \textrm{e}^{-M} ] <1$. Thus $y_n \to \infty$ as $n \to \infty$. The rest of the proof establishes that the Type II error does not converge to $0$ as $n \to \infty$ for any such sequence of cutoffs $k_n$, and alternatives $\mathbf{\Theta}$ outlined above. To this end, note that 
	\begin{align}
	\P_{\mathbf{\Theta},a,b} \Big[ \max_i d_i > k_n \Big] &=\P_{\mathbf{\Theta},a,b} \Big[ \max_{i \in S(\mathbf{\Theta})} d_i > k_n \Big] + \P_{\mathbf{\Theta},a,b} \Big[\max_{i \in S(\mathbf{\Theta})^c } d_i > k_n \Big] \nonumber \\
	&\leq n^{1- \alpha} \P_{\mathbf{\Theta},a,b} \Big[ d_{i_0} > k_n \Big] +  \P_{\mathbf{\Theta},a,b} \Big[\max_{i \in S(\mathbf{\Theta})^c } d_i > k_n \Big] \nonumber, 
	\end{align}
	for some $i_0 \in S(\mathbf{\Theta})$. We will establish that each of these terms converge to zero as $n \to \infty$. To this end, first, we note that $y_n \to \infty$ implies that $y_n \geq 0$ eventually. Thus we have,
	\begin{align}
	\P_{\mathbf{\Theta},a,b}[ d_{i_0} > k_n ] &\leq \P_{\mathbf{\Theta},a,b} \Big[ d_{i_0} > \muzero+ \sigmazero \sqrt{2\log n} \Big( 1 - \frac{\log \log n + \log (4\pi)}{4\log n} \Big) \Big] \nonumber \\
	&= \P_{\mathbf{\Theta},a,b} \Big[d_{i_0} > \muzero + \sigmazero \sqrt{2 \log n} ( 1 + o(1) ) \Big]. \nonumber 
	\end{align}
	We note that under $\P_{\mathbf{\Theta},a,b}[\cdot]$, $d_{i_0} \sim U + V$, with $U,V$ independent random variables and $U \sim \textrm{Bin}\Big( \frac{n}{k} , (1 + A) \frac{a}{n} \Big) $ and $V \sim \textrm{Bin}\Big(\frac{n(k-1)}{k}, ( 1 + A) \frac{b}{n} \Big)$ respectively. Thus $\E_{\mathbf{\Theta},a,b} [d_{i_0}] = \frac{n}{k} (1+ A) \Big( \frac{a}{n} + (k-1)\frac{b}{n} \Big)$ and $\Var_{\mathbf{\Theta},a,b}[d_{i_0}] = \frac{n}{k} (1+A) \Big( \frac{a}{n} (1 - (1+A) \frac{a}{n} ) + (k-1)\frac{b}{n} ( 1- (1 +A) \frac{b}{n} ) \Big)$. Thus $\E_{\mathbf{\Theta},a,b} [d_{i_0} ] - \muzero = \frac{n}{k} A \Big( \frac{a}{n} + (k-1) \frac{b}{n} \Big)$. Further, we observe that $A \to 0$ as $n \to \infty$ to conclude that $\var_{\mathbf{\Theta},a,b}[d_{i_0}] / \var_{\mathbf{1},a,b}[d_{i_0}] \to 1$ as $n \to \infty$. Thus we have,
	\begin{align}
	& \P_{\mathbf{\Theta},a,b} \Big[d_{i_0} > \muzero + \sigmazero \sqrt{2 \log n} ( 1 + o(1) ) \Big] \nonumber \\
	& = \P_{\mathbf{\Theta},a,b}\Big[ \frac{d_{i_0} - \E_{\mathbf{\Theta},a,b}[d_{i_0}] }{\sqrt{\var_{\mathbf{\Theta},a,b}[d_{i_0}]} } > \sqrt{2 \log n} (1 + o (1)) - \sqrt{C \log n} \frac{\frac{n}{k} \Big( \frac{a}{n} +(k-1) \frac{b}{n} \Big)}{\sigma_n^2} \Big]. \nonumber \\
	&= \P_{\mathbf{\Theta},a,b}\Big[ \frac{d_{i_0} - \E_{\mathbf{\Theta},a,b}[d_{i_0}] }{\sqrt{\var_{\mathbf{\Theta},a,b}[d_{i_0}]} } > \sqrt{2 \log n} \Big( 1 - \frac{\tau_a + (k-1) \tau_b }{\tau_a (1 - \tau_a) + (k-1) \tau_b (1 - \tau_b) } \sqrt{\frac{C}{2}} \Big) (1 + o(1)) \Big] \nonumber \\
	&= n^{ - \Big( 1 - \frac{\tau_a + (k-1) \tau_b }{ \tau_a (1 - \tau_a) + (k-1) \tau_b (1 - \tau_b) } \sqrt{\frac{C}{2}} \Big)^2 + o(1) }, \nonumber 
	\end{align}
	using Lemma \ref{lemma:binomial_master} Part \ref{lemma:binomial_tail}. As a result, we have, 
	\begin{align}
	n^{1- \alpha} \P_{\mathbf{\Theta},a,b} [d_{i_0} > k_n ] = o(1) , \nonumber 
	\end{align}
	as $ 1 - \alpha - \Big(1 - \frac{\tau_a + \tau_b }{ \tau_a (1 - \tau_a) + \tau_b (1 - \tau_b) } \sqrt{\frac{C}{2}}   \Big)^2 <0 $, in this case. This controls the first term. 
	
	The control of the second term is similar to the control of the Type I error. However, we have to carefully control the contribution due to the contamination edges with the non-null vertices. To this end, note that for $i \in S(\mathbf{\Theta})^c$, $d_i = Z_{1i} + Z_{2i}$, where $Z_{1i} = \sum_{j \neq i , j \in S(\mathbf{\Theta})^c} Y_{ij}$, $Z_{2i} = \sum_{j \neq i, j \in S(\mathbf{\Theta})} Y_{ij}$. Thus we have,
	\begin{align}
	\P_{\mathbf{\Theta},a,b} \Big[ \max_{i \in S(\mathbf{\Theta})^c} d_i > k_n \Big] &\leq \P_{\mathbf{\Theta},a,b} \Big[\max_{i \in S(\mathbf{\Theta})^c} Z_{1i}  + \max_{i \in S(\mathbf{\Theta})^c } Z_{2i} > k_n \Big] \nonumber \\
	&\leq \P_{\mathbf{\Theta},a,b} \Big[ \max_{i \in S(\mathbf{\Theta})^c } Z_{1i} > k_n - k_n'  \Big] + \P_{\mathbf{\Theta},a,b} \Big[ \max_{i \in S(\mathbf{\Theta})^c} Z_{2i} > k_n' \Big], \nonumber 
	\end{align}
	for some sequence $k_n'$ to be chosen appropriately. For each $i \in S(\mathbf{\Theta})^c$, we note that $Z_{2i}$ is stochastically dominated by a $\textrm{Bin}\Big(s, \frac{a}{n}(1 + A) \Big)$ random variable. We choose $k_n' = \frac{\sigma_n \zeta_n'}{\sqrt{2 \log n}}$, for some sequence $\zeta_n' \to \infty$ to be chosen appropriately. We note that $\alpha \in (\frac{1}{2},1)$ implies that $k_n' \gg s \frac{a}{n} (1 + A)$ and thus, by Bernstein's inequality, for $i \in S(\mathbf{\Theta})^c$, 
	\begin{align}
	\P_{\mathbf{\Theta},a,b}[Z_{2i} > k_n' ] &\leq \exp\Big[ - \frac{\Big(k_n' - s\frac{a}{n}(1+A) \Big)^2}{2\Big[ s \frac{a}{n} (1+ A) \Big( 1 - \frac{a}{n}(1+A) \Big) + \frac{1}{3} \Big(k_n'- s \frac{a}{n} (1+A) \Big) \Big]} \Big]
	\leq \exp\Big[-C k_n' \Big] \nonumber 
	\end{align}
	for some $C >0$. 
	Further, $a,b \gg (\log n)^3$ implies that $k_n' \gg \log n$ and thus 
	\begin{align}
	\P_{\mathbf{\Theta},a,b}\Big[ \max_{i \in S(\mathbf{\Theta})^c} Z_{2i} > k_n' \Big] \leq n \exp[ - C k_n' ] = o(1). \nonumber 
	\end{align}
	Finally, we have, for $i \in S(\mathbf{\Theta})^c$, $Z_{1i}$ is stochastically dominated by $M_1 + M_2$, where $M_1, M_2$ are independent random variables with $M_1 \sim \textrm{Bin}\Big( \frac{n}{k}, \frac{a}{n} \Big) $ and $M_2 \sim \textrm{Bin} \Big( \frac{n (k-1)}{k} , \frac{b}{n} \Big)$. This implies 
	\begin{align}
	& \P_{\mathbf{\Theta},a,b } \Big[ \max_{i \in S(\mathbf{\Theta})^c } Z_{1i} > k_n - k_n' \Big] \leq \P_{\mathbf{1},a,b} \Big[ \max_i d_i > k_n - k_n' \Big] \nonumber \\
	&= \P_{\mathbf{1},a,b}\Big[ \max_i d_i > \muzero + \sigmazero \sqrt{2 \log n} \Big( 1 - \frac{\log \log n + \log (4 \pi) }{4 \log n } + \frac{y_n - \zeta_n'}{2 \log n}  \Big) \Big]. \nonumber 
	\end{align}
	We note that for any sequence $y_n \to \infty$, we can choose a sequence $\zeta_n' \to \infty$ sufficiently slow such that $y_n - \zeta_n' \to \infty$. Under such a choice of $\zeta_n'$, $\P_{\mathbf{\Theta},a,b } \Big[ \max_{i \in S(\mathbf{\Theta})^c } Z_{1i} > k_n - k_n' \Big] = o(1) $ as $n \to \infty$. This establishes that no such test can control the Type I and Type II errors simultaneously, and thus completes the proof. \qed

	\subsection{Proof of Theorem \ref{thm:max_deg_null}}In this theorem, since all computations are under the true underlying $\C$ and the test based $d_{\max}(1,1)$ does not depend on it, we drop the notational dependence on $\C$ from $\Ptheta^{(\C)}$, $\Etheta^{(\C)}$, and $\mathrm{Var}^{(\C)}_{\mathbf{\Theta},a,b}$.

	For $y \in \mathbb{R}$, we define $x : = x(n,y)$ as the solution of the equation $\frac{1}{\sqrt{2 \pi}} \frac{n}{x} \exp[ - x^2 /2] = \exp[-y]$. We will establish that 
	\begin{align}
	\P_{\mathbf{1},a,b}\Big[   \frac{\max_i d_i - \muzero }{\sigmazero}  \leq x\Big] \to \exp[ \textrm{e}^{-y} ] \label{eq:max_toprove}
	\end{align}
	as $n \to \infty$, where  as usual we set $\muzero = \E_{\mathbf{1},a,b}[d_1]$ and $\sigmazero^2 = \Var_{\mathbf{1},a,b}[d_1]$. Upon direct computation, we obtain
	\begin{align}
	x(n,y) = \sqrt{2 \log n} \Big( 1 - \frac{\log \log n + \log (4 \pi)}{4 \log n} + \frac{y}{2 \log n} \Big) + o(1) \nonumber 
	\end{align}
	as $n \to \infty$, thus immediately implying the desired result. Thus it remains to establish \eqref{eq:max_toprove}. To this end, we define $Z= \sum_{i} \mathbf{1}(d_i > \muzero + x \sigmazero)$. We claim that as $n \to \infty$, $Z$ converges in distribution to a Poisson($\exp[-y])$ random variable. This immediately implies 
	\begin{align}
	\P_{\mathbf{1},a,b}\Big[ \frac{\max_i d_i - \muzero }{\sigmazero} \leq x\Big] = \P_{\mathbf{1},a,b}[Z=0] \to \exp[\textrm{e}^{-y}] \nonumber 
	\end{align} 
	as $n \to \infty$. This yields \eqref{eq:max_toprove}. 
	
	Finally, it remains to establish the Poisson approximation for $Z$ as $n \to \infty$. To this end, we use the following version of Stein's method for Poisson approximation \citep[Theorem 2.C and Corollary 2.C.4]{barbour1992poisson}. We define a sequence of Bernoulli random variables $\{X_i : i \in I\}$ to be \textit{positively related} if for every $i \in I$, we can construct $\{Y_j^{(i)}: j \neq i \}$, coupled with $\{X_i : i \in I\}$ such that 
	$\{Y_j^{(i)} : j \neq i \}$ is distributed as $\{X_j : j \neq i \} | X_i =1$ and $ \forall j \neq i$, $ Y_{j}^{(i)} \geq X_j$. We set $W= \sum_i X_i$, with $X_i \sim \textrm{Ber}(p_i)$, and $\lambda = \sum_i p_i$. The following theorem \citep[Corollary 2.C.4]{barbour1992poisson} bounds the TV distance between $W$ and a Poisson random variable with mean $\lambda$. 
	\begin{theorem*}\citep[Corollary 2.C.4]{barbour1992poisson}\\
		$$d_{\textrm{TV}}(W, \textrm{Poi}(\lambda) ) \leq \min\Big\{ 1, \frac{1}{\lambda} \Big\} \Big( \Var(W)  -\lambda + 2 \sum_i p_i^2\Big).$$
	\end{theorem*} 
	
	The desired Poisson approximation result follows immediately from the lemma below. 
	\begin{lemma}
		\label{lem:suff_poisson}
		The Bernoulli variables $\{\mathbf{1}(d_i > \muzero + x \sigmazero ): 1 \leq i \leq n \}$ are \textit{positively related}. We set $\lambda =\lambda_n := n \P_{\mathbf{1},a,b}[d_1 > \muzero + x \sigmazero]$. Then we have, if $b \gg (\log n)^3$, as $n \to \infty$, $\lambda \to \exp[-y]$ and $\Var(Z) \to \exp[-y]$
	\end{lemma}
	An application of the Poisson approximation theorem above concludes the proof modulo the proof of Lemma \ref{lem:suff_poisson}. \qed


	\begin{proof}[Proof of Lemma \ref{lem:suff_poisson}]
		First, we establish that $X_i = \mathbf{1}(d_i > \muzero + x \sigmazero)$ are \textit{positively related}. We note that the $X_i$ are increasing functions of independent random variables $\mathbf{Y}$ and thus the $\{X_i: 1 \leq i \leq n \}$ are positively related \citep[Theorem 2.G]{barbour1992poisson}. Next, we check that $\lambda \to \exp[-y]$. We have, 
		\begin{align}
		\lambda = n \P_{\mathbf{1},a,b} [ d_1 > \muzero + x \sigmazero] = n (1 - \Phi(x))(1 + o(1)), \nonumber 
		\end{align}
		where the last equality follows from Lemma \ref{lemma:binomial_tail_exp_scale}. Combining Mills ratio with the definition of $x$ immediately gives us the desired result. Finally, we check the variance condition. 
		\begin{align}
		\Var(Z) = n \P_{\mathbf{1},a,b}[d_1 > \muzero + x \sigmazero] ( 1 - \P_{\mathbf{1},a,b}[d_1 > \mu_n + x \sigma_n]) + {n \choose 2} \cov(X_1, X_2). \nonumber
		\end{align}
		By computations similar to those involved in control of term $T_4$ of \textcolor{black}{Lemma \ref{lemma:bounds_five_sparse_graphs} proved} in Appendix \ref{section:technical_lemmas} in the Supplement \cite{ms2019}, ${n \choose 2} \cov(X_1, X_2)=n^{-1(1+o(1))}$ for any fixed $y\in \mathbb{R}$. This completes the proof. 
	\end{proof}

	\subsection{Proof of Theorem \ref{thm:sparsesignal_optimal}} We claim that the claim of the theorem follows from Theorem \ref{thm:sparsesignal_vanilla_upper_gen} since $\mathrm{Risk}_n(\hat{\C},\Xi(s, A))\rightarrow 0$. To see this note that $\mathrm{Risk}_n(\hat{\C},\Xi(s, A))\rightarrow 0$ implies $\Ptheta^{(\C)}(\mathrm{dist}(\hat{\C},\C)\geq 1)\rightarrow 0$ uniformly over all $\C\in \mathcal{S}_k^{\nu}$ and $\Xi(s, A)$. Since $\mathrm{dist}(\hat{\C},\C)$ is a $\mathbb{N}$-valued random variable, we immediately have $\Ptheta^{(\C)}(\mathrm{dist}(\hat{\C},\C)=0)\rightarrow 1$ uniformly over all $\C\in \mathcal{S}_k^{\nu}$ and 	$\Xi(s, A)$. The proof therefore follows from Theorem \ref{thm:sparsesignal_vanilla_upper_gen}  by working on the event $\mathrm{dist}(\hat{\C},\C)=0$.

	\subsection{Proof of Theorem \ref{thm:lower_below_logn}\ref{thm:lower_below_logn_lower}}
	The proof uses an appropriate truncated second moment argument. We shall also use the standard Chernoff bound which we collect in the following lemma.
	\begin{lemma}
		\label{lem:binom_chernoff}
		For any positive integer $n$ and $p,q \in (0,1)$, we have, 
		\begin{align}
		\P[{\rm{Bin}}(n, p) \geq q n] \leq \exp( - n H_p(q)), \nonumber 
		\end{align}
		where $H_p(q)$ is the binary entropy function. Further, setting $h(x) = x\log x - x + 1$, we have, 
		\begin{align}
		0 \leq H_p(q) - p h(q/p) \leq O\Big(\frac{q^2}{1-q}\Big). \nonumber 
		\end{align}
	\end{lemma}
	Consider a prior which first selects a fixed community assignment $[n]= \tilde{S}_1 \cup \cdots \cup \tilde{S}_k$, with $|\tilde{S}_\ell| = n/k$ for all $1\leq \ell \leq k$. Note that this choice is allowed by the assumption that $\nu\geq 1$ for $\mathcal{S}^{\nu}_k$. Given the community assignment, the prior selects $s/k$ locations from each $\tilde{S}_\ell$. This constructs the signal set $S$. Finally, for $i \notin S$, set $\theta_i = 1$, while for $i \in S$, $\theta_i = 1+A$. This specifies our prior on the parameter space. We denote the prior on this space as $\pi$, and the signal set as $S(\mathbf{\Theta})$. Next, given $\mathbf{\Theta}$, for each $i \in [n]$, we define 
	\begin{align}
	d_i(1) = \sum_{j \in \mathcal{C}(i)\cap S(\mathbf{\Theta})^c} Y_{ij}, \quad
	d_i(2) = \sum_{j \in \mathcal{C}_i^c \cap S(\mathbf{\Theta})^c}Y_{ij}, \nonumber 
	\end{align}
	where $\mathcal{C}(i) = \tilde{S}_\ell$ if $i \in \tilde{S}_{\ell}$. We set $d_i = d_i(1) + d_i(2)$, $1\leq i \leq n$, and define the ``good" event 
	\begin{align}
	\Gamma_{S,i} = \left\{ d_i \leq (1+c) \frac{\log n}{\log \Xi} \right\}, \nonumber 
	\end{align}
	where $\Xi= \log n/ \Big( \frac{a}{k} + \frac{(k-1)b}{k} \Big) $, and $c>0$ is a sufficiently small constant to be chosen later. Recall the likelihood $L_S$ \eqref{eq:ls}, define the truncated likelihood 
	$\tilde{L}_{\pi} = \E_{\pi}[L_S \mathbf{1}_{\cap_{i} \Gamma_{S,i}   }]$. Note that it suffices to prove that whenever $\gamma < \alpha$, $\E_{\mathbf{1},a,b}[\tilde{L}_{\pi}] = 1+ o(1)$ and $\E_{\mathbf{1},a,b}[(\tilde{L}_{\pi})^2] = 1+ o(1)$. We establish each of these in turn. 
	
	First, note that upon Fubini's theorem, 
	\begin{align}
	\E_{\mathbf{1},a,b}[\tilde{L}_{\pi}] = \E_{S}[\E_{\mathbf{1},a,b}[L_{S} \mathbf{1}_{\cap_{i=1}^{n}\Gamma_{S,i} }  ]]. \nonumber 
	\end{align}
	Therefore, to establish $\E_{\mathbf{1},a,b}[\tilde{L}_{\pi}] = 1+o(1)$, it suffices to establish that $\E_{\mathbf{1},a,b}[L_S \mathbf{1}_{\cap_{i=1}^{n}\Gamma_{S,i} }] = 1+ o(1)$ uniformly in $S$. Setting $\P_S$ to denote the alternative measure with $\theta_i =1+ A$ for $i\in S$ and $\theta_i = 1$ otherwise, we have, 
	\begin{align}
	&\E_{\mathbf{1},a,b}[L_S \mathbf{1}_{\cap_{i=1}^{n}\Gamma_{S,i} }] = 1 - \P_{S} [ \cup_{i=1}^{n} \Gamma_{S,i}^{c}] \geq 1 - \sum_{i=1}^{n}\P_{S}[\Gamma_{S,i}^c]\nonumber\\
	&= 1 - (n-s) \P\Big[ {\rm{Bin}}\Big(\frac{n-s}{k}, \frac{a}{n} \Big) + {\rm{Bin}}\Big(\frac{(k-1)(n-s)}{k}, \frac{b}{n}  \Big) > (1+c) \frac{\log n}{\log \Xi}  \Big]  \nonumber \\
	&- s \P\Big[ {\rm{Bin}}\Big(\frac{n-s}{k}, \frac{a}{n}(1+A) \Big) 
	+ {\rm{Bin}}\Big(\frac{(k-1)(n-s)}{k}, \frac{b}{n} (1+A)  \Big) > (1+c) \frac{\log n}{\log \Xi}  \Big].\nonumber
	\end{align}
	The following probability bounds complete the proof of the first part. First, note that 
	\begin{align}
	&\P\Big[{\rm{Bin}}\Big(\frac{n-s}{k}, \frac{a}{n} \Big) + {\rm{Bin}}\Big(\frac{(k-1)(n-s)}{k}, \frac{b}{n}  \Big) > (1+c) \frac{\log n}{\log \Xi}  \Big] \nonumber \\
	&\leq \P\Big[{\rm{Bin}} \Big( n-s, \frac{a}{n} \Big)> (1+c) \frac{\log n}{\log \Xi} \Big] \leq \exp\Big(-n H_{a/n}\Big( (1+c) \frac{\log n}{n \log \Xi} \Big) \Big) = o\Big(\frac{1}{n} \Big), \nonumber 
	\end{align}
	for any $c>0$, where the final inequality follows using Lemma~\ref{lem:binom_chernoff}. Finally, we note that 
	\begin{align}
	& \P\Big[ {\rm{Bin}}\Big(\frac{n-s}{k}, \frac{a}{n}(1+A) \Big) 
	+ {\rm{Bin}}\Big(\frac{(k-1)(n-s)}{k}, \frac{b}{n} (1+A)  \Big) > (1+c) \frac{\log n}{\log \Xi}  \Big] \nonumber \\
	&\leq \P\Big[ {\rm{Bin}} \Big(n-s, \frac{a}{n}(1+A) \Big)> (1+c) \frac{\log n}{\log \Xi} \Big]\leq \exp\Big(-(1+c/3) (1- \gamma) \log n \Big)\nonumber 
	\end{align}
	for $n$ sufficiently large. Thus
	\begin{align}
	\P\Big[ {\rm{Bin}}\Big(\frac{n-s}{k}, \frac{a}{n}(1+A) \Big) 
	+ {\rm{Bin}}\Big(\frac{(k-1)(n-s)}{k}, \frac{b}{n} (1+A)  \Big) > (1+c) \frac{\log n}{\log \Xi}  \Big] = o\Big(\frac{1}{s}\Big), \nonumber 
	\end{align}
	as $\gamma< \alpha$. This verifies the required bound. 
	
	Next, we focus on $\E_{\mathbf{1},a,b}[(\tilde{L}_{\pi})^2]$. Note that $\E_{\mathbf{1},a,b}[(\tilde{L}_{\pi})^2] = \E_{S_1, S_2}[\E_{\mathbf{1},a,b}[L_{S_1} L_{S_2} \mathbf{1}_{\cap_{i}( \Gamma_{S_1,i} \cap \Gamma_{S_2,i}) } ]]$, where $S_1, S_2$ are iid samples from the prior $\pi$. Note that for $i \in S_1 \cap S_2$, 
	\begin{align}
	\sum_{j\in S_1^c \cap S_2^c \cap \mathcal{C}(i)} Y_{ij} + \sum_{j \in S_1^c \cap S_2^c \cap \mathcal{C}(i)} Y_{ij} \leq (1+c) \frac{\log n}{ \log \Xi}, \nonumber 
	\end{align}
	an event that we denote as $\mathscr{C}_{S_1,S_2,i}$. Setting $\mathscr{C}_{S_1, S_2} = \cap_{i \in S_1 \cap S_2}\mathscr{C}_{S_1, S_2,i}$, we note that $\cap_{i} (\Gamma_{S_1,i} \cap \Gamma_{S_2,i}) \subseteq \mathscr{C}_{S_1, S_2}$. Armed with this observation, we seek to derive an upper bound on $E_{\mathbf{1},a,b}[L_{S_1} L_{S_2} \mathbf{1}_{\cap_{i}( \Gamma_{S_1,i} \cap \Gamma_{S_2,i}) }]$. We proceed exactly as in the proof of Theorem~\ref{thm:sparse_lower}. Thus we obtain, as in \eqref{eq:int1}, 
	\begin{align}
	&L_{S_1}L_{S_2} := \gamma_0 \prod_{i \in S_1 \cap S_2} \tilde{T}_i, \nonumber  \\
	&\E_{\mathbf{1},a,b}[L_{S_1} L_{S_2}\mathbf{1}_{\cap_i (\Gamma_{S_1,i} \cap \Gamma_{S_2,i})} ] \leq \E_{\mathbf{1},a,b}[\gamma_0] \E_{\mathbf{1},a,b} \Big[\Big( \prod_{i \in S_1 \cap S_2} \tilde{T}_i \Big) \mathbf{1}_{\mathscr{C}_{S_1, S_2}} \Big]. \nonumber 
	\end{align}
	The following lemma controls $\E_{\mathbf{1},a,b}[\gamma_0]$. The proof is exactly same as that of Lemma \ref{lemma:gamma0}, and is thus omitted. 
	\begin{lemma}
		\label{lemma:gamma0_new}
		As $n \to \infty$, $\E_{\mathbf{1},a,b}[\gamma_0] = 1+o(1)$, uniformly over all $S_1, S_2 \subset \{1, \cdots, n\}$, with $|S_i| =s= n^{1- \alpha}$, $i =1,2$, such that $|S_i \cap \tilde{S}_{\ell} | = \frac{s}{k}$, $i=1,2$, $1\leq \ell \leq k$. 
	\end{lemma}
	We complete the proof assuming Lemma~\ref{lemma:gamma0_new}. Using Lemma~\ref{lemma:binomial_change_of_measure}, setting $Z_{\ell} = |S_1 \cap S_2 \cap \tilde{S}_{\ell}|$, $1\leq \ell \leq k$, we have, 
	\begin{align}
	&\E_{\mathbf{1},a,b} \Big[ \tilde{T}_i \mathbf{1}_{\mathscr{C}_{S_1, S_2,i}} \Big] = \Big( 1 + \frac{\frac{a}{n} A^2 }{ 1 - \frac{a}{n}} \Big)^{\frac{n}{k} - \frac{2s}{k} + Z_{\mathcal{C}(i)}} \Big( 1 + \frac{\frac{b}{n} A^2}{1- \frac{b}{n}} \Big)^{\frac{(k-1)n}{k} -  \frac{2(k-1)s}{k} + \sum_{\ell \neq \mathcal{C}(i)} Z_\ell} \times \nonumber \\
	&\P\Big[ {\rm{Bin}}\Big(\frac{n}{k} - \frac{2s}{k} + Z_{\mathcal{C}(i)}, \frac{\frac{a}{n}(1+A)^2 }{1 + \frac{\frac{a}{n} A^2}{1- \frac{a}{n}}} \Big) + {\rm{Bin}}\Big(\frac{(k-1)n}{k} - \frac{2(k-1)s}{k} +\sum_{\ell \neq \mathcal{C}(i)} Z_{\ell}, \frac{\frac{b}{n}(1+A)^2 }{1 + \frac{\frac{b}{n} A^2}{1- \frac{b}{n}}} \Big) \leq (1+c) \frac{\log n}{\log \Xi}\Big]. \label{eq:int2}
	\end{align}
	To derive the desired bound, consider first the simpler case $\gamma < \frac{1}{2}$. In this case, using \eqref{eq:int2}, we obtain, 
	\begin{align}
	\E_{\mathbf{1},a,b} \Big[\Big( \prod_{i \in S_1 \cap S_2} \tilde{T}_i \Big) \mathbf{1}_{\mathscr{C}_{S_1, S_2}} \Big] \leq \Big( 1 + \frac{\frac{a}{n} A^2}{1- \frac{a}{n}} \Big)^{n |S_1 \cap S_2| }. \nonumber
	\end{align}
	Thus in this regime, 
	\begin{align}
	\E_{\mathbf{1},a,b}[(\tilde{L}_{\pi})^2] \leq \E_{S_1, S_2}\Big[\Big( 1 + \frac{\frac{a}{n} A^2}{1- \frac{a}{n}} \Big)^{n |S_1 \cap S_2| } \Big]. \nonumber 
	\end{align}
	Observe that $|S_1\cap S_2| \sim {\rm{Hyp}}(n,s,s)$ and we seek to bound an increasing function of a Hypergeometric variable. Thus by stochastic dominance, we have, 
	\begin{align}
	\E_{\mathbf{1},a,b}[(\tilde{L}_{\pi})^2] \leq (1+o(1))\E_{S_1, S_2}\Big[\Big( 1 + \frac{\frac{a}{n} A^2}{1- \frac{a}{n}} \Big)^{n U }\Big], \nonumber 
	\end{align}
	where $U \sim {\rm{Bin}}(s, s/n)$. Thus 
	\begin{align}
	\E_{\mathbf{1},a,b}[(\tilde{L}_{\pi})^2] \leq \Big[ 1 - \frac{s}{n} + \frac{s}{n}\Big( 1 + \frac{\frac{a}{n} A^2}{1- \frac{a}{n}} \Big)^n \Big] = 1 + o(1).  \nonumber 
	\end{align}
	as $\gamma <1/2$. Finally, we turn to the case $\gamma > 1/2$. Using $\log(1+x) \leq x$, we observe, 
	\begin{align}
	&\Big( 1 + \frac{\frac{a}{n} A^2 }{ 1 - \frac{a}{n}} \Big)^{\frac{n}{k} - \frac{2s}{k} + Z_{\mathcal{C}(i)}} \Big( 1 + \frac{\frac{b}{n} A^2}{1- \frac{b}{n}} \Big)^{\frac{(k-1)n}{k} -  \frac{2(k-1)s}{k} + \sum_{\ell \neq \mathcal{C}(i)} Z_\ell} \nonumber \\
	&\leq \exp\Big(\frac{a}{k}\cdot \frac{ A^2 }{ 1 - \frac{a}{n}} + \frac{(k-1)b}{k} \cdot\frac{ A^2}{1- \frac{b}{n}}\Big) \nonumber \\
	&= \exp\Big( A^2\Big(\frac{a}{k} + \frac{(k-1)b}{k} \Big)  + o(1) \Big), \nonumber 
	\end{align}
	where the last step follows as $a,b, A \ll \log n$ in this setting. Further, using Taylor expansion, as $\gamma \in (0,1)$, for $n$ sufficiently large, 
	\begin{align}
	\frac{\frac{a}{n}(1+A)^2}{1+ \frac{\frac{a}{n}A^2}{1- \frac{a}{n}} }= \frac{a}{n}(1+A)^2 + o\Big(\frac{1}{n}\Big), \quad \frac{\frac{b}{n}(1+A)^2}{1+ \frac{\frac{b}{n}A^2}{1- \frac{b}{n}} }= \frac{b}{n}(1+A)^2  + o\Big(\frac{1}{n}\Big). \label{eq:int3}
	\end{align}
	Finally, let $M_1\sim {\rm{Bin}}(n_1, p_1)$ and $M_2\sim {\rm{Bin}}(n_2, p_2)$ be independent random variables. By Chernoff inequality, for any $c_0>0$ and $t>0$, we have, 
	\begin{align}
	\P[M_1 + M_2 \leq c_0 (n_1 p_1 + n_2 p_2)] &\leq \exp\Big( c_0t (n_1p_1 + n_2 p_2) + n_1 \log(1- p_1 + p_1 \mathrm{e}^{-t}) + n_2 \log(1-p_2 +p_2 \mathrm{e}^{-t}) \Big) \nonumber \\
	&\leq \exp\Big( c_0t (n_1p_1 + n_2 p_2) + (n_1 p_1 + n_2p_2) (\mathrm{e}^{-t}-1) + O \Big(n_1p_1^2 + n_2p_2^2 \Big) \Big). \nonumber
	\end{align}
	Setting $t= - \log c_0$, we obtain 
	\begin{align}
	\P[M_1 + M_2 \leq c (n_1 p_1 + n_2 p_2)] \leq \exp\Big(- (n_1 p_1 + n_2 p_2) h(c_0) + O \Big(n_1p_1^2 + n_2p_2^2 \Big) \Big), \label{eq:int4}
	\end{align}
	where $h(c_0) = c_0 \log c_0 - c_0 +1$. 
	\noindent 
	Therefore, combining \eqref{eq:int2}, \eqref{eq:int3}, and \eqref{eq:int4}, we have, for $i \in S_1 \cap S_2$
	\begin{align}
	\E_{\mathbf{1},a,b} \Big[ \tilde{T}_i \mathbf{1}_{\mathscr{C}_{S_1, S_2,i}} \Big] \leq \exp\Big(A^2 \Big(\frac{a}{k} + \frac{(k-1)b}{k} \Big) +o(1) \Big).\exp\Big( - (1+A)^2 \Big(\frac{a}{k} + \frac{(k-1)b}{k}\Big) h(\xi_n) + o(1)  \Big),  \nonumber 
	\end{align}
	where we set 
	\begin{align}
	\xi_n = \frac{(1+c) \frac{\log n}{\log \Xi}}{(1+A)^2 \Big(\frac{a}{k} + \frac{(k-1)b}{k} \Big)}. \nonumber 
	\end{align}
	Thus we have, 
	\begin{align}
	\E_{\mathbf{1},a,b}\Big[ \tilde{T}_i \mathbf{1}_{\mathscr{C}_{S_1, S_2,i}} \Big] \leq \exp\Big(  (1+A)^2\Big(\frac{a}{k} + \frac{(k-1)b}{k}\Big) (-\xi_n \log \xi_n + \xi_n) + o(1)  \Big). \nonumber 
	\end{align}
	This implies, 
	\begin{align}
	\E_{\mathbf{1},a,b} \Big[\Big( \prod_{i \in S_1 \cap S_2} \tilde{T}_i \Big) \mathbf{1}_{\mathscr{C}_{S_1, S_2}} \Big] \leq \exp\Big(  |S_1 \cap S_2|\Big( (1+A)^2    \Big(\frac{a}{k} + \frac{(k-1)b}{k}\Big) (-\xi_n \log \xi_n + \xi_n) + o(1) \Big)  \Big). \nonumber 
	\end{align}
	Using stochastic dominance, we can bound the above expectation using that of $U\sim\mathrm{Bin}(s, s/n)$. This implies
	\begin{align}
	\E_{\mathbf{1},a,b}[ (\tilde{L}_{\pi})^2] &\leq (1+o(1)) \E\Big[\exp\Big( U\Big( (1+A)^2    \Big(\frac{a}{k} + \frac{(k-1)b}{k}\Big) (-\xi_n \log \xi_n + \xi_n) + o(1) \Big)  \Big) \Big]\nonumber\\
	&\leq (1+o(1))\exp\Big(\frac{s^2}{n} \exp\Big( \Big( (1+A)^2    \Big(\frac{a}{k} + \frac{(k-1)b}{k}\Big) (-\xi_n \log \xi_n + \xi_n)  \Big)\Big). \nonumber\\
	& = 1+o(1), \nonumber 
	\end{align}
	whenever $c>0$ is chosen small enough such that $2\alpha -1 > (1+c) (2\gamma -1)$. This completes the proof.

	\subsection{Proof of Theorem \ref{thm:lower_below_logn}\ref{thm:lower_below_logn_upper}}
	Throughout this proof, we set $h(u) = u\log u - u +1$. Fix a community partition $\mathcal{C} = \tilde{S}_1 \cup \cdots \cup \tilde{S}_k$ in $\mathcal{S}^{\nu}_k$, with $|\tilde{S}_l | = \lambda_l n$, $1\leq l \leq k$.  Recall that we denote the degree of the $i^{th}$ vertex as $d_i$, and set $d_i' = \sum_{j \in S(\mathbf{\Theta})^c} Y_{ij}$. For $i \in \tilde{S}_{l}$, observe that 
	\begin{align}
	\E_{\mathbf{1},a,b}[d_i] = \Big( |\tilde{S}_{l}| -1 \Big) \frac{a}{n} + (n- |\tilde{S}_{l}|)\frac{b}{n}:= \lambda_l. \nonumber
	\end{align}
	Define $\lambda_{\max} = \max_{1\leq l \leq k} \lambda_i$ and $\lambda_{\min} = \min_{1\leq l \leq k} \lambda_i$, and observe that $\mathcal{C} \in\mathcal{S}^{\nu}_k$ implies that there exist universal constants  $0<C_1<C_2$ such that 
	\begin{align}
	C_1 \Big(\frac{a}{k} + \frac{(k-1)b}{k} \Big)< \lambda_{\min} < \lambda_{\max} < C_2\Big(\frac{a}{k} + \frac{(k-1)b}{k} \Big). \label{eq:same_order}
	\end{align}
	Select $t_0$ such that $\lambda_{\max} h(t_0/\lambda_{\max}) = \log(n \log n)$. Consider a test which rejects the null whenever $\max_{1\leq i \leq n} d_i > t_0$. Then we have, 
	\begin{align}
	\P_{\mathbf{1},a,b}\Big[ \max_{1\leq i \leq n} d_i > t_0 \Big] \leq \sum_{l=1}^{k} |\tilde{S}_{l}| \P_{\mathbf{1},a,b}[d_i > t_0], \label{eq:int5}
	\end{align}
	where $i \in \tilde{S}_{l}$ in the RHS of the inequality above. 
	Observe that for $i \in \tilde{S}_l$, 
	\begin{align}
	\P_{\mathbf{1},a,b}[d_i > t_0] = \P[ X+ Y > t_0], \nonumber  
	\end{align}
	where $X \sim \mathrm{Bin}\Big( |\tilde{S}_l| -1, \frac{a}{n} \Big)$ and $Y\sim \mathrm{Bin}\Big( n- |\tilde{S}_l|, \frac{b}{n}\Big)$ are independent. Using Chernoff bound, for any $c>0$,
	\begin{align}
	\P[X+Y > t_0] &\leq \exp\Big( -c t_0 + \Big(|\tilde{S}_l|-1\Big) \log \Big(1- \frac{a}{n} + \frac{a}{n} e^c \Big)  + (n- |\tilde{S}_l|) \log \Big(1- \frac{b}{n} + \frac{b}{n} e^c\Big)  \Big)\nonumber \\
	&\leq \exp\Big(- ct_0 + \lambda_{\max} (e^c -1) \Big), \nonumber 
	\end{align}
	where the last inequality follows from $\log(1+x) \leq x$. Optimizing over $c>0$, for any $i \in [n]$, we get 
	\begin{align}
	\P_{\mathbf{1},a,b}[d_i > t_0] \leq \exp( - \log( n \log n)) = o\Big(\frac{1}{n} \Big). \nonumber 
	\end{align}
	Plugging this bound back into \eqref{eq:int5}, we note that this test has asymptotically zero Type I error. 
	
	Finally, we turn to the Type II error. We have, 
	\begin{align}
	\P_{\mathbf{\Theta},a,b}\Big[\max_{1\leq i \leq n} d_i > t_0\Big] \geq \P_{\mathbf{\Theta},a,b}\Big[ \max_{i \in S(\mathbf{\Theta})}d_i' > t_0 \Big]= \prod_{l=1}^{k}\prod_{i \in S(\mathbf{\Theta}) \cap \tilde{S}_l} \P_{\mathbf{\Theta},a,b}[d_i' > t_0]. \label{eq:int6}
	\end{align}
	Next, using stochastic dominance, for $i \in S(\mathbf{\Theta}) \cap \tilde{S}_l$, 
	\begin{align}
	\P_{\mathbf{\Theta},a,b}[d_i' > t_0]\geq \P\Big[\mathrm{Bin}\Big(|\tilde{S}_l|-s, (1+A) \frac{a}{n} \Big) + \mathrm{Bin}\Big((n- |\tilde{S}_l|)-s, (1+A) \frac{b}{n} \Big) > t_0 \Big]. \nonumber 
	\end{align}
	To further lower bound this probability, we note that 
	\begin{align}
	& \P\Big[\mathrm{Bin}\Big(|\tilde{S}_l|-s, (1+A) \frac{a}{n} \Big) + \mathrm{Bin}\Big(n- |\tilde{S}_l|-s, (1+A) \frac{b}{n} \Big) > t_0 \Big] \nonumber \\
	& = 1-  \P\Big[\mathrm{Bin}\Big(|\tilde{S}_l|-s, (1+A) \frac{a}{n} \Big) + \mathrm{Bin}\Big(n- |\tilde{S}_l|-s, (1+A) \frac{b}{n} \Big) \leq   t_0   \Big]\nonumber \\
	&\geq 1- \exp\Big(-  \lambda_l (1+A) h\Big(\frac{t_0}{\lambda_l (1+A)}\Big) + o(1)  \Big),  \nonumber
	\end{align}
	where the last inequality follows using \eqref{eq:int4}. 
	From our choice of $t_0$, observe that $t_0 = (1+o(1)) \log n/ \log (\frac{\log n}{\lambda_{\max}})$. By direct computation, we have, 
	\begin{align}
	\P\Big[\mathrm{Bin}\Big(|\tilde{S}_l|-s, (1+A) \frac{a}{n} \Big) + \mathrm{Bin}\Big(n- |\tilde{S}_l|-s, (1+A) \frac{b}{n} \Big) > t_0 \Big] 
	\geq 1 - 
	n^{-(1-\gamma) + o(1)}.  \nonumber 
	\end{align}
	Plugging this back into \eqref{eq:int6}, we obtain 
	\begin{align}
	\P_{\mathbf{\Theta},a,b}\Big[\max_{1\leq i \leq n} d_i > t_0\Big] \geq (1- n^{-(1-\gamma) + o(1)})^s \to 1 \nonumber 
	\end{align}
	as $\gamma > \alpha$. This controls the Type II error, and concludes the proof.

	\section{Proof of Technical Lemmas}\label{section:technical_lemmas}
	\subsection{Proof of Lemma \ref{lemma:binomial_change_of_measure}}
	Let $\mathcal{A}=\{(t_1,t_2)\in \mathbb{N}^2: 0\leq t_1\leq n_1,\ 0\leq t_2\leq n_2,\ \beta_1t_1+\beta_2t_2\in B\}$.
	\be
	\ & \E\left(\alpha_1^X\alpha_2^Y\mathbf{1}\left(\beta_1X+\beta_2Y\in B\right)\right)\\
	&=\sum_{(t_1,t_2)\in \mathcal{A}}{n_1\choose t_1}{n_2 \choose t_2}\alpha_1^{t_1}\alpha_2^{t_2}p_1^{t_1}p_2^{t_2}(1-p_1)^{n_1-t_1}(1-p_2)^{n_2-t_2} \\
	&=\sum_{(t_1,t_2)\in \mathcal{A}}{n_1\choose t_1}{n_2 \choose t_2}(\alpha_1p_1)^{t_1}(\alpha_2p_2)^{t_2}(1-p_1)^{n_1-t_1}(1-p_2)^{n_2-t_2} \\
	&=(1-p_1+\alpha_1 p_1)^{n_1}(1-p_2+\alpha_2 p_2)^{n_2} \sum_{(t_1,t_2)\in \mathcal{A}}\left\{\begin{array}{c}{n_1\choose t_1}{n_2 \choose t_2}\left(\frac{\alpha_1p_1}{1-p_1+\alpha_1p_1}\right)^{t_1}\left(\frac{\alpha_2p_2}{1-p_2+\alpha_2p_2}\right)^{t_2}\\
		\times\left(1-\frac{\alpha_1p_1}{1-p_1+\alpha_1p_1}\right)^{n_1-t_1}\left(1-\frac{\alpha_2p_2}{1-p_2+\alpha_2p_2}\right)^{n_2-t_2}\end{array}\right\}\\
	&=(1-p_1+\alpha_1 p_1)^{n_1}(1-p_2+\alpha_2 p_2)^{n_2}\P(\beta_1X'+\beta_2Y'\in B).
	\ee

	\subsection{Proof of Proposition \ref{lemma:hcnew_main}}
	We analyze each term in turn.
	
	{\textbf{Analysis of $\Etheta^{(\C)}(HC(\hcpar;t))$}}\\
	\be 
	\ &\Etheta^{(\C)}(HC(\hcpar;t))\\
	&=\sum_{l=1}^k\sum_{i\in \tilde{S}_l\cap S}\left\{\Ptheta^{(\C)}\left(D_i(\C,\beta_1,\beta_2)>t\right)-\Pzero^{(\C)}\left(D_i(\C,\beta_1,\beta_2)>t\right)\right\}
	\\&+\sum_{l=1}^k\sum_{i\in \tilde{S}_l\cap S^c}\left\{\Ptheta^{(\C)}\left(D_i(\tilde{S}_l,\beta_1,\beta_2)>t\right)-\Pzero^{(\C)}\left(D_i(\tilde{S}_l,\beta_1,\beta_2)>t\right)\right\}\\
	&\geq \sum_{l=1}^k\sum_{i\in \tilde{S}_l\cap S}\left\{\Ptheta^{(\C)}\left(D_i(\tilde{S}_l,\beta_1,\beta_2)>t\right)-\Pzero^{(\C)}\left(D_i(\tilde{S}_l,\beta_1,\beta_2)>t\right)\right\}.
	\ee
	Fix $l=1,\ldots,k$ and $i\in \tilde{S}_l\cap S$. First by direct application of Lemma \ref{lemma:binomial_master} part \ref{lemma:binomial_tail_pure}, we have
	\be 
	\Pzero^{(\C)}\left(D_i(\tilde{S}_l,\beta_1,\beta_2)>t\right)&\leq n^{-r+o(1)}.
	\ee
	
	Now, for $i\in \tilde{S}_l\cap S$, it easy to see that $$\beta_1d_i(1,\tilde{S}_l)+\beta_2d_i(2,\tilde{S}_l)\stackrel{\mathrm{stoch}}{\gtrsim} \beta_1\mathrm{Bin}\left(|\tilde{S}_l|-1,(1+A)\frac{a}{n}\right)+\beta_2\mathrm{Bin}\left(n-|\tilde{S}_l|,(1+A)\frac{b}{n}\right).$$
	
	Moreover,

	Therefore, for $i\in \tilde{S}_l\cap S$ we have using Lemma \ref{lemma:binomial_master} part \ref{lemma:binomial_tail_pure}
	\be 
	\Ptheta^{(\C)}\left(D_i(\tilde{S}_l,\beta_1,\beta_2)>t\right)&=\exp\left(-\frac{1}{2}\left(\sqrt{2r}-\sqrt{C^*\bar{\rho_l}}\right)^2(1+o(1))\log{n}\right),
	\ee
	where we have with $|\tilde{S}_l|=\lambda_ln$ and $\bar{\lambda}=\frac{1}{k}\sum_{l=1}^k\lambda_l$
	\be 
	\bar{\rho_l}^{-1}=\left[ \frac{(\beta_1^2 \tau_a (1 - \tau_a)\lambda_l + \beta_2^2 (1-\lambda_l)\tau_b(1- \tau_b) ) ( \tau_a (1 - \tau_a)\bar{\lambda} + (1-\bar{\lambda})\tau_b (1- \tau_{b}) ) }{(\beta_1 \tau_a\lambda_l + \beta_2 (1-\lambda_l)\tau_b )^2} \right].
	\ee
	
	Therefore, 
	\be 
	\ & \sum_{l=1}^k\sum_{i\in \tilde{S}_l\cap S}\left\{\Ptheta^{(\C)}\left(D_i(\tilde{S}_l,\beta_1,\beta_2)>t\right)-\Pzero^{(\C)}\left(D_i(\tilde{S}_l,\beta_1,\beta_2)>t\right)\right\}\\
	&=\sum_{l=1}^k|\tilde{S}_l\cap S|\left(\exp\left(-\frac{1}{2}\left(\sqrt{2r}-\sqrt{C^*\bar{\rho_l}}\right)^2(1+o(1))\log{n}\right)-\exp\left(-\frac{2r}{2}(1+o(1))\log{n}\right)\right).
	\ee
	
	This completes the proof of \eqref{eq:alt_expnew} noting that there exists at least one $l=1,\ldots,k$ such that $|\tilde{S}_l\cap S|\geq s/k$.

	{\textbf{Analysis of $\mathrm{Var}^{(\C)}_{\mathbf{\Theta},a,b}\left(HC(\hcpar;t)\right)$}}\\
	We begin by the following basic decomposition of the variance between diagonal and off-diagonal terms.
	\be 
	\ &\mathrm{Var}^{(\C)}_{\mathbf{\Theta},a,b}\left(HC(\hcpar;t)\right):=\sum_{i=1}^5 T_i,\\
	T_1&=\sum_{l=1}^k\sum_{i\in \tilde{S}_l\cap S}\mathrm{Var}^{(\C)}_{\mathbf{\Theta},a,b}\left(\I\left(D_i(\tilde{S}_l,\beta_1,\beta_2)>t\right)\right),\\
	T_2&=\sum_{l=1}^k\sum_{i\in \tilde{S}_l\cap S^c}\mathrm{Var}^{(\C)}_{\mathbf{\Theta},a,b}\left(\I\left(D_i(\tilde{S}_l,\beta_1,\beta_2)>t\right)\right),\\
	T_3&=\sum\limits_{{{\{l,l'\}\in \{1,\ldots,k\},\atop
				i\in \tilde{S}_l\cap S},\atop
			i'\in \tilde{S}_{l'}\cap S},\atop
		\{l,i\}\neq \{l',i'\}}\mathrm{Cov}^{(\C)}_{\mathbf{\Theta},a,b}\left(\I\left(D_i(\tilde{S}_l,\beta_1,\beta_2)>t\right),\I\left(D_{i'}(\tilde{S}_{l'},\beta_1,\beta_2)>t\right)\right),\\
	T_4&=\sum\limits_{{{\{l,l'\}\in \{1,\ldots,k\},\atop
				i\in \tilde{S}_l\cap S^c},\atop
			i'\in \tilde{S}_{l'}\cap S^c},\atop
		\{l,i\}\neq \{l',i'\}}\mathrm{Cov}^{(\C)}_{\mathbf{\Theta},a,b}\left(\I\left(D_i(\tilde{S}_l,\beta_1,\beta_2)>t\right),\I\left(D_{i'}(\tilde{S}_{l'},\beta_1,\beta_2)>t\right)\right),\\
	T_5&=\sum_{{\{l,l'\}\in \{1,\ldots,k\},\atop
			i\in \tilde{S}_l\cap S,}\atop
		i'\in \tilde{S}_{l'}\cap S^c}\mathrm{Cov}^{(\C)}_{\mathbf{\Theta},a,b}\left(\I\left(D_i(\tilde{S}_l,\beta_1,\beta_2)>t\right),\I\left(D_{i'}(\tilde{S}_{l'},\beta_1,\beta_2)>t\right)\right).
	\ee
	
	The control of the various terms above is achieved by the following lemma.
	
	
	\begin{lemma}
		\label{lemma:bounds_five_sparse_graphs}
		Assume the conditions of Theorem \ref{thm:sparsesignal_vanilla_upper_gen}. For $t = \sqrt{2 r \log n}$ with $r > \frac{C^*\genconk}{2}$, we have for any $\epsilon>0$
		\be
		\lim_{n\to \infty} \frac{\log T_1}{\log n} &= 1-\alpha -\frac{1}{2} \left(\sqrt{2r}-\sqrt{C^*\genconk}\right)^2, \quad
		\lim_{n \to \infty} \frac{\log T_2}{\log n}  = 1-r, \\
		\lim_{n \to \infty} \frac{\log T_3}{\log n} &\leq 1-2\alpha-\left(\sqrt{2r}-\sqrt{C^*\genconk} \right)^2(1-\varepsilon),  \quad
		\lim_{n\to \infty} \frac{\log T_4}{\log n} \leq 1-2r(1-\varepsilon), \\
		\lim_{n \to \infty}& \frac{\log T_5}{\log n} \leq 1-\alpha-\left(\frac{1}{2}\left(\sqrt{2r}-\sqrt{C^*\genconk} \right)^2+ r\right)(1-\varepsilon). 
		\ee
	\end{lemma}
	We note that Lemma \ref{lemma:bounds_five_sparse_graphs} indeed verifies \eqref{eq:alt_varnew} by taking $\epsilon>0$ small enough. Also this verifies \eqref{eq:null_varnew} by taking $C^*=0$ in Lemma \ref{lemma:bounds_five_sparse_graphs}.
	\qed
	
	\begin{proof}[Proof of Lemma \ref{lemma:bounds_five_sparse_graphs}]

		We will constantly use the following simple identity.
		\begin{lemma}\label{lemma_simple_identity}
			For real numbers $p,x_1,x_2,y_1,y_2$ 
			$$px_1x_2+(1-p)y_1y_2-(px_1+(1-p)y_1)(px_2+(1-p)y_2)=p(1-p)(x_1-y_1)(x_2-y_2).$$	
		\end{lemma}

		\paragraph*{\textbf{Control of $T_1$}}
		
		Note that for $i\in \tilde{S}_l\cap S$, $l=1,\ldots,k$, we have
		\be 
		\mathrm{Var}^{(\C)}_{\mathbf{\Theta},a,b}\left(\I\left(D_i(\tilde{S}_l,\beta_1,\beta_2)>t\right)\right)&=a^{(s_l)}(t)\times (1-a^{(s_l)}(t)).
		\ee
		where 
		\be 
		a^{(s_l)}(t)=\Ptheta^{(\C)}\left(\beta_1d_i(1,\tilde{S}_l)+\beta_2d_i(2,\tilde{S}_l)>t\sigmabetanulll+\beta_1 \muonenull(\tilde{S}_l)+\beta_2\mutwonull(\tilde{S}_l)\right).
		\ee
		Note that under $\Ptheta^{(\C)}$ we have $\beta_1d_i(1,\tilde{S}_l)+\beta_2d_i(2,\tilde{S}_l)\sim \beta_1(Z_{11}^{(l)}+Z_{12}^{(l)})+\beta_2(Z_{21}^{(l)}+Z_{22}^{(l)})$ with the following independent components
		\be 
		Z_{11}^{(l)}\sim \Bin\left(s_l-1,(1+A)^2\frac{a}{n}\right),\\
		Z_{12}^{(l)}\sim \Bin\left(n\lambda_l-s_l,(1+A)\frac{a}{n}\right),\\
		Z_{21}^{(l)}\sim \Bin\left(n(1-\lambda_l)-\sum_{j\neq l}s_j,(1+A)\frac{b}{n}\right),\\
		Z_{22}^{(l)}\sim \Bin\left(\sum_{j\neq l}s_j,(1+A)^2\frac{b}{n}\right).
		\ee
		To operationalize Lemma \ref{lemma:binomial_master}, note that
		\be 
		\ & \frac{t\sigmabetanulll+\beta_1\muonenull(\tilde{S}_l)+\beta_2\mutwonull(\tilde{S}_l)-\beta_1(n\lambda_l-s_l)(1+A)\frac{a}{n}-\beta_2(n(1-\lambda_l)-\sum_{j\neq l}s_j)(1+A)\frac{b}{n}}{\sqrt{\beta_1^2(n\lambda_l-s_l)(1+A)\frac{a}{n}\left(1-(1+A)\frac{a}{n}\right)+\beta_2^2(n(1-\lambda_l)-\sum_{j\neq l}s_j)(1+A)\frac{b}{n}\left(1-(1+A)\frac{b}{n}\right)}}\\
		&=\left(\sqrt{2r}-\sqrt{C^*\bar{\rho_l}}\right)(1+o(1))\sqrt{\log{n}},
		\ee
		where recall that
		\be 
		\bar{\rho_l}^{-1}=\left[ \frac{(\beta_1^2 \tau_a (1 - \tau_a)\lambda_l + \beta_2^2 (1-\lambda_l)\tau_b(1- \tau_b) ) ( \tau_a (1 - \tau_a)\bar{\lambda} + (1-\bar{\lambda})\tau_b (1- \tau_{b}) ) }{(\beta_1 \tau_a\lambda_l + \beta_2 (1-\lambda_l)\tau_b )^2} \right].
		\ee
		Therefore, applying  Lemma \ref{lemma:binomial_master} Part (b, ii), 
		\be 
		\ & \sum_{l=1}^k\sum_{i\in \tilde{S}_l\cap S}\mathrm{Var}^{(\C)}_{\mathbf{\Theta},a,b}\left(\I\left(D_i(\tilde{S}_l,\beta_1,\beta_2)>t\right)\right)\\
		&=\sum_{l=1}^k|\tilde{S}_l\cap S|\exp\left(-\frac{1}{2}\left(\sqrt{2r}-\sqrt{C^*\bar{\rho_l}}\right)^2(1+o(1))\log{n}\right)(1+o(1)).
		\ee
		This completes the verification of the control of Term $T_1$ by noting that there exists at least one $l=1,\ldots,k$ such that $|\tilde{S}_l\cap S|\geq s/k$ and considering the asymptotics of $\tau_a,\tau_b$ under the assumptions of Theorem \ref{thm:sparsesignal_vanilla_upper_gen}.
		
		\paragraph*{\textbf{Control of $T_2$}}
		Note that for $i\in \tilde{S}_l\cap S^c$, $l=1,\ldots,k$, we have
		\be 
		\mathrm{Var}^{(\C)}_{\mathbf{\Theta},a,b}\left(\I\left(D_i(\tilde{S}_l,\beta_1,\beta_2)>t\right)\right)&=a^{(s_l^c)}(t)\times (1-a^{(s_l^c)}(t)).
		\ee
		where 
		\be 
		a^{(s_l^c)}(t)=\Ptheta^{(\C)}\left(\beta_1d_i(1,\tilde{S}_l)+\beta_2d_i(2,\tilde{S}_l)>t\sigmabetanulll+\beta_1 \muonenull(\tilde{S}_l)+\beta_2\mutwonull(\tilde{S}_l)\right).
		\ee
		and under $\Ptheta^{(\C)}$ we have $\beta_1d_i(1,\tilde{S}_l)+\beta_2d_i(2,\tilde{S}_l)\sim \beta_1(Z_{11}^{(l)}+Z_{12}^{(l)})+\beta_2(Z_{21}^{(l)}+Z_{22}^{(l)})$ with the following independent components
		\be 
		Z_{11}^{(l)}\sim \Bin\left(s_l,(1+A)\frac{a}{n}\right),\\
		Z_{12}^{(l)}\sim \Bin\left(n\lambda_l-s_l-1,\frac{a}{n}\right),\\
		Z_{21}^{(l)}\sim \Bin\left(n(1-\lambda_l)-\sum_{j\neq l}s_j,\frac{b}{n}\right),\\
		Z_{22}^{(l)}\sim \Bin\left(\sum_{j\neq l}s_j,(1+A)\frac{b}{n}\right).
		\ee
		To operationalize Lemma \ref{lemma:binomial_master}, note that
		\be 
		\ & \frac{t\sigmabetanulll+\beta_1\muonenull(\tilde{S}_l)+\beta_2\mutwonull(\tilde{S}_l)-\beta_1(n\lambda_l-s_l)\frac{a}{n}-\beta_2(n(1-\lambda_l)-\sum_{j\neq l}s_j)\frac{b}{n}}{\sqrt{\beta_1^2(n\lambda_l-s_l)\frac{a}{n}\left(1-\frac{a}{n}\right)+\beta_2^2(n(1-\lambda_l)-\sum_{j\neq l}s_j)\frac{b}{n}\left(1-\frac{b}{n}\right)}}\\
		&=\sqrt{2r}(1+o(1))\sqrt{\log{n}}.
		\ee
		Therefore, applying  Lemma \ref{lemma:binomial_master} Part (b, ii), 
		\be 
		\ & \sum_{l=1}^k\sum_{i\in \tilde{S}_l\cap S^c}\mathrm{Var}^{(\C)}_{\mathbf{\Theta},a,b}\left(\I\left(D_i(\tilde{S}_l,\beta_1,\beta_2)>t\right)\right)\\
		&=\sum_{l=1}^k|\tilde{S}_l\cap S^c|n^{-r+o(1)}(1+o(1)).
		\ee
		
		This completes the verification of the control of Term $T_1$ by noting that there exists at least one $l=1,\ldots,k$ such that $|\tilde{S}_l\cap S|\geq s/k$.
		\paragraph*{\textbf{Control of $T_3$}}
		Similar to \cite{mms2016}, we begin by noting the following simple identities followed by local central limit theorem type estimates. However, in order to deal with arbitrary linear combinations, one need more detailed computations and uniform control of local central limit type estimates.
		
		First consider $\mathrm{Cov}^{(\C)}_{\mathbf{\Theta},a,b}\left(\I\left(D_i(\tilde{S}_l,\beta_1,\beta_2)>t\right),\I\left(D_{i'}(\tilde{S}_l,\beta_1,\beta_2)>t\right)\right)$ for $i\neq i'\in \tilde{S}_l\cap S$ for some $l=1,\ldots,k$. Then
		\be 
		\mathrm{Cov}^{(\C)}_{\mathbf{\Theta},a,b}\left(\I\left(D_i(\tilde{S}_l,\beta_1,\beta_2)>t\right),\I\left(D_{i'}(\tilde{S}_l,\beta_1,\beta_2)>t\right)\right)&=b^{(s_l)}(t)-(a^{(s_l)}(t))^2,
		\ee
		where
		\be 
		b^{(s_l)}(t)&=\Ptheta^{(\C)}(D_i(\hcparl)>t,D_{i'}(\hcparl)>t),\\
		a^{(s_l)}(t)&=\Ptheta^{(\C)}\left(D_i(\hcparl)>t\right).
		\ee
		Now note that
		\be 
		b^{(s_l)}(t)&=\frac{a}{n}(1+A)^2 (a^{(s_l)'}(t) )^2+\left(1-\frac{a}{n}(1+A)^2\right)(a^{(s_l)''}(t))^2,\\
		a^{(s_l)}(t)&=\frac{a}{n}(1+A)^2 (a^{(s_l)'}(t) )+\left(1-\frac{a}{n}(1+A)^2\right)(a^{(s_l)''}(t)),
		\ee
		where
		\be 
		a^{(s_l)'}(t)&=\Ptheta^{(\C)}\left(\beta_1d_i'(1,\tilde{S}_l)+\beta_2d_i(2,\tilde{S}_l)>t\sigmabetanulll+\mubetanulll-\beta_1\right),\\
		a^{(s_l)''}(t)&=\Ptheta^{(\C)}\left(\beta_1d_i'(1,\tilde{S}_l)+\beta_2d_i(2,\tilde{S}_l)>t\sigmabetanulll+\mubetanulll\right), \label{eqn:aprimes_signal}
		\ee
		where $d_i'(1,\tilde{S}_l)=\sum\limits_{t\in S_l:\atop t\neq i'}Y_{it}$. Therefore using Lemma \ref{lemma_simple_identity} for $l \in \{1,2\}$ we have
		\be
		(b^{(s_l)}(t)-(a^{(s_l)}(t))^2)&=(1+A)^2\frac{a}{n}\left(1-(1+A)^2\frac{a}{n}\right)\left(a^{(s_l)'}(t)-a^{(s_l)''}(t)\right)^2.\\ \label{eqn:signal_cov_within_block}
		\ee
		
		Now note that for $l \in \{1,\ldots,k\}$ 
		\be 
		\ & \left(a^{(s_l)'}(t)-a^{(s_l)''}(t)\right)\\
		&=\Ptheta^{(\C)}\left(\beta_1d_i'(1,\tilde{S}_l)+\beta_2d_i(2,\tilde{S}_l)>t\sigmabetanulll+\mubetanulll-\beta_1\right)
		\\
		&-\Ptheta^{(\C)}\left(\beta_1d_i'(1,\tilde{S}_l)+\beta_2d_i(2,\tilde{S}_l)>t\sigmabetanulll+\mubetanulll\right)\\
		&\leq const\cdot \sup_{|\xi|\leq \beta_1}\Ptheta^{(\C)}\left(\beta_1d_i'(1,\tilde{S}_l)+\beta_2d_i(2,\tilde{S}_l)=t\sigmabetanulll+\mubetanulll+\xi\right),\\ \label{eqn:signal_lclt}
		\ee
		where $const$ depends only on $\beta_1,\beta_2$. Now $\beta_1d_i'(1,\tilde{S}_l)+\beta_2d_i(2,\tilde{S}_l)\sim \sum\limits_{t=1}^4 Z_t^{(l)}$ with independent components
		\be 
		Z_1^{(l)}\sim \Bin\left(s_l-2,(1+A)^2\frac{a}{n}\right),\\
		Z_2^{(l)}\sim \Bin\left(n\lambda_l-s_l,(1+A)\frac{a}{n}\right),\\
		Z_3^{(l)}\sim \Bin\left(n\lambda_l-\sum_{j\neq l}s_j,(1+A)\frac{b}{n}\right),\\
		Z_4^{(l)}\sim \Bin\left(\sum_{j\neq l}s_j,(1+A)^2\frac{b}{n}\right).\label{eqn:signal_bin_sim}
		\ee
		Further note that
		\be 
		\ & \frac{t\sigmabetanulll+\beta_1\muonenull(\tilde{S}_l)+\beta_2\mutwonull(\tilde{S}_l)-\beta_1(n\lambda_l-s_l)(1+A)\frac{a}{n}-\beta_2(n(1-\lambda_l)-\sum_{j\neq l}s_j)(1+A)\frac{b}{n}}{\sqrt{\beta_1^2(n\lambda_l-s_l)(1+A)\frac{a}{n}\left(1-(1+A)\frac{a}{n}\right)+\beta_2^2(n(1-\lambda_l)-\sum_{j\neq l}s_j)(1+A)\frac{b}{n}\left(1-(1+A)\frac{b}{n}\right)}}\\
		&=\left(\sqrt{2r}-\sqrt{C^*\bar{\rho_l}}\right)(1+o(1))\sqrt{\log{n}},\label{eqn:signal_cov_exponent}
		\ee
		where recall that
		\be 
		\bar{\rho_l}^{-1}=\left[ \frac{(\beta_1^2 \tau_a (1 - \tau_a)\lambda_l + \beta_2^2 (1-\lambda_l)\tau_b(1- \tau_b) ) ( \tau_a (1 - \tau_a)\bar{\lambda} + (1-\bar{\lambda})\tau_b (1- \tau_{b}) ) }{(\beta_1 \tau_a\lambda_l + \beta_2 (1-\lambda_l)\tau_b )^2} \right].
		\ee
		
		Next consider $\mathrm{Cov}^{(\C)}_{\mathbf{\Theta},a,b}\left(\I\left(D_i(\tilde{S}_{l_1},\beta_1,\beta_2)>t\right),\I\left(D_{i'}(\tilde{S}_{l_2},\beta_1,\beta_2)>t\right)\right)$ for $i\in \tilde{S}_{l_1}\cap S$ and $i'\in \tilde{S}_{l_2}\cap S$ for some $l_1\neq l_2\in\{1,\ldots,k\}$. Then
		\be 
		\mathrm{Cov}^{(\C)}_{\mathbf{\Theta},a,b}\left(\I\left(D_i(\tilde{S}_{l_1},\beta_1,\beta_2)>t\right),\I\left(D_{i'}(\tilde{S}_{l_2},\beta_1,\beta_2)>t\right)\right)&=b^{(s_{l_1},s_{l_2})}(t)-a^{(s_{l_1})}(t)a^{(s_{l_2})}(t),
		\ee
		where
		\be 
		b^{(s_{l_1},s_{l_2})}&=\Ptheta^{(\C)}(D_i(\hcparlone)>t,D_{i'}(\hcparltwo)>t),i\in \tilde{S}_{l_1}\cap S, i'\in \tilde{S}_{l_2}\cap S\\
		a^{(s_l)}(t)&=\Ptheta^{(\C)}\left(D_i(\hcparl)>t\right).
		\ee
		Now note that
		\be
		b^{(s_{l_1},s_{l_2})}(t)&=\frac{b}{n}(1+A)^2 (a^{(s_{l_1},s_{l_2})'}(t)a^{(s_{l_2},s_{l_1})'}(t) )+\left(1-\frac{b}{n}(1+A)^2\right)(a^{(s_{l_1},s_{l_2})''}(t)a^{(s_{l_2},s_{l_1})''}(t) ),
		\ee
		where
		\be 
		a^{(s_{l_1},s_{l_2})'}(t)&=\Ptheta^{(\C)}\left(\beta_1d_i(1,\tilde{S}_{l_1})+\beta_2d_i'(2,\tilde{S}_{l_1})>t\sigmabetanulllone+\mubetanulllone-\beta_2\right),\\
		a^{(s_{l_1},s_{l_2})''}(t)&=\Ptheta^{(\C)}\left(\beta_1d_i(1,\tilde{S}_{l_1})+\beta_2d_i'(2,\tilde{S}_{l_1})>t\sigmabetanulllone+\mubetanulllone\right),\\
		a^{(s_{l_2},s_{l_1})'}(t)&=\Ptheta^{(\C)}\left(\beta_1d_{i'}(1,\tilde{S}_{l_2})+\beta_2d_{i'}'(2,\tilde{S}_{l_2})>t\sigmabetanullltwo+\mubetanullltwo-\beta_2\right),\\
		a^{(s_{l_2},s_{l_1})''}(t)&=\Ptheta^{(\C)}\left(\beta_1d_{i'}(1,\tilde{S}_{l_2})+\beta_2d_{i'}'(2,\tilde{S}_{l_2})>t\sigmabetanullltwo+\mubetanullltwo\right), \label{eqn:aprimes_cross_signal}
		\ee
		where we define $d_i'(2,\tilde{S}_{l_1})=\sum\limits_{t\in \tilde{S}_{l_1}^c:\atop t\neq i'}Y_{it}$ and $d_{i'}'(2,\tilde{S}_{l_2})=\sum\limits_{t\in \tilde{S}_{l_2}^c:\atop t\neq i}Y_{i't}$
		Therefore using Lemma \ref{lemma_simple_identity}
		\be 
		\ & (b^{(s_{l_1},s_{l_2})}(t)-a^{(s_{l_1})}(t)a^{(s_{l_2})}(t))\\&=(1+A)^2\frac{b}{n}\left(1-(1+A)^2\frac{b}{n}\right)\left(a^{(s_{l_1},s_{l_2})'}(t)-a^{(s_{l_1},s_{l_2})''}(t)\right)\left(a^{(s_{l_2},s_{l_1})'}(t)-a^{(s_{l_2},s_{l_1})''}(t)\right).\\\label{eqn:signal_cov_across_block}
		\ee
		Similar to before
		\be 
		\ &\left(a^{(s_{l_1},s_{l_2})'}(t)-a^{(s_{l_1},s_{l_2})''}(t)\right)\\
		&\leq const\cdot \sup_{|\xi|\leq \beta_1}\Ptheta^{(\C)}\left(\beta_1d_i(1,\tilde{S}_{l_1})+\beta_2d_i'(2,\tilde{S}_{l_1})=t\sigmabetanulllone+\mubetanulllone+\xi\right),\\
		\ & \left(a^{(s_{l_2},s_{l_1})'}(t)-a^{(s_{l_2},s_{l_1})''}(t)\right)\\
		&\leq const\cdot \sup_{|\xi|\leq \beta_1}\Ptheta^{(\C)}\left(\beta_1d_{i'}(1,\tilde{S}_{l_2})+\beta_2d_{i'}'(2,\tilde{S}_{l_2})=t\sigmabetanullltwo+\mubetanullltwo+\xi\right). \\ \label{eqn:across_signal_lclt}
		\ee
		where $const$ depends only on $\beta_1,\beta_2$.
		
		Now $\beta_1d_i(1,\tilde{S}_{l_1})+\beta_2d_i'(2,\tilde{S}_{l_1})\sim \sum\limits_{t=1}^4 Z_t^{(l_1)}$ with independent components
		\be 
		Z_1^{(l_1)}\sim \Bin\left(s_{l_1}-1,(1+A)^2\frac{a}{n}\right),\\
		Z_2^{(l_1)}\sim \Bin\left(n\lambda_{l_1}-s_{l_1},(1+A)\frac{a}{n}\right),\\
		Z_3^{(l_1)}\sim \Bin\left(n(1-\lambda_{l_1})-\sum_{j\neq l_1}s_j,(1+A)\frac{b}{n}\right),\\
		Z_4^{(l_1)}\sim \Bin\left(\sum_{j\neq l_1}s_j-1,(1+A)^2\frac{b}{n}\right),\label{eqn:acrosssignal_bin_sim_s1}
		\ee
		and note that
		
		\be 
		\ & \frac{t\sigmabetanulllone+\mubetanulllone-\beta_1(n\lambda_{l_1}-s_{l_1})(1+A)\frac{a}{n}-\beta_2(n(1-\lambda_{l_1})-\sum_{j\neq l_1}s_j)(1+A)\frac{b}{n}}{\sqrt{\beta_1^2(n\lambda_{l_1}-s_{l_1})(1+A)\frac{a}{n}\left(1-(1+A)\frac{a}{n}\right)+\beta_2^2(n(1-\lambda_{l_1})-\sum_{j\neq l_1}s_j)(1+A)\frac{b}{n}\left(1-(1+A)\frac{b}{n}\right)}}\\
		&=\left(\sqrt{2r}-\sqrt{C^*\bar{\rho}_{l_1}}\right)(1+o(1))\sqrt{\log{n}}. \label{eqn:across_signal_cov_exponent_s1}
		\ee

		Similarly $\beta_1d_{i'}(1,\tilde{S}_{l_2})+\beta_2d_{i'}'(2,\tilde{S}_{l_2})\sim \sum\limits_{t=1}^4 Z_t^{(l_2)}$ with independent components
		\be 
		Z_1^{(l_2)}\sim \Bin\left(s_{l_2}-1,(1+A)^2\frac{a}{n}\right),\\
		Z_2^{(l_2)}\sim \Bin\left(n\lambda_{l_2}-s_{l_2},(1+A)\frac{a}{n}\right),\\
		Z_3^{(l_2)}\sim \Bin\left(n(1-\lambda_{l_2})-\sum_{j\neq l_2}s_j,(1+A)\frac{b}{n}\right),\\
		Z_4^{(l_2)}\sim \Bin\left(\sum_{j\neq l_2}s_j-1,(1+A)^2\frac{b}{n}\right),\label{eqn:acrosssignal_bin_sim_s2}
		\ee
		and note that
		
		\be 
		\ & \frac{t\sigmabetanullltwo+\mubetanullltwo-\beta_1(n\lambda_{l_2}-s_{l_2})(1+A)\frac{a}{n}-\beta_2(n(1-\lambda_{l_2})-\sum_{j\neq l_2}s_j)(1+A)\frac{b}{n}}{\sqrt{\beta_1^2(n\lambda_{l_2}-s_{l_2})(1+A)\frac{a}{n}\left(1-(1+A)\frac{a}{n}\right)+\beta_2^2(n(1-\lambda_{l_2})-\sum_{j\neq l_2}s_j)(1+A)\frac{b}{n}\left(1-(1+A)\frac{b}{n}\right)}}\\
		&=\left(\sqrt{2r}-\sqrt{C^*\bar{\rho}_{l_2}}\right)(1+o(1))\sqrt{\log{n}}. \label{eqn:across_signal_cov_exponent_s2}
		\ee
		
		Therefore, applying  Lemma \ref{lemma:binomial_master} Part (a, ii) along with \eqref{eqn:signal_cov_within_block},
		\eqref{eqn:signal_lclt}, \eqref{eqn:signal_bin_sim},	\eqref{eqn:signal_cov_exponent},
		\eqref{eqn:signal_cov_across_block},
		\eqref{eqn:across_signal_lclt},
		\eqref{eqn:acrosssignal_bin_sim_s1},
		\eqref{eqn:across_signal_cov_exponent_s1},
		\eqref{eqn:acrosssignal_bin_sim_s2},
		\eqref{eqn:across_signal_cov_exponent_s2}, we have for any fixed $\varepsilon>0$ that the desired control of $T_3$ holds.

		\paragraph*{\textbf{Control of $T_4$}}
		
		The analysis of $T_4$ is similar in philosophy to that of $T_3$ and goes through a reduction to supremum of local central limit theorem type probability estimates for linear combination of independent Binomial random variables. However, since we need to control a similar term in the proof of Theorem \ref{thm:max_deg_null}, we present the control of $T_4$ below.

		First consider $\mathrm{Cov}^{(\C)}_{\mathbf{\Theta},a,b}\left(\I\left(D_i(\tilde{S}_l,\beta_1,\beta_2)>t\right),\I\left(D_{i'}(\tilde{S}_l,\beta_1,\beta_2)>t\right)\right)$ for $i\neq i'\in \tilde{S}_l\cap S^c$ for some $l=1,\ldots,k$. Then
		\be 
		\mathrm{Cov}^{(\C)}_{\mathbf{\Theta},a,b}\left(\I\left(D_i(\tilde{S}_l,\beta_1,\beta_2)>t\right),\I\left(D_{i'}(\tilde{S}_l,\beta_1,\beta_2)>t\right)\right)&=b^{(s_l^c)}(t)-(a^{(s_l^c)}(t))^2,
		\ee
		where for $i,i'\in \tilde{S}_l\cap S^c$
		\be 
		b^{(s_l^c)}(t)&=\Ptheta^{(\C)}(D_i(\hcparl)>t,D_{i'}(\hcparl)>t),\\
		a^{(s_l)}(t)&=\Ptheta^{(\C)}\left(D_i(\hcparl)>t\right).
		\ee
		Now note that
		\be 
		b^{(s_l^c)}(t)&=\frac{a}{n} (a^{(s_l^c)'}(t) )^2+\left(1-\frac{a}{n}\right)(a^{(s_l^c)''}(t))^2,\\
		a^{(s_l^c)}(t)&=\frac{a}{n} (a^{(s_l^c)'}(t) )+\left(1-\frac{a}{n}\right)(a^{(s_l^c)''}(t)),
		\ee
		where
		where
		\be 
		a^{(s_l^c)'}(t)&=\Ptheta^{(\C)}\left(\beta_1d_i'(1,\tilde{S}_l)+\beta_2d_i(2,\tilde{S}_l)>t\sigmabetanulll+\mubetanulll-\beta_1\right),\\
		a^{(s_l^c)''}(t)&=\Ptheta^{(\C)}\left(\beta_1d_i'(1,\tilde{S}_l)+\beta_2d_i(2,\tilde{S}_l)>t\sigmabetanulll+\mubetanulll\right), \label{eqn:aprimes_nonsignal}
		\ee
		where $d_i'(1,\tilde{S}_l)=\sum\limits_{t\in S_l:\atop t\neq i'}Y_{it}$. Therefore using Lemma \ref{lemma_simple_identity} for $l \in \{1,2\}$ we have
		\be
		(b^{(s_l^c)}(t)-(a^{(s_l^c)}(t))^2)&=\frac{a}{n}\left(1-\frac{a}{n}\right)\left(a^{(s_l^c)'}(t)-a^{(s_l^c)''}(t)\right)^2.\\ \label{eqn:nonsignal_cov_within_block}
		\ee
		
		Now note that for $l \in \{1,\ldots,k\}$ 
		\be 
		\ & \left(a^{(s_l^c)'}(t)-a^{(s_l^c)''}(t)\right)\\
		&=\Ptheta^{(\C)}\left(\beta_1d_i'(1,\tilde{S}_l)+\beta_2d_i(2,\tilde{S}_l)>t\sigmabetanulll+\mubetanulll-\beta_1\right)\\
		&-\Ptheta^{(\C)}\left(\beta_1d_i'(1,\tilde{S}_l)+\beta_2d_i(2,\tilde{S}_l)>t\sigmabetanulll+\mubetanulll\right)\\
		&\leq const\cdot \sup_{|\xi|\leq \beta_1}\Ptheta^{(\C)}\left(\beta_1d_i'(1,\tilde{S}_l)+\beta_2d_i(2,\tilde{S}_l)=t\sigmabetanulll+\mubetanulll+\xi\right),\\ \label{eqn:nonsignal_lclt}
		\ee
		where $const$ depends only on $\beta_1,\beta_2$. Now $\beta_1d_i'(1,\tilde{S}_l)+\beta_2d_i(2,\tilde{S}_l)\sim \sum\limits_{t=1}^4 Z_t^{(l)}$ with independent components
		\be 
		Z_1^{(l)}\sim \Bin\left(s_l,(1+A)\frac{a}{n}\right),\\
		Z_2^{(l)}\sim \Bin\left(n\lambda_l-s_l-2,\frac{a}{n}\right),\\
		Z_3^{(l)}\sim \Bin\left(n\lambda_l-\sum_{j\neq l}s_j,\frac{b}{n}\right),\\
		Z_4^{(l)}\sim \Bin\left(\sum_{j\neq l}s_j,(1+A)\frac{b}{n}\right).\label{eqn:nonsignal_bin_sim}
		\ee
		Further note that
		\be 
		\ & \frac{t\sigmabetanulll+\beta_1\muonenull(\tilde{S}_l)+\beta_2\mutwonull(\tilde{S}_l)-\beta_1(n\lambda_l-s_l-2)\frac{a}{n}-\beta_2(n(1-\lambda_l)-\sum_{j\neq l}s_j)\frac{b}{n}}{\sqrt{\beta_1^2(n\lambda_l-s_l-2)\frac{a}{n}\left(1-\frac{a}{n}\right)+\beta_2^2(n(1-\lambda_l)-\sum_{j\neq l}s_j)\frac{b}{n}\left(1-\frac{b}{n}\right)}}\\
		&=\sqrt{2r}(1+o(1))\sqrt{\log{n}}.\label{eqn:nonsignal_cov_exponent}
		\ee

		Next consider $\mathrm{Cov}^{(\C)}_{\mathbf{\Theta},a,b}\left(\I\left(D_i(\tilde{S}_{l_1},\beta_1,\beta_2)>t\right),\I\left(D_{i'}(\tilde{S}_{l_2},\beta_1,\beta_2)>t\right)\right)$ for $i\in \tilde{S}_{l_1}\cap S^c$ and $i'\in \tilde{S}_{l_2}\cap S^c$ for some $l_1\neq l_2\in\{1,\ldots,k\}$. Then
		\be 
		\mathrm{Cov}^{(\C)}_{\mathbf{\Theta},a,b}\left(\I\left(D_i(\tilde{S}_{l_1},\beta_1,\beta_2)>t\right),\I\left(D_{i'}(\tilde{S}_{l_2},\beta_1,\beta_2)>t\right)\right)&=b^{(s_{l_1}^c,s_{l_2}^c)}(t)-a^{(s_{l_1}^c)}(t)a^{(s_{l_2}^c)}(t),
		\ee
		where
		\be 
		b^{(s_{l_1}^c,s_{l_2}^c)}&=\Ptheta^{(\C)}(D_i(\hcparlone)>t,D_{i'}(\hcparltwo)>t),i\in \tilde{S}_{l_1}\cap S, i'\in \tilde{S}_{l_2}\cap S\\
		a^{(s_l^c)}(t)&=\Ptheta^{(\C)}\left(D_i(\hcparl)>t\right).
		\ee
		Now note that
		\be
		b^{(s_{l_1}^c,s_{l_2}^c)}(t)&=\frac{b}{n} (a^{(s_{l_1},s_{l_2})'}(t)a^{(s_{l_2},s_{l_1})'}(t) )+\left(1-\frac{b}{n}\right)(a^{(s_{l_1},s_{l_2})''}(t)a^{(s_{l_2},s_{l_1})''}(t) ),
		\ee
		where
		\be 
		a^{(s_{l_1}^c,s_{l_2}^c)'}(t)&=\Ptheta^{(\C)}\left(\beta_1d_i(1,\tilde{S}_{l_1})+\beta_2d_i'(2,\tilde{S}_{l_1})>t\sigmabetanulllone+\mubetanulllone-\beta_2\right),\\
		a^{(s_{l_1}^c,s_{l_2}^c)''}(t)&=\Ptheta^{(\C)}\left(\beta_1d_i(1,\tilde{S}_{l_1})+\beta_2d_i'(2,\tilde{S}_{l_1})>t\sigmabetanulllone+\mubetanulllone\right),\\
		a^{(s_{l_2}^c,s_{l_1}^c)'}(t)&=\Ptheta^{(\C)}\left(\beta_1d_{i'}(1,\tilde{S}_{l_2})+\beta_2d_{i'}'(2,\tilde{S}_{l_2})>t\sigmabetanullltwo+\mubetanullltwo-\beta_2\right),\\
		a^{(s_{l_2}^c,s_{l_1}^c)''}(t)&=\Ptheta^{(\C)}\left(\beta_1d_{i'}(1,\tilde{S}_{l_2})+\beta_2d_{i'}'(2,\tilde{S}_{l_2})>t\sigmabetanullltwo+\mubetanullltwo\right),\\ \label{eqn:aprimes_cross_nonsignal}
		\ee
		where we define $d_i'(2,\tilde{S}_{l_1})=\sum\limits_{t\in \tilde{S}_{l_1}^c:\atop t\neq i'}Y_{it}$ and $d_{i'}'(2,\tilde{S}_{l_2})=\sum\limits_{t\in \tilde{S}_{l_2}^c:\atop t\neq i}Y_{i't}$
		Therefore using Lemma \ref{lemma_simple_identity}
		\be 
		\ & (b^{(s_{l_1}^c,s_{l_2}^c)}(t)-a^{(s_{l_1}^c)}(t)a^{(s_{l_2}^c)}(t))\\&=\frac{b}{n}\left(1-\frac{b}{n}\right)\left(a^{(s_{l_1}^c,s_{l_2}^c)'}(t)-a^{(s_{l_1}^c,s_{l_2}^c)''}(t)\right)\left(a^{(s_{l_2}^c,s_{l_1}^c)'}(t)-a^{(s_{l_2}^c,s_{l_1}^c)''}(t)\right).\\\label{eqn:nonsignal_cov_across_block}
		\ee
		Similar to before
		\be 
		\ &\left(a^{(s_{l_1}^c,s_{l_2}^c)'}(t)-a^{(s_{l_1}^c,s_{l_2}^c)''}(t)\right)
		\\&\leq const\cdot \sup_{|\xi|\leq \beta_1}\Ptheta^{(\C)}\left(\beta_1d_i(1,\tilde{S}_{l_1})+\beta_2d_i'(2,\tilde{S}_{l_1})=t\sigmabetanulllone+\mubetanulllone+\xi\right),\\
		\ & \left(a^{(s_{l_2}^c,s_{l_1}^c)'}(t)-a^{(s_{l_2}^c,s_{l_1}^c)''}(t)\right)\\
		&\leq const\cdot \sup_{|\xi|\leq \beta_1}\Ptheta^{(\C)}\left(\beta_1d_{i'}(1,\tilde{S}_{l_2})+\beta_2d_{i'}'(2,\tilde{S}_{l_2})=t\sigmabetanullltwo+\mubetanullltwo+\xi\right). \\ \label{eqn:across_nonsignal_lclt}
		\ee
		where $const$ depends only on $\beta_1,\beta_2$.
		
		Now $\beta_1d_i(1,\tilde{S}_{l_1})+\beta_2d_i'(2,\tilde{S}_{l_1})\sim \sum\limits_{t=1}^4 Z_t^{(l_1)}$ with independent components
		\be 
		Z_1^{(l_1)}\sim \Bin\left(s_{l_1},(1+A)\frac{a}{n}\right),\\
		Z_2^{(l_1)}\sim \Bin\left(n\lambda_{l_1}-s_{l_1}-1,\frac{a}{n}\right),\\
		Z_3^{(l_1)}\sim \Bin\left(n(1-\lambda_{l_1})-\sum_{j\neq l_1}s_j,\frac{b}{n}\right),\\
		Z_4^{(l_1)}\sim \Bin\left(\sum_{j\neq l_1}s_j-1,(1+A)\frac{b}{n}\right),\label{eqn:acrossnonsignal_bin_sim_s1}
		\ee
		and note that
		
		\be 
		\ & \frac{t\sigmabetanulllone+\mubetanulllone-\beta_1(n\lambda_{l_1}-s_{l_1}-1)\frac{a}{n}-\beta_2(n(1-\lambda_{l_1})-\sum_{j\neq l_1}s_j)\frac{b}{n}}{\sqrt{\beta_1^2(n\lambda_{l_1}-s_{l_1}-1)\frac{a}{n}\left(1-\frac{a}{n}\right)+\beta_2^2(n(1-\lambda_{l_1})-\sum_{j\neq l_1}s_j)\frac{b}{n}\left(1-\frac{b}{n}\right)}}\\
		&=\sqrt{2r}(1+o(1))\sqrt{\log{n}}. \label{eqn:across_nonsignal_cov_exponent_s1}
		\ee

		Similarly $\beta_1d_{i'}(1,\tilde{S}_{l_2})+\beta_2d_{i'}'(2,\tilde{S}_{l_2})\sim \sum\limits_{t=1}^4 Z_t^{(l_2)}$ with independent components
		\be 
		Z_1^{(l_2)}\sim \Bin\left(s_{l_2},(1+A)\frac{a}{n}\right),\\
		Z_2^{(l_2)}\sim \Bin\left(n\lambda_{l_2}-s_{l_2}-1,\frac{a}{n}\right),\\
		Z_3^{(l_2)}\sim \Bin\left(n(1-\lambda_{l_2})-\sum_{j\neq l_2}s_j,\frac{b}{n}\right),\\
		Z_4^{(l_2)}\sim \Bin\left(\sum_{j\neq l_2}s_j-1,(1+A)\frac{b}{n}\right),\label{eqn:acrossnonsignal_bin_sim_s2}
		\ee
		and note that
		
		\be 
		\ & \frac{t\sigmabetanullltwo+\mubetanullltwo-\beta_1(n\lambda_{l_2}-s_{l_2}-1)\frac{a}{n}-\beta_2(n(1-\lambda_{l_2})-\sum_{j\neq l_2}s_j)\frac{b}{n}}{\sqrt{\beta_1^2(n\lambda_{l_2}-s_{l_2}-1)\frac{a}{n}\left(1-\frac{a}{n}\right)+\beta_2^2(n(1-\lambda_{l_2})-\sum_{j\neq l_2}s_j)\frac{b}{n}\left(1-\frac{b}{n}\right)}}\\
		&=\sqrt{2r}(1+o(1))\sqrt{\log{n}}. \label{eqn:across_nonsignal_cov_exponent_s2}
		\ee
		
		Therefore, applying  Lemma \ref{lemma:binomial_master} Part (a, ii) along with \eqref{eqn:nonsignal_cov_within_block},
		\eqref{eqn:nonsignal_lclt}, \eqref{eqn:nonsignal_bin_sim},	\eqref{eqn:nonsignal_cov_exponent},
		\eqref{eqn:nonsignal_cov_across_block},
		\eqref{eqn:across_nonsignal_lclt},
		\eqref{eqn:acrossnonsignal_bin_sim_s1},
		\eqref{eqn:across_nonsignal_cov_exponent_s1},
		\eqref{eqn:acrossnonsignal_bin_sim_s2},
		\eqref{eqn:across_nonsignal_cov_exponent_s2}, we have for any fixed $\varepsilon>0$ that the desired control of $T_3$ holds.

		\paragraph*{\textbf{Control of $T_5$}} The analysis of $T_5$ is similar in philosophy to those of $T_3$ and $T_4$, and goes through a reduction to supremum of local central limit theorem type probability estimates for linear combination of independent Binomial random variables. We therefore omit the details.

	\end{proof}

	\appendix
	
	\section{Proofs of Binomial Deviation Bounds}	
	\label{sec:binom_proofs}
	Throughout we let $\tau_{a'}=\lim a'/n$ and $\tau_{b'}=\lim b'/n$ and let $M=\sup_{n\geq 1}\max\{|\tau_{a'}-a'/n|,|\tau_{b'}-b'/n|,|C_n-C|\}$.
	
	\subsection{Proof of Lemma \ref{lemma:binomial_master}}We prove each part of the lemma separately below.\\
	\textit{Proof of Lemma  \ref{lemma:binomial_master}  Part \ref{lemma:binomial_equal_pure}:}
	Let $$h^*=\frac{h(\beta_1\sigmaone^2)}{\sigmabeta^2},$$ 
	and 
	\be
	\A:=\{x:\muone+x/\beta_1\in\mathbb{N}\cap [0,\gamma_1 n/k]\}.
	\ee
	Then we have for any $C^*>0$
	\be 
	\ & \sup_{|t|\leq \xi_n}\P\left(d(\beta_1,\beta_2)=h+t\right)\\
	&=\sup_{|t|\leq \xi_n}\sum\limits_{h_1 \in \A}\P(X=\muone+h_1/\beta_1)\P(Y=\mutwo+(h-h_1+t)/\beta_2)\\
	&=\sup_{|t|\leq \xi_n}\sum\limits_{h_1 \in \atop\A \cap I(C^*,h^*)^c }\P(X=\muone+h_1/\beta_1)\P(Y=\mutwo+(h-h_1+t)/\beta_2)\\&+\sup_{|t|\leq \xi_n}\sum\limits_{h_1 \in \atop\A \cap I(C^*,h^*)}\P(X=\muone+h_1/\beta_1)\P(Y=\mutwo+(h-h_1+t)/\beta_2),
	\ee
	where $I(C^*,h^*):=[h^*-C^*\sigmaone\sqrt{\log{n}},h^*+C^*\sigmaone\sqrt{\log{n}}]$ and $I(C^*,h^*)^c$ denotes its complement. Now
	\be 
	\ & \sup_{|t|\leq \xi_n}\sum\limits_{h_1 \in \atop\A \cap I(C^*,h^*)^c }\P(X=\muone+h_1/\beta_1)\P(Y=\mutwo+(h-h_1+t)/\beta_2)\\
	&\leq \sup_{|t|\leq \xi_n}\P\left(Y>\mutwo+(h-h^*+C^*\sigmaone\sqrt{\log{n}}+t)/\beta_2\right)\\
	&+\sup_{|t|\leq \xi_n}\P\left(X>\muone+(h^*+C^*\sigmaone\sqrt{\log{n}})/\beta_1\right).
	\ee
	Now by Lemma 6.2 of \cite{mms2016} Part (a, ii) 
	\be 
	\sup_{|t|\leq \xi_n}\P\left(X>\muone+(h^*+C^*\sigmaone\sqrt{\log{n}})/\beta_1\right)&\leq n^{-\frac{\kappa_1^2(C^*)}{2}+o(1)},
	\ee 
	where \textcolor{black}{$\kappa_1(C^*)=\frac{C^*+c\beta_1^2\sqrt{\frac{\frac{\gamma_1}{k}\tau_{a'}(1-\tau_{a'})}{\beta_1^2\frac{\gamma_1}{k}\tau_{a'}(1-\tau_{a'})+\beta_2^2\frac{\gamma_2}{k}\tau_{b'}(1-\tau_{b'})}}}{\beta_1}$ since $c<\liminf{\frac{h}{\sigmabeta\sqrt{\log{n}}}}$}.
	Again by Lemma 6.2 of \cite{mms2016} Part (a, ii) we have for $|t|\leq \xi_n \ll \log{n}$
	
	\be 
	\sup_{|t|\leq \xi_n}\P\left(Y>\mutwo+(h-h^*+C^*\sigmaone\sqrt{\log{n}}+t)/\beta_2\right)&\leq n^{-\frac{\kappa_2^2(C^*)}{2}+o(1)},
	\ee 
	\textcolor{black}{where $\kappa_2(C^*)=\frac{C^*\sqrt{\frac{\gamma_1\tau_{a'}(1-\tau_{a'})}{\gamma_2\tau_{b'}(1-\tau_{b'})}}+c\beta_2^2\sqrt{\frac{\gamma_2\tau_{b'}(1-\tau_{b'})}{\beta_1^2\gamma_1\tau_{a'}(1-\tau_{a'})+\beta_2^2\gamma_2\tau_{b'}(1-\tau_{b'})}}}{\beta_2}$.}\\
	
	Now by Theorem 1.2 of \cite{bollobas}, whenever $|t|\leq \xi_n \ll \log{n} $, one has for any fixed $\varepsilon\in (0,1)$ and $n$ large enough (depending on $\varepsilon,\beta_1,\beta_2,C^*,c,c',M$) 
	\be 
	\ & \sup_{|t|\leq \xi_n}\sum\limits_{h_1 \in \atop\A \cap I(C^*,h^*)}\P(X=\muone+h_1/\beta_1)\P(Y=\mutwo+(h-h_1+t)/\beta_2)\\
	&\leq \sum\limits_{h_1 \in \atop\A \cap I(C^*,h^*)}\frac{1}{\sqrt{2\pi \sigmaone^2}}\exp\left(-\frac{h_1^2}{2(\beta_1\sigmaone)^2}(1-\varepsilon)\right)\times \frac{1}{\sqrt{2\pi \sigmatwo^2}}\exp\left(-\frac{(h-h_1)^2}{2(\beta_2\sigmatwo)^2}(1-\varepsilon)\right).
	\ee
	Since the function $f(h_1)=\frac{h_1^2}{2(\beta_1\sigmaone)^2}+\frac{(h-h_1)^2}{2(\beta_2\sigmatwo)^2}$ is minimized at $h_1=h^*$, we have
	\be 
	\ & \sup_{|t|\leq \xi_n}\sum\limits_{h_1 \in \atop\A \cap I(C^*,h^*)}\P(X=\muone+h_1/\beta_1)\P(Y=\mutwo+(h-h_1+t)/\beta_2)\\
	&\leq |\A \cap I(C^*,h^*)|\frac{1}{2\pi \sigmaone\sigmatwo}\exp\left(-\frac{h^2}{2\sigmabeta^2}(1-\varepsilon)\right) \label{eqn:lclt_intermsof_h}\\
	\ee
	Therefore for any given sequence $\{\xi_n\}$ such that $|\xi_n|\ll \log{n}$
	\be 
	\ & \sup\limits_{|t|\ll \xi_n }\P\left(d(\beta_1,\beta_2)=h+t\right)\\
	&\leq |\A \cap I(C^*,h^*)|\frac{1}{2\pi \sigmaone\sigmatwo}\exp\left(-\frac{h^2}{2\sigmabeta^2}(1-\varepsilon)\right)+ n^{-\frac{\kappa_1^2(C^*)}{2}+o(1)}+n^{-\frac{\kappa_2^2(C^*)}{2}+o(1)}. \label{eqn:dbeta_control}
	\ee
	Now note that $|\A \cap I(C^*,h^*)|\leq const\cdot \sigmaone\sqrt{\log{n}}$ for a constant $const$. (depending on $C^*,c,c',\beta_1,\beta_2,M$) and $\kappa_1^2(C^*),\kappa_2^2(C^*)$ are increasing function of $C^*$. The proof is therefore complete by choosing $C^*$ large enough constant (depending on $c,c',\beta_1,\beta_2,M$).\\

	\textit{Proof of Lemma  \ref{lemma:binomial_master}  Part \ref{lemma:binomial_equal_contam}:}
	For any sequence $\{\delta_n\}$ let $t_n(\delta_n)=\delta_n\sigmabeta\sqrt{\log{n}}$. We make use of the following lemma, the proof of which being simple is omitted.
	\begin{lemma}\label{lemma:delta_choice}
		There exists a positive sequence $\delta_n^*\rightarrow 0$ such that the following hold.
		\begin{enumerate}
			\item[(i)]$ 
			t_n(\delta_n^*)\wedge {\delta_n^*}\left(\beta_1\wedge\beta_2\right)\times \sqrt{a' \wedge b'}\times \sqrt{\log{n}}\gg \log{n}$.
			\item [(ii)]$
			\frac{\sqrt{\log{n}}}{\delta_n^*}\frac{(s_1\vee s_2)\times (a''\vee b'')}{n\sqrt{a'\wedge b'}}\rightarrow 0$.
		\end{enumerate}
		
	\end{lemma}
	Fix a sequence $\delta_n^*$ satisfying $(i)$ and $(ii)$ of Lemma \ref{lemma:delta_choice}. Then 
	\be 
	\ & \sup\limits_{|t|\leq \xi_n}\P\left(d'(\beta_1,\beta_2)=h+t\right)\\
	& \leq \sup\limits_{|t|\leq \xi_n}\left\{\begin{array}{c}\P\left(\beta_1 X'+\beta_2Y'>t_n(\delta_n^*)\right)\\+\sup\limits_{r\in [0,t_n(\delta_n^*)]}\P\left(d(\beta_1,\beta_2)=h+t-r\right)\end{array}\right\}\\
	&=\P\left(\beta_1 X'+\beta_2Y'>t_n(\delta_n^*)\right)+\sup\limits_{|t|\leq \xi_n,\atop r\in [0,t_n(\delta_n^*)]}\P\left(d(\beta_1,\beta_2)=h+t-r\right).
	\ee
	For the first term we have
	\be 
	\ & \P\left(\beta_1 X'+\beta_2Y'>t_n(\delta_n^*)\right)\\
	&=\P\left(\beta_1 (X'-s_1a''/n)+\beta_2(Y'-s_2b''/n)>t_n(\delta_n^*)-(\beta_1s_1a''/n+\beta_2s_2b''/n)\right).
	\ee
	Now by Lemma \ref{lemma:delta_choice}
	\be 
	\ & t_n(\delta_n^*)-(\beta_1s_1a''/n+\beta_2s_2b''/n)\\
	&\geq \frac{\delta_n^*}{2}\left(\beta_1\wedge\beta_2\right)\times \sqrt{a' \wedge b'}\times \sqrt{\log{n}}-\frac{2}{n}(\beta_1 \vee\beta_2)\times (s_1\vee s_2)\times (a''\vee b'')\\
	&=\frac{\delta_n^*}{2}\left(\beta_1\wedge\beta_2\right)\times \sqrt{a' \wedge b'}\times \sqrt{\log{n}}\left(1-2\frac{\sqrt{\log{n}}}{\delta_n^*}\frac{(s_1\vee s_2)\times (a''\vee b'')}{n\sqrt{a'\wedge b'}}\right)\\
	&\gg \log{n}.
	\ee
	Therefore by Bernstein's Inequality for $\theta>0$ one has for $n$ large enough (depending on $\theta,\beta_1,\beta_2$)  
	\be 
	\P\left(\beta_1 X'+\beta_2Y'>t_n(\delta_n^*)\right)&\leq n^{-\theta}.
	\ee
	Finally by Lemma \ref{lemma:delta_choice} and \eqref{eqn:dbeta_control} 
	\be 
	\ &\sup\limits_{|t|\leq \xi_n,\atop r\in [0,t_n(\delta_n^*)]}\P\left(d(\beta_1,\beta_2)=h+t-r\right)\\
	&\leq |\A \cap I(C^*,h^*)|\frac{1}{2\pi \sigmaone\sigmatwo}\exp\left(-\frac{h^2}{2\sigmabeta^2}(1-\varepsilon)\right)+ n^{-\frac{\kappa_1^2(C^*)}{2}+o(1)}+n^{-\frac{\kappa_2^2(C^*)}{2}+o(1)},
	\ee
	for any $C^*>0$ and $\varepsilon \in (0,1)$, whenever  $|t|\leq \xi_n\ll \log{n} $ 
	
	Therefore for any given sequence $\{\xi_n\}$ such that $|\xi_n|\ll \log{n}$, any $C^*,\theta>0$, and $\varepsilon \in (0,1)$ we have for $n$ large enough (depending on $\varepsilon,c,c',\beta_1,\beta_2,C^*,\theta,M$)
	\be 
	\ & \sup\limits_{|t|\ll \xi_n }\P\left(d'(\beta_1,\beta_2)=h+t\right)\\
	&\leq n^{-\theta}+|\A \cap I(C^*,h^*)|\frac{1}{2\pi \sigmaone\sigmatwo}\exp\left(-\frac{h^2}{2\sigmabeta^2}(1-\varepsilon)\right)+ n^{-\frac{\kappa_1^2(C^*)}{2}+o(1)}+n^{-\frac{\kappa_2^2(C^*)}{2}+o(1)}.\\ \label{eqn:dbetaprime_control}
	\ee
	Now note that $|\A \cap I(C^*,h^*)|\leq const\cdot \sigmaone\sqrt{\log{n}}$ for a constant $const$. (depending on $C^*,c,c',\beta_1,\beta_2,M$) and $\kappa_1^2(C^*),\kappa_2^2(C^*)$ are increasing function of $C^*$. The proof is therefore complete by choosing $C^*$ and $\theta$ large enough constant (depending on $c,c',\beta_1,\beta_2,M$).\\
	%
	
	\textit{Proof of Lemma  \ref{lemma:binomial_master}  Part \ref{lemma:binomial_tail_pure}:} The proof proceeds by producing upper and lower bounds on the desired moderate deviation probability. 
	\subsection*{\textbf{Upper Bound}}
	

	For $h=C_n\sigmabeta\sqrt{\log{n}}$, any $\varepsilon>0$, and $\Delta_n>0$ one has 
	\be 
	P(d(\beta_1,\beta_2)>C_n\sigmabeta\sqrt{\log{n}})&= \sum_{l=0}^{M_n-1}\P(d(\beta_1,\beta_2)\in (h+l\Delta_n,h+(l+1)\Delta_n)),
	\ee
	where $M_n=\left(\frac{n(\gamma_1\beta_1+\gamma_2\beta_2)}{k}-\mubeta-h\right)/\Delta_n$. Fix  $B>\limsup C_n$ to be chosen later and let $m_n=B\sigmabeta\sqrt{\log{n}}/\Delta_n$. Then
	
	\be 
	\ & \sum_{l=0}^{m_n-1}\P(d(\beta_1,\beta_2)\in (h+l\Delta_n,h+(l+1)\Delta_n))\\
	&\leq \sum_{l=0}^{m_n-1}|\mathscr{H}\cap (h+l\Delta_n,h+(l+1)\Delta_n)|\sup_{t \in [0,\Delta_n]}\P(d(\beta_1,\beta_2)=h+l\Delta_n+t)
	\ee
	where 	$\mathscr{H}=\{x:\mubeta+x\in \beta_1\mathbb{N}+\beta_2\mathbb{N}\}$. \textcolor{black}{Now it is easy to see that $|\mathscr{H}\cap (h+l\Delta_n,h+(l+1)\Delta_n)|\leq \frac{\Delta_n^2}{\beta_1\beta_2}$.} Also by the choice of $m_n$, for any $\varepsilon>0$
	\be 
	\sup_{t \in [0,\Delta_n]}\P(d(\beta_1,\beta_2)=h+l\Delta_n+t)&\leq \frac{1}{\sqrt{2\pi}\sigmabeta}\exp\left(-\frac{(h+l\Delta_n)^2}{2\sigmabeta^2}(1-\varepsilon)^{1/2}\right).
	\ee
	Therefore as long as $\Delta_n$ is bounded we have by arguments similar to the proof of part (a, i)
	\be 
	\ & \sum_{l=0}^{m_n-1}\P(d(\beta_1,\beta_2)\in (h+l\Delta_n,h+(l+1)\Delta_n))\\
	&\leq \sum_{l=0}^{m_n-1}|\mathscr{H}\cap (h+l\Delta_n,h+(l+1)\Delta_n)|\sup_{t \in [0,\Delta_n]}\P(d(\beta_1,\beta_2)=h+l\Delta_n+t)\\
	&\leq \frac{\Delta_n^2}{\beta_1\beta_2}\sum_{l=0}^{m_n-1}\frac{1}{\sqrt{2\pi}\sigmabeta}\exp\left(-\frac{(h+l\Delta_n)^2}{2\sigmabeta^2}(1-\varepsilon)^{1/2}\right)\\
	&\leq \frac{\Delta_n^2}{\beta_1\beta_2}\int_{0}^{m_n}\frac{1}{\sqrt{2\pi}\sigmabeta}\exp\left(-\frac{(h+x\Delta_n)^2}{2\sigmabeta^2}(1-\varepsilon)^{1/2}\right)dx\\
	&=\frac{\Delta_n^2}{\beta_1\beta_2}\left(\bar{\Phi}\left(\frac{h(1-\varepsilon)^{1/2}}{\sigmabeta}\right)-\bar{\Phi}\left(\frac{(h+m_n\Delta_n)(1-\varepsilon)^{1/2}}{\sigmabeta}\right)\right)\\
	&=\frac{\Delta_n^2}{\beta_1\beta_2}\left(\bar{\Phi}\left(\frac{h(1-\varepsilon)^{1/2}}{\sigmabeta}\right)-\bar{\Phi}\left(\frac{(h+B\sigmabeta\sqrt{\log{n}})(1-\varepsilon)^{1/2}}{\sigmabeta}\right)\right).
	\ee
	Therefore if $\Delta_n$ is bounded
	\be 
	\ & P(d(\beta_1,\beta_2)>C_n\sigmabeta\sqrt{\log{n}})\\
	&= \sum_{l=0}^{m_n-1}\P(d(\beta_1,\beta_2)\in (h+l\Delta_n,h+(l+1)\Delta_n))\\&+\sum_{l=m_n}^{M_n}\P(d(\beta_1,\beta_2)\in (h+l\Delta_n,h+(l+1)\Delta_n))\\
	&\leq \frac{\Delta_n^2}{\beta_1\beta_2}\left(\bar{\Phi}\left(\frac{h(1-\varepsilon)^{1/2}}{\sigmabeta}\right)-\bar{\Phi}\left(\frac{(h+B\sigmabeta\sqrt{\log{n}})(1-\varepsilon)^{1/2}}{\sigmabeta}\right)\right)\\&+P(d(\beta_1,\beta_2)>h+B\sigmabeta\sqrt{\log{n}}).
	\ee
	It remains to control $P(d(\beta_1,\beta_2)>h+B\sigmabeta\sqrt{\log{n}})$ which we will do using a naive Bernstein bound. In particular we have by Bernstein's Inequality 
	\be 
	\ & P(d(\beta_1,\beta_2)>h+B\sigmabeta\sqrt{\log{n}})\\&\leq \exp\left(-\frac{1}{2}\frac{(h+B\sigmabeta\sqrt{\log{n}})^2}{\sigmabeta^2+\frac{1}{3}(\beta_1\vee\beta_2)(h+B\sigmabeta\sqrt{\log{n}})}\right).
	\ee
	Now note that $\sigmabeta^2\gg h+B\sigmabeta\sqrt{\log{n}}$. Therefore for sufficiently large $n$
	\be 
	\sigmabeta^2+\frac{1}{3}(\beta_1\vee\beta_2)(h+B\sigmabeta\sqrt{\log{n}})\leq 2\sigmabeta^2.
	\ee
	As a consequence for sufficiently large $n$
	\be 
	P(d(\beta_1,\beta_2)>h+B\sigmabeta\sqrt{\log{n}})&\leq \exp\left(-\frac{1}{4}(C_n+B)^2\log{n}\right).
	\ee
	The desired control of the upper bound is thereafter complete by choosing $B$ large enough depending on $\varepsilon>0$.
	\subsection*{\textbf{Lower Bound}}
	We first claim that for any $C_n\sigmabeta\sqrt{\log{n}}\leq h\ll b'$
	\be 
	\P(d(\beta_1,\beta_2)\in (h,h+3\beta_2))\geq \frac{1}{\sqrt{2\pi}\sigmabeta}\exp\left(-\frac{h^2}{2\sigmabeta^2}(1+\varepsilon)^{1/2}\right).\label{eqn:claim_binomial_tail_pure}
	\ee
	Deferring the proof of \eqref{eqn:claim_binomial_tail_pure}, we first finish the proof of the lower bound.  
	In view of the claim, for $t=C_n\sigmabeta\sqrt{\log{n}}$ and any $M_n\ll b'$ one has for any $\varepsilon$
	\be 
	\ & P(d(\beta_1,\beta_2)>C_n\sigmabeta\sqrt{\log{n}})\\&\geq \sum_{l=0}^{M_n}\P(d(\beta_1,\beta_2)\in (t+3l\beta_2,t+3(l+1)\beta_2))\\
	&\geq \sum_{l=0}^{M_n}\frac{1}{\sqrt{2\pi}\sigmabeta}\exp\left(-\frac{(t+3l\beta_2)^2}{2\sigmabeta^2}(1+\varepsilon)^{1/2}\right)\\
	&\geq \int_{0}^{M_n}\frac{1}{\sqrt{2\pi}\sigmabeta}\exp\left(-\frac{(t+3x\beta_2)^2}{2\sigmabeta^2}(1+\varepsilon)^{1/2}\right)dx\\
	&=\bar{\Phi}\left(\frac{t(1+\varepsilon)^{1/2}}{\sigmabeta}\right)-\bar{\Phi}\left(\frac{(t+3M_n\beta_2)(1+\varepsilon)^{1/2}}{\sigmabeta}\right).
	\ee
	Using Mill's ratio the proof of the lower bound is therefore complete by choosing $C_n\sigmabeta\sqrt{\log{n}}\ll M_n\ll b'$.

	We now complete the proof of the claim in \eqref{eqn:claim_binomial_tail_pure}. 
	
	The main idea of the proof is simple and relies on  finding $O(\sigmabeta)$ distinct pairs $(h_1,h_2)$ such that $\beta_1h_1+\beta_2h_2-\mubeta\in (h,h+3\beta_2)$  and  $P(X=h_1)P(Y=h_2)\geq \frac{1}{2\pi \sigmaone\sigmatwo}\exp\left(-\frac{h^2}{2\sigmabeta^2}(1+\varepsilon)\right)$. The proof can thereby be completed by adding over these contributing $O(\sigmabeta)$ distinct pairs $(h_1,h_2)$.
	
	For a fixed $\bar{h}> C_n\sigmabeta\sqrt{\log{n}}$ (to be decided on later), consider any
	$$C_n\sigmabeta\sqrt{\log{n}}\leq h\leq \bar{h} ,\quad h^*=\frac{h(\beta_1\sigmaone)^2}{\sigmabeta^2},$$
	and let for any $h_1>0$ such that $h_1=O(\sigma_{n1}\sigma_{n2}/\sigma_n(\beta_1,\beta_2))$
	$$h_1^*(h_1)=\left\lceil \muone+\frac{h^*+h_1}{\beta_1}\right\rceil,$$
	
	and 
	\be 
	t^*(h_1)=\beta_2\left\{\left\lceil \mutwo+\frac{h+\beta_1\left(\muone-h_1^*(h_1)\right)}{\beta_2}\right\rceil-\left(\mutwo+\frac{h+\beta_1\left(\muone-h_1^*(h_1)\right)}{\beta_2}\right)\right\}. \\\label{eqn:tstar_alt_difficult}
	\ee
	
	We will need the fact that $\left\lceil \mutwo+\frac{h+\beta_1\left(\muone-h_1^*(h_1)\right)}{\beta_2}\right\rceil\geq 0$ which is guaranteed by the next easy to show lemma whose proof is omitted.
	\begin{lemma}
		\textcolor{black}{For $n$ sufficiently large (depending on $\beta_1,\beta_2$) we have for $h_1=O(\sigma_{n1}\sigma_{n2}/\sigma_n(\beta_1,\beta_2))$ that $\mutwo+\frac{h+\beta_1\left(\muone-h_1^*(h_1)\right)}{\beta_2}\geq 0$.
		}
	\end{lemma}
	We can now proceed as follows. 
	Note that with $\delta_n(h_1)=\left\lceil \muone+\frac{h^*+h_1}{\beta_1}\right\rceil-(\muone+\frac{h^*+h_1}{\beta_1})$ and $\tilde{h}(h_1)=\frac{h^*+h_1}{\beta_1}\left(1+\frac{\beta_1\delta_n(h_1)}{h^*+h_1}\right)$ one has by Theorem 1.5 of \cite{bollobas}
	\be 
	\ & \P\left(X=h_1^*(h_1)\right)\\
	&=\P\left(X= \muone+\frac{h^*+h_1}{\beta_1}+\delta_n(h_1)\right)\\
	&=\P\left(X=\muone+\tilde{h}(h_1)\right)\\
	&\geq \frac{1}{\sqrt{2\pi p_1q_1n_1}}\exp\left(-\frac{\tilde{h}(h_1)^2}{2p_1q_1n_1}\left(\begin{array}{c}1+\frac{\tilde{h}(h_1)p_1}{q_1 n_1}+\frac{2q_1\tilde{h}(h_1)^2}{3p_1^2n_1^2}+\frac{q_1}{\tilde{h}(h_1)}\\+\left(\frac{1}{h_1^*(h_1)}+\frac{1}{n_1-h_1^*(h_1)}\right)\frac{n_1p_1q_1}{6\tilde{h}(h_1)^2}\end{array}\right)\right),
	\ee
	where $n_1=\gamma_1n/k$, $p_1=a'/n$ and $q_1=1-a'/n$. Now it is easy to see that
	\be 
	\frac{\tilde{h}(h_1)p_1}{q_1 n_1}=O\left(\frac{\tilde{h}(h_1)a'}{n^2}\right),\ 
	\frac{2q_1\tilde{h}(h_1)^2}{3p_1^2n_1^2}=O\left(\frac{\tilde{h}(h_1)^2}{(a')^2}\right),\
	\frac{q_1}{\tilde{h}(h_1)}=O\left(\frac{1}{\tilde{h}(h_1)}\right),\\
	\left(\frac{1}{h_1^*(h_1)}+\frac{1}{n_1-h_1^*(h_1)}\right)\frac{n_1p_1q_1}{6\tilde{h}(h_1)^2}=O\left(\frac{a'}{h_1^*(h_1)\tilde{h}(h_1)^2}\right),
	\ee
	where the $O$-notations involve universal constants free from $\beta_1,\beta_2,C_n$. If $\bar{h},h_1$ is such that,
	\be 
	\bar{h}\ll a', \quad h_1\leq \beta_1\sqrt{2\pi}\frac{\sigmaone\sigmatwo}{\sigmabeta} \label{eqn:condition_hbar_and_h_1},
	\ee
	then since $b' \gg \log{n}$, we have for any $\varepsilon>0$, sufficiently large $n$ (depending on $M$ and $\varepsilon>0$)
	\be 
	\P\left(X=h_1^*(h_1)\right)\geq \frac{1}{\sqrt{2\pi \sigmaone^2}}\exp\left(-\frac{(h^*)^2}{2(\beta_1 \sigmaone)^2}\left(1+\varepsilon\right)^{1/2}\right).\label{eqn:lowerbound_x_difficult}
	\ee
	
	Similarly for $n_2=\gamma_2n/k$, $p_2=b'/n$ and $q_2=1-b'/n$ and any $m \in \mathbb{N}$ one has
	\be 
	\ &\P\left(Y=\mutwo+\frac{h+t^*(h_1)+\beta_2m+\beta_1\left(\muone-h_1^*(h_1)\right)}{\beta_2}\right)\\
	&=\left(Y=\left\lceil \mutwo+\frac{h-(h^*+h_1)-\beta_1\delta_n(h_1)}{\beta_2}\right\rceil+m\right)\\
	&=\P\left(Y=\mutwo+{\tilde{h}}(m)\right),
	\ee
	where $\tilde{h}(m)=\frac{h-(h^*+h_1)}{\beta_2}\left(1-\frac{\frac{\beta_1}{\beta_2}\delta_n(h_1)-\beta_2m-\beta_2\delta_n'(h_1)}{h-(h^*+h_1)}\right)$ with $\delta_n'(h_1)=\left\lceil \mutwo+\frac{h-(h^*+h_1)-\beta_1\delta_n(h_1)}{\beta_2}\right\rceil-\left(\mutwo+\frac{h-(h^*+h_1)-\beta_1\delta_n(h_1)}{\beta_2}\right)$. Therefore, once again by Theorem 1.5 of \cite{bollobas}
	\be
	\ & \P\left(Y=\mutwo+{\tilde{h}}(m)\right)\\
	&\geq \frac{1}{\sqrt{2\pi p_2q_2n_2}}\exp\left(-\frac{\tilde{h}(m)^2}{2p_2q_2n_2}\left(\begin{array}{c}1+\frac{\tilde{h}(m)p_2}{q_2 n_2}+\frac{2q_2\tilde{h}(m)^2}{3p_2^2n_2^2}+\frac{q_2}{\tilde{h}(m)}\\+\left(\frac{1}{\mutwo+\tilde{h}(m)}+\frac{1}{n_1-\mutwo-\tilde{h}(m)}\right)\frac{n_2p_2q_2}{6\tilde{h}(m)^2}\end{array}\right)\right)
	\ee
	Now it is easy to see that
	\be 
	\frac{\tilde{h}(m)p_2}{q_2 n_2}=O\left(\frac{\tilde{h}(m)b'}{n^2}\right),\
	\frac{2q_2\tilde{h}(m)^2}{3p_2^2n_2^2}=O\left(\frac{\tilde{h}(m)^2}{(b')^2}\right),\
	\frac{q_2}{\tilde{h}(m)}=O\left(\frac{1}{\tilde{h}(m)}\right),\\
	\left(\frac{1}{\mutwo+\tilde{h}(m)}+\frac{1}{n_2-\mutwo-\tilde{h}(m)}\right)\frac{n_2p_2q_2}{6\tilde{h}(m)^2}=O\left(\frac{b'}{(\mutwo+\tilde{h}(m))\tilde{h}(m)^2}\right),
	\ee
	where the $O$-notations involve universal constants free from $\beta_1,\beta_2,C_n$. 
	
	If $\bar{h},h_1,m$ is such that,
	\be 
	\bar{h}\ll a', \quad  h_1\leq \beta_1\sqrt{2\pi}\frac{\sigmaone\sigmatwo}{\sigmabeta}, \quad m \leq \sigmatwo,\label{eqn:condition_hbar_and_h_2}
	\ee
	then since $b' \gg \log{n}$, we have for any $\varepsilon>0$, sufficiently large $n$ (depending on $M$ and $\varepsilon>0$)
	\be 
	\P\left(Y=\mutwo+{\tilde{h}}(m)\right)\geq \frac{1}{\sqrt{2\pi \sigmatwo^2}}\exp\left(-\frac{(h-h^*)^2}{2(\beta_2\sigmatwo)^2}\left(1+\varepsilon\right)^{1/2}\right).\label{eqn:lowerbound_y_difficult}
	\ee
	Combining \eqref{eqn:lowerbound_x_difficult} and \eqref{eqn:lowerbound_y_difficult}, we have that under the common conditions \eqref{eqn:condition_hbar_and_h_1}, \eqref{eqn:condition_hbar_and_h_2}
	\be 
	\P\left(X=h_1^*(h_1)\right)\P\left(Y=\mutwo+{\tilde{h}}(m)\right)&\geq \frac{1}{2\pi \sigmaone \sigmatwo}\exp\left(-\frac{h^2}{2\sigmabeta^2}(1+\varepsilon)^{1/2}\right).
	\ee
	\textcolor{black}{Now for each $h_1 \subseteq \left[0, \beta_1\sqrt{2\pi}\frac{\sigmaone\sigmatwo}{\sigmabeta}\right]\cap \beta_1\mathbb{N}$  the number of $h_1^*(h_1)$ is distinct and 
	}
	\be 
	\textcolor{black}{\beta_1 h_1^*(h_1)+\beta_2\mutwo+{\tilde{h}}(m)-\mubeta\in (h,h+3m\beta_2).}\quad  
	\ee
	\textcolor{black}{Therefore we can choose $m=1$ to complete the proof of claim \eqref{eqn:claim_binomial_tail_pure}.}\\

	\textit{Proof of Lemma  \ref{lemma:binomial_master}  Part \ref{lemma:binomial_tail_contam}:}
	Recall the proof of Part (a, ii) and Fix a sequence $\delta_n^*$ satisfying $(i)$ and $(ii)$ of Lemma \ref{lemma:delta_choice}. Then
	\be 
	\ & \P\left(d'(\beta_1,\beta_2)>C_n\sigmabeta\sqrt{\log{n}}\right)\\ 
	&\leq \P(\beta_1X'+\beta_2Y'>t_n(\delta_n^*))+\P(d(\beta_1,\beta_2)>C_n\sigmabeta\sqrt{\log{n}}-t_n(\delta_n^*))
	\ee
	where $t_n(\delta_n^*)=\delta_n^*\sigmabeta\sqrt{\log{n}^*}$. Now by our choice of $\delta_n^*$ we have $C_n\sigmabeta\sqrt{\log{n}}-t_n(\delta_n^*)=C_n(1+o(1))\sigmabeta\sqrt{\log{n}}$. Moreover, similar to the proof of Part (a,ii) , by Bernstein's Inequality for $\theta>0$ one has for $n$ large enough (depending on $\theta,\beta_1,\beta_2,\tau_a,\tau_b$)  
	\be 
	\P\left(\beta_1 X'+\beta_2Y'>t_n(\delta_n^*)\right)&\leq n^{-\theta}.
	\ee
	Therefore by Part (b, ii) we have by choosing $\theta>2C$
	\be 
	\P\left(d'(\beta_1,\beta_2)>C_n\sigmabeta\sqrt{\log{n}}\right)&\leq n^{-\frac{C^2}{2}+o(1)}.
	\ee
	The lower bound is trivial from Part (b, ii) since $d'(\beta_1,\beta_2)\geq d(\beta_1,\beta_2)$.
	\qed

	\subsection{Proof of Lemma \ref{lemma:binomial_tail_exp_scale}}
	Recall the definitions $\mu_{n1} = \E[X]$, $\mu_{n2} = \E[Y]$, $\sigma_{n1}^2 = \Var(X)$, $\sigma_{n2}^2 = \Var(Y)$ and $\mubeta=\E(\beta_1X+\beta_2Y)$ and $\sigmabeta=\var(\beta_1X+\beta_2Y)$. For brevity, we let $\mu_n=\mu_n(1,1)$ and $\sigma_n=\sigma_n(1,1)$. Let $\mathscr{H} = \{ h>0: \mu_n + h  \in \mathbb{N} \}$. Thus $\P[X+Y > \mu_n + \sigma_n x_n] = \sum_{ h \in \mathscr{H}: h > \sigma_n x_n} \P[X+ Y = \mu_n + h]$. Let $\mathscr{H}_1 = \{ h : \mu_{n1} + h \in \mathbb{N} \}$ and set $h_1^* = \inf\{ h \in \mathscr{H}_1 : h > \frac{\sigma_{n1}^2 }{\sigma_n} x_n\}$. Thus we have, for $h \in \mathscr{H}$, 
	\begin{align}
	\P[X+Y = \mu_n + h] &= \sum_{h_1 \in \mathscr{H}_1} \P[X = \mu_{n1} + h_1 ] \P[ Y = \mu_{n2} + (h-h_1)] \nonumber \\
	&\geq \sum_{h_1 = h_1^*}^{h_1^* +  m^* } \P[X = \mu_{n1} + h_1 ] \P[ Y = \mu_{n2} + (h-h_1)], \label{eq:int1}
	\end{align}
	for some $m^*$ to be chosen appropriately. 
	Using \cite[Theorem 1.5]{bollobas}, we have, 
	\begin{align}
	\P[X = \mu_{n1} + h_1 ] \geq \frac{1}{\sqrt{2\pi} \sigma_{n1}} \exp\Big[ - \frac{h_1^2}{ 2 \sigma_{n1}^2 } - \xi_{1n} \Big],  \nonumber 
	\end{align}
	for an explicit sequence $\xi_{1n}(h_1)$, depending on $h_1$. Upon using the fact that $a', b' \gg (\log n )^3$, $h_1 = O( \sigma_n \sqrt{2 \log n})$, for any $m^* \ll h_1^*$, it is immediate that $\xi_{1n}(h_1) = o(1)$, uniformly over $ h_1^* < h_1 < h_1^* + m^*$. Thus we have, 
	\begin{align}
	\P[X = \mu_{n1} + h_1 ] \geq (1 + o(1)) \frac{1}{\sqrt{2\pi} \sigma_{n1}} \exp\Big[ - \frac{h_1^2}{ 2 \sigma_{n1}^2 } \Big] . \nonumber
	\end{align}
	Similar arguments immediately imply that for $h_1^* < h_1 < h_1^* + m^*$,
	\begin{align}
	\P[ Y = \mu_{n2} + (h-h_1) ] \geq (1 + o(1)) \frac{1}{\sqrt{2 \pi} \sigma_{n2}} \exp\Big[ - \frac{(h- h_1)^2}{2 \sigma_{n2}^2} \Big]. \nonumber
	\end{align}
	Using these bounds in \eqref{eq:int1}, for $h = O(\sigma_n \sqrt{2 \log n})$, we obtain the lower bound 
	\begin{align}
	\P[ X+ Y = \mu_n + h] &\geq (1 + o(1)) \frac{m^*}{2\pi \sigma_{n1} \sigma_{n2}} \exp\Big[ - \frac{h^2 }{2\sigma_{n}^2} \Big] \nonumber\\
	&= (1 + o(1)) \frac{1}{\sqrt{2 \pi} \sigma_n }  \exp\Big[ - \frac{h^2 }{2\sigma_{n}^2} \Big] , \nonumber
	\end{align} 
	where we choose $m^* = \sqrt{2\pi} \sigma_{n1} \sigma_{n2} / \sigma_{n} \ll h_1^*$. Finally, we have, setting $M_n = \sigma_n \sqrt{C\log n}$ for some constant $C$ sufficiently large, 
	\begin{align}
	&\P[X+ Y > \mu_n + \sigma_n x_n ] \geq \sum_{h \in \mathscr{H} : \sigma_n x_n < h < M_n } \P[X+ Y = \mu_n + h] \nonumber\\
	&\geq (1 + o (1)) \frac{1}{\sqrt{2\pi}\sigma_n} \sum_{h \in \mathscr{H} : \sigma_n x_n < h < M_n } \exp\Big[ - \frac{h^2}{2 \sigma_n^2} \Big] \geq (1 + o(1)) ( 1 - \Phi(x_n)), \nonumber 
	\end{align}
	if $C$ is chosen sufficiently large. This establishes the required lower bound.
	
	Next, we turn to the upper bound. We have,
	\begin{align}
	&\P[X+Y > \mu_n + \sigma_n x_n] = \sum_{h \in \mathscr{H}: h > \sigma_n x_n } \P[X+Y = \mu_n + h]\nonumber\\
	&= \sum_{h \in \mathscr{H} : \sigma_n x_n < h < \sigma_n \sqrt{C \log n}} \P[X+Y = \mu_n +h] + \sum_{h \in \mathscr{H}: h > \sigma_n \sqrt{C \log n}} \P[X+Y = \mu_n + h], \label{eq:tail_upper_exact}
	\end{align}
	for some constant $C>0$ sufficiently large, to be chosen later. We have, using Lemma \ref{lemma:binomial_master} Part \ref{lemma:binomial_tail}, we have,
	\begin{align}
	\sum_{h \in \mathscr{H}: h > \sigma_n \sqrt{C \log n}} \P[X+Y = \mu_n + h] \leq n^{- \frac{C^2}{2} (1 + o(1)) }. \label{eq:temp_veryhightail} 
	\end{align}
	Finally, we will use the following ``local limit" lemma. 
	\begin{lemma}
		\label{lemma:eq_upper}
		Let $X \sim \textrm{Bin}\Big(\gamma_1n/k , \frac{a'}{n} \Big)$ and $Y \sim \textrm{Bin}\Big(\gamma_2n/k, \frac{b'}{n} \Big)$ be independent random variables with $a' \geq b'$, $\liminf b'/a' >0$ and $\gamma_1,\gamma_2,k$ given numbers such that $\gamma_1n/k,\gamma_2n/k\in \mathbb{N}$. Assume $b' \gg (\log n)^3$ and set $\mu_n = \E[X+Y]$, $\sigma_n^2 = \Var(X+Y)$. Then for any constant $C>2$ and $ \sigma_n \sqrt{2 \log n} < h < \sigma_n \sqrt{C \log n}$, we have, for $h \in \mathcal{H}$,
		\begin{align}
		\P[X+Y = \mu_n + h ] \leq  (1+ o(1))\frac{1}{\sqrt{2 \pi} \sigma_n} \exp\Big[ - \frac{h^2}{2 \sigma_n^2} \Big]. \nonumber 
		\end{align}
	\end{lemma}
	
	We defer the proof of Lemma \ref{lemma:eq_upper} and complete upper bound proof. Lemma \ref{lemma:eq_upper} immediately yields
	\begin{align}
	\sum_{h \in \mathscr{H} : \sigma_n x_n < h < \sigma_n \sqrt{C \log n}} \P[X+Y = \mu_n +h] &\leq (1 + o(1)) \sum_{h \in \mathscr{H} : \sigma_n x_n < h < \sigma_n \sqrt{C \log n}} \frac{1}{\sqrt{2\pi} \sigma_n} \exp\Big[ - \frac{h^2}{2 \sigma_n^2} \Big]. \nonumber\\
	&\leq (1+o(1)) \int_{x_n}^{\sqrt{C \log n}} \phi(x) \textrm{d}x, \nonumber 
	\end{align}
	where $\phi(\cdot)$ is the density of the standard Gaussian distribution. We know that $(1 - \Phi(\sqrt{C\log n})) \leq n^{-\frac{C^2}{2}}$. 
	Thus for $C$ sufficiently large, $\frac{\Phi(\sqrt{C \log n} ) - \Phi(x_n)}{1 - \Phi(x_n)} \to 1$ as $n \to \infty$. For any such choice of $C$, we immediately have, using \eqref{eq:tail_upper_exact} and \eqref{eq:temp_veryhightail}, 
	\begin{align}
	\P[X+Y > \mu_n + \sigma_n x_n] \leq  (1 + o(1)) (1 - \Phi(x_n)) + n^{-\frac{C^2}{2} (1 + o(1))} = (1 + o(1)) (1 - \Phi(x_n)). \nonumber 
	\end{align}
	This completes the proof modulo proof of Lemma \ref{lemma:eq_upper}. \qed

	\begin{proof}[Proof of Lemma \ref{lemma:eq_upper}]
		We have, for $h \in \mathscr{H}$, setting $h_1^* = h \sigma_{n1}^2/ \sigma_{n}^2$, and $m^* = \sqrt{2\pi}\frac{{\sigma_{n1} \sigma_{n2}}}{2\sigma_n}$, 
		\begin{align} 
		&\P[X+ Y = \mu_n + h ] = \sum_{h_1 \in \mathscr{H}_1} \P[X= \mu_{n1} + h_1 , Y = \mu_{n2} + h- h_1]  =T_1 + T_2 + T_3, \nonumber \\
		&T_1 = \sum_{h_1 \in \mathscr{H}_1: h_1 < h_1^* - m^* } \P[X= \mu_{n1} + h_1 , Y = \mu_{n2} + h- h_1]  , \nonumber \\
		&T_2 = \sum_{h_1 \in \mathscr{H}_1: h_1^* - m^*  < h_1 < h_1^* + m^* } \P[X= \mu_{n1} + h_1 , Y = \mu_{n2} + h- h_1] , \nonumber\\
		&T_3 = \sum_{h_1 \in \mathscr{H}_1: h_1 > h_1^* + m^* } \P[X= \mu_{n1} + h_1 , Y = \mu_{n2} + h- h_1] , \nonumber 
		\end{align}
		First, we analyze the term $T_2$. \cite[Theorem 1.2]{bollobas} implies that for $h_1 = O( \sigma_{n1} \sqrt{\log n} )$, 
		\begin{align}
		\P[ X = \mu_{n1} + h_1 ] < \frac{1}{\sqrt{2\pi} \sigma_{n1}} \exp\Big[- \frac{h_1^2}{2 \sigma_{n1}^2} + \xi_n(1) \Big] , \nonumber 
		\end{align}
		for an explicit sequence $\xi_n(1)$. Using $a'\geq b' \gg (\log n)^3$, and $h = O(\sigma_{n} \sqrt{ \log n})$ it immediately follows that $\xi_1(n) = o(1)$. Using similar arguments for $\P[ Y = \mu_{2 n} + h - h_1]$ , we obtain that 
		\begin{align}
		T_2 &\leq (1 + o(1)) \sum_{h_1 \in \mathscr{H}_1: h_1^* - m^*  < h_1 < h_1^* + m^* } \frac{1}{2 \pi \sigma_{n1} \sigma_{n2} } \exp\Big[ - \frac{h_1^2}{2 \sigma_{n1}^2} - \frac{(h- h_1)^2}{2 \sigma_{n2}^2} \Big] \nonumber \\
		&\leq (1 + o(1)) \frac{1}{\sqrt{2 \pi} \sigma_{n}}\exp\Big[ - \frac{h^2}{2 \sigma_n^2} \Big] =: (1 + o(1)) z_n, \nonumber 
		\end{align}
		using the definition of $h_1^*$. We will be done once we establish $T_1, T_3 = o(z_n)$. We will sketch this proof for $T_3$--- the argument for $T_1$ is analogous and will be omitted. We note that 
		\begin{align}
		T_3 &= \sum_{h \in \mathscr{H}_1: h_1^* + m^* < h_1 < x_n \sigma_{n1} (1 + \tau_n)} \P[ X = \mu_1 + h_1, Y = \mu_2 + h-h_1] \nonumber\\
		&+ \sum_{h \in \mathscr{H}_1: h_1 > x_n \sigma_{n1}(1 + \tau_n) } \P[X = \mu_{n1} + h_1, Y = \mu_{n2} + (h - h_1)], \label{eq:int2}
		\end{align}
		for some sequence $\tau_n>0$ to be chosen appropriately. We will establish that each of these terms is $o(T_2)$. To this end, we note that 
		\begin{align}
		\sum_{h \in \mathscr{H}_1: h_1 > x_n \sigma_{n1} (1 + \tau_n)} \P[X = \mu_{n1} + h_1, Y = \mu_{n2} + (h - h_1)] \leq \P[X > \mu_{n1} + \sigma_{n1} x_n ( 1  + \tau_n) ] \sup_x \P[Y = x]. \nonumber
		\end{align}
		By direct computation, it is easy to see that $\sup_x \P[Y= x] = O(\frac{1}{\sqrt{b'}})$. Using the results in \cite[Lemma 6.2]{mms2016}, we have, 
		\begin{align}
		\P[X> \mu_{n1} + \sigma_{n1} x_n (1 + \tau_n) ] \leq n^{- (1 + \tau_n)^2 (1 + o(1))}. \nonumber 
		\end{align}
		Thus for any sequence $\tau_n >0 $ such that $\liminf \tau_n >0$, the second term in \eqref{eq:int2} is $o(z_n)$. Next, we study the first term in the RHS of \eqref{eq:int2}. We note that $\sigma_{n} \geq \sigma_{n1}$, $h< \sigma_n \sqrt{C\log n}$ implies $h - x_n \sigma_{n1} (1 + \tau_n) \geq x_n \sigma_{n} \Big( 1 - \frac{\sigma_{n1}}{\sigma_n}  ( 1 + \tau_n) \Big)$ and $h- x_n \sigma_{n1} (1+ \tau_n) \leq \sqrt{\log n} \sigma_n \Big(\sqrt{C} - \sqrt{2} \frac{\sigma_{n1}}{\sigma_n} (1 + \tau_n) \Big)$. This implies that $h - x_n \sigma_{n1} (1 + \tau_n) = O( \sqrt{\log n} \sigma_n )$ for some sequence $\tau_n$ sufficiently small. We will fix any such sequence in the rest of the proof. For any such $  h_1^* + m^* < h_1 < x_n \sigma_n (1 + \tau_n)$, \cite[Theorem 1.2]{bollobas} implies that 
		\begin{align}
		\P[ X = \mu_{n1} + h_1] \leq \frac{1}{\sqrt{2 \pi} \sigma_{n1}} \exp\Big[ - \frac{h_1^2}{2 \sigma_{n1}^2 } + \xi_n(1, h_1) \Big], \nonumber 
		\end{align}
		for some explicit sequence $\xi_n(1, h_1)$, depending on $h_1$. Further, $a',b' \gg (\log n )^3$ and $ h_1^* + m^* < h_1 < x_n \sigma_{n1} (1 + \tau_n)$ implies that 
		\begin{align}
		\P[ X = \mu_{n1} + h_1 ] \leq (1 + o(1)) \frac{1} { \sqrt{2 \pi} \sigma_{n1}} \exp\Big[ - \frac{h_1^2}{2 \sigma_{n1}^2}  \Big]. \nonumber
		\end{align}
		where $o(1)$ is a term uniformly controlled for all $ h_1^* + m^* < h_1 < x_n \sigma_n (1 + \tau_n)$.
		Exactly analogous considerations imply that 
		\begin{align}
		\P[Y = \mu_{n2} + h - h_1] \leq (1 + o(1)) \frac{1}{ \sqrt{2 \pi} \sigma_{n2}} \exp\Big[ - \frac{ (h- h_1)^2}{2 \sigma_{n2}^2} \Big]. \nonumber 
		\end{align}
		Thus we have, 
		\begin{align}
		&\sum_{h \in \mathscr{H}_1: h_1^* + m^* < h_1 < x_n \sigma_{n1} (1 + \tau_n)} \P[ X = \mu_1 + h_1, Y = \mu_2 + h-h_1] \nonumber \\
		&\leq (1 + o(1))  \sum_{h \in \mathscr{H}_1 : h_1^* + m^* < h_1 < x_n \sigma_{n1}( 1 + \tau_n)} \frac{1}{2 \pi \sigma_{n1} \sigma_{n2}} \exp\Big[ - \frac{h_1^2}{2 \sigma_{n1}^2} - \frac{(h- h_1)^2}{2 \sigma_{n2}^2} \Big]. \nonumber \\
		&\leq (1 + o(1)) \frac{1}{\sqrt{2\pi} \sigma_n } \exp\Big[ - \frac{h^2 \sigma_{n1}^2}{2 \sigma_n^2} \Big] \int_{h_1^* + m^*}^{x_n \sigma_{n1} (1+\tau_n)}  \frac{1}{\sqrt{2\pi} \sigma_0} \exp\Big[  -\frac{(x- \frac{h \sigma_0^2 }{\sigma_{n2}^2})^2}{2 \sigma_0^2} \Big]\textrm{d} x \nonumber \\
		&\leq (1 + o(1)) \frac{1}{\sqrt{2 \pi}\sigma_n }   \exp\Big[ - \frac{h^2 \sigma_{n1}^2}{2 \sigma_n^2} \Big] = o(z_n), \nonumber 
		\end{align}
		where we set $\sigma_0^2 = \frac{\sigma_{n1}^2 \sigma_{n2}^2}{\sigma_n^2}$.  This completes the proof. 
	\end{proof}

\end{document}